\documentclass{article}
\usepackage[T1]{fontenc}
\usepackage{amsthm, fullpage}
\usepackage{amsmath}
\usepackage{amssymb}
\usepackage{amsfonts}
\usepackage{eucal}
\usepackage[all]{xy}
\usepackage[frenchb, english]{babel}
\usepackage{pslatex}
\usepackage{hyperref}
\providecommand{\scr}{\mathcal}
\usepackage{mathrsfs}

\newtheorem{prop}{Proposition}[subsection]
\newtheorem{theo}[prop]{Théor\`eme}
\newtheorem*{theo**}{Théorème}
\newtheorem{coro}[prop]{Corollaire}
\newtheorem*{conj*}{Conjecture}
\newtheorem{lemm}[prop]{Lemme}
\newtheorem{lemm*}{Lemme}[prop]

\theoremstyle{definition}

\newtheorem{vide}[prop]{}
\newtheorem{defi}[prop]{Définition}
\newtheorem*{defi*}{Définition}
\newtheorem{nota}[prop]{Notations}

\theoremstyle{remark}
\newtheorem{rema}[prop]{Remarques}
\newtheorem{exem}[prop]{Exemples}

\numberwithin{equation}{prop}

\newcommand{\riso}{ \overset{\sim}{\longrightarrow}\, }
\newcommand{\liso}{ \overset{\sim}{\longleftarrow}\, }

\newcommand{\Spf}{\mathrm{Spf}\,}

\newcommand{\FF}{{\mathcal{F}}}

\newcommand{\E}{{\mathcal{E}}}
\newcommand{\G}{{\mathcal{G}}}

\newcommand{\M}{{\mathcal{M}}}

\newcommand{\D}{{\mathcal{D}}}
\newcommand{\I}{{\mathcal{I}}}
\newcommand{\PP}{{\mathcal{P}}}

\renewcommand{\O}{{\mathcal{O}}}

\newcommand{\V}{\mathcal{V}}
\renewcommand{\S}{\mathcal{S}}

\newcommand{\Y}{\mathcal{Y}}
\newcommand{\ZZ}{\mathcal{Z}}
\newcommand{\X}{\mathfrak{X}}

\newcommand{\U}{\mathfrak{U}}

\newcommand{\A}{\mathbb{A}}

\newcommand{\DD}{\mathbb{D}}
\renewcommand{\L}{\mathbb{L}}
\newcommand{\R}{\mathbb{R}}
\newcommand{\Q}{\mathbb{Q}}
\newcommand{\Z}{\mathbb{Z}}
\newcommand{\N}{\mathbb{N}}

\newcommand{\hdag}{  \phantom{}{^{\dag} }    }

\begin{document}

\selectlanguage{frenchb}

\title{La surcohérence entraîne l'holonomie}
%Overcoherence implies holonomicity}
\author{Daniel Caro}

\date{}

\maketitle

\selectlanguage{english}
\begin{abstract}
Let $\V$ be a mixed characteristic complete discrete valuation ring with perfect residue field.
Let $\X$ be a smooth formal scheme over $\V$ and $D$ a divisor of its special fiber.
We define the notion of $\D ^\dag _{\X}(\hdag D) _{\Q}$-overcoherence in $\X$ (after any change of basis), 
which is a priori a weaker notion than the $\D ^\dag _{\X}(\hdag D) _{\Q}$-overcoherence. 
We prove that a $\D ^\dag _{\X}(\hdag D) _{\Q}$-overcoherent after any change of basis module is
$\D ^\dag _{\X}(\hdag D) _{\Q}$-holonomic.
Furthermore, we check that this implies the following property of stability of the overholonomicity:  a bounded complex $\E$ of 
$\D ^\dag _{\X,\,\Q}$-modules is overholonomic after any change of basis if and only if, for any integer $j$,  
$\mathcal{H} ^{j} (\E) $ is overholonomic after any change of basis.
\end{abstract}

\selectlanguage{frenchb}

\begin{abstract}
Soit $\V$ un anneau de valuation discrète complet d'inégales caractéristiques, de corps résiduel parfait. 
Soit $\X$ un schéma formel lisse sur $\V$.
Nous définissons la notion de $\D ^\dag _{\X}(\hdag D) _{\Q}$-surcohérence dans $\X$ (après tout changement de base), 
ce qui correspond a priori à une notion plus faible que celle de $\D ^\dag _{\X}(\hdag D) _{\Q}$-surcohérence. 
Nous établissons qu'un module $\D ^\dag _{\X}(\hdag D) _{\Q}$-surcohérent après tout changement de base est $\D ^\dag _{\X}(\hdag D) _{\Q}$-holonome.
De plus, nous en déduisons la propriété suivante de stabilité de la surholonomie: un complexe borné de 
$\D ^\dag _{\X,\,\Q}$-modules $\E$ est surholonome après tout changement de base si et seulement si, 
pour tout entier $j$,  
$\mathcal{H} ^{j} (\G) $ est surholonome après tout changement de base.
\end{abstract}

\tableofcontents

\section*{Introduction}

Ce travail s'inscrit dans la théorie des $\D$-modules arithmétiques de P. Berthelot.
Cette théorie constitue une version arithmétique de celle des modules sur le faisceau des opérateurs différentiels 
(pour une introduction générale consulter
\cite{Beintro2}, autrement lire dans l'ordre \cite{Be0}, \cite{Be1}, \cite{Be2}).
Comme l'avait conjectuée Berthelot,
cette théorie, via la notion de surholonomie (voir \cite{caro_surholonome}),
permet d'obtenir dans un travail en commun avec Tsuzuki (voir \cite{caro-Tsuzuki})
une cohomologie $p$-adique sur les variétés algébriques sur un corps de caractéristique $p>0$
stable par les six opérations cohomologiques de Grothendieck, i.e., image directe, image directe extraordinaire,
image inverse, image inverse extraordinaire, produit tensoriel, foncteur dual. 
Voici le contexte de cette théorie :
soit $\V$ un anneau de valuation discrète complet d'inégales caractéristiques $(0,p)$,
de corps résiduel supposé parfait.
Soit $\X$ un $\V$-schéma formel intègre et lisse, $X$ sa fibre spéciale.
Dans la version arithmétique de Berthelot, le faisceau des opérateurs différentiels usuels $\D$  est remplacé par $\D ^{\dag} _{\Q}$.
Plus précisément, il construit le faisceau sur $\X$ des opérateurs différentiels de niveau fini et d'ordre infini
noté $\D ^\dag _{\X,\Q}$ ; ce dernier correspondant à la tensorisation par $\Q$ (indiqué par l'indice $\Q$) du complété faible $p$-adique (indiqué par le symbole {\og $\dag$\fg})
du faisceau classique $\D  _{\X}$ des opérateurs différentiels sur $\X$ (i.e. au sens de Grothendieck \cite[16]{EGAIV4}).
Lorsque l'on parlera de structure de Frobenius (seulement dans ce cas-là), 
on suppose qu'il existe un isomorphism $\sigma \colon \V \riso \V$ relevant l'endomorphisme de Frobenius de $k$.
Berthelot a aussi obtenu la notion de $F\text{-}\D ^\dag _{\X,\Q} $-module {\og holonome \fg}  
en s'inspirant du cas classique : un $F\text{-}\D ^\dag _{\X,\Q} $-module cohérent (i.e un $\D ^\dag _{\X,\Q} $-module cohérent avec une structure de Frobenius) est holonome
s'il est nul ou si la dimension de sa variété caractéristique est égale à la dimension de $X$.
On dispose de plus de la caractérisation de Virrion de l'holonomie (voir \cite{virrion}): 
un $F\text{-}\D ^\dag _{\X,\Q} $-module cohérent $\E$ est holonome si et seulement si 
le foncteur dual $\D ^\dag _{\X,\Q}$-linéaire est acyclique pour $\E$. 
Plus généralement, soit $D$ un diviseur de $X$.
Le critère de Virrion permet d'étendre la notion d'holonomie de la manière suivante:
un $\D ^\dag _{\X}(\hdag D) _{\Q} $-module cohérent $\E$ est $\D ^\dag _{\X}(\hdag D) _{\Q} $-holonome 
si le foncteur dual $\D ^\dag _{\X}(\hdag D) _{\Q} $-linéaire est acyclique pour $\E$. 
La conjecture la plus forte (car elle implique toutes les autres) 
de Berthelot sur la stabilité de l'holonomie (voir \cite[5.3.6]{Beintro2}) 
prédit qu'un $F\text{-}\D ^\dag _{\X}(\hdag D) _{\Q} $-module holonome 
est un $F\text{-}\D ^\dag _{\X,\Q} $-module holonome. 
Lorsque $\X$ est projectif,
cette conjecture a été établie (voir \cite{caro-stab-holo}).
Par contre, un $\D ^\dag _{\X}(\hdag D) _{\Q} $-module holonome n'est pas toujours
un $\D ^\dag _{\X,\Q} $-module holonome, i.e., il faut faire plus attention sans structure de Frobenius. 
En effet, les $\D ^\dag _{\X}(\hdag D) _{\Q} $-modules cohérents  $\O _{\X}(\hdag D) _{\Q} $-cohérent (i.e. les isocristaux surconvergents sur $( X \setminus D, X)$
dans le langage de la cohomologie rigide) sont des $\D ^\dag _{\X}(\hdag D) _{\Q} $-modules holonomes mais ils ne sont
pas toujours $\D ^\dag _{\X,\Q} $-module cohérent (et donc encore moins $\D ^\dag _{\X,\Q} $-holonome).
Pour espérer généraliser cette conjecture de Berthelot sans structure de Frobenius, 
il faut rajouter des conditions supplémentaires
de type non-Liouville (voir \cite{caro-stab-holo-courbe} pour un tout premier travail dans cette direction).

Soit $\E \in (F$-$)D _\mathrm{coh} ^\mathrm{b} (\D ^\dag _{\X} (\hdag D) _{\Q})$.
Dans ce papier nous introduisons la notion de {\og $\D ^\dag _{\X} (\hdag D) _{\Q}$-surcohérence dans $\X$\fg} :
le complexe $\E$ est $\D ^\dag _{\X} (\hdag D) _{\Q}$-surcohérent dans $\X$ si,
    pour tout diviseur $T $ de $X$, on ait 
    $(\hdag T )  (\E) \in (F$-$)D _\mathrm{coh} ^\mathrm{b} (\D ^\dag _{\X} (\hdag D) _{\Q})$, où 
    $(\hdag T ) $ est le foncteur localisation en dehors de $T$. 
    Cette notion rejoint celle de la $\D ^\dag _{\X} (\hdag D) _{\Q}$-surcohérence définie dans 
    \cite{caro_surcoherent}. En effet, 
    $\E$ est $\D ^\dag _{\X} (\hdag D) _{\Q}$-surcohérent si et seulement si pour tout morphisme lisse de la forme
    $f \,:\, \X ' \to \X$, le complexe $f ^! (\E) $ est $\D ^\dag _{\X'} (\hdag f ^{-1}(D)) _{\Q}$-surcohérent dans $\X'$.
    Nous prouvons dans ce papier l'implication $(*)$ : si 
    $\E$ est $\D ^\dag _{\X} (\hdag D) _{\Q}$-surcohérent dans $\X$ (et après tout changement de la base $\V$)
    alors $\E$ est $\D ^\dag _{\X} (\hdag D) _{\Q}$-holonome. 
    En particulier, si $\E$ est $\D ^\dag _{\X} (\hdag D) _{\Q}$-surcohérent (et après tout changement de la base $\V$)
    alors $\E$ est $\D ^\dag _{\X} (\hdag D) _{\Q}$-holonome.

Cette implication $(*)$ améliore le théorème plus faible déjà connu établissant qu'un 
$F$-$\D ^\dag _{\X} (\hdag D) _{\Q}$-module surcohérent est
surholonome et donc holonome 
(voir \cite[2.3.17]{caro-Tsuzuki} ou \cite[6.2.4]{caro-pleine-fidelite} pour la version la plus générale possible). 
Lorsque $\X$ est de plus projectif, il découle de \cite{caro-stab-holo}
que les notions de $F$-$\D ^\dag _{\X}(\hdag D) _{\Q} $-holonomie 
et de $F$-$\D ^\dag _{\X}(\hdag D) _{\Q} $-surcohérence dans $\X$ sont équivalentes, i.e. la réciproque est valable au moins dans
cette situation géométrique et avec une structure de Frobenius. 

Même dans le cas de la surcohérence (notion plus restrictive à priori que la surcohérence dans $\X$)
avec structure de Frobenius évoquée ci-dessus,
la preuve de ce papier est aussi intéressante 
puisqu'elle est directe et n'utilise pas le puissant théorème de la réduction semi-stable de Kedlaya.
En effet, la preuve avec structure de Frobenius de l'équivalence entre surcohérence et surholonomie 
est le résultat d'un long processus dont voici une esquisse : 
Afin d'obtenir de surcroît des coefficients stables par dualité,  
la notion de surcohérence avait été raffinée via la notion de {\og surholonomie \fg}.
Nous avions établi comme son nom l'indique qu'un module surholonome est holonome (voir \cite{caro_surholonome}). 
Après avoir fait le lien avec la cohomologie rigide de Berthelot (voir \cite{Berig}, \cite{LeStum-livreRigCoh}), lien qui a été initié par Berthelot dans \cite{Be0,Be1,Be2}
puis prolongé dans l'ordre \cite{caro-construction}, \cite{caro_devissge_surcoh,caro-2006-surcoh-surcv} (avec structure de Frobenius),
\cite{caro-pleine-fidelite} (sans structure de Frobenius), 
la notion de {\og dévissabilité en isocristaux surconvergents \fg} a été développée. 
Lorsque que l'on dispose d'une structure de Frobenius, 
d'après \cite[5.2.4]{caro-pleine-fidelite} (qui étend très légèrement \cite{caro-Tsuzuki}),
on démontre en fait que ces trois notions de surholonomie, de surcohérence et de dévissabilité en $F$-isocristaux surconvergents sont toujours identiques.
Pour prouver cette équivalence, on remarque d'abord par définition que la surholonomie entraîne la surcohérence. 
Puis, on a établi que la surcohérence implique la dévissabilité en $F$-isocristaux surconvergents (voir \cite{caro-2006-surcoh-surcv}). 
Enfin, pour boucler la preuve, il reste par dévissage à prouver la surholonomie
des $F$-isocristaux surconvergents, ce qui est établi dans \cite{caro-Tsuzuki}. 
Un ingrédient fondamental de cette preuve est le théorème de la réduction semi-stable de Kedlaya (voir \cite{kedlaya-semistableI,kedlaya-semistableII,kedlaya-semistableIII,kedlaya-semistableIV}), théorème qui utilise la structure de Frobenius, qui est faux en général et qui prolonge de manière très remarquable le théorème de la monodromie $p$-adique établi séparément par André, Kedlaya, Mebkhout (voir \cite{andreHA,Kedlaya-localmonodromy,mebkhout-monodr-padique}) ou du théorème analogue de Tsuzuki dans le cas des $F$-isocristaux unités (voir \cite{Tsuzuki-mono-unitroot}). 

On termine ce travail par une application sur la stabilité de la surholonomie: un complexe de $\D ^\dag _{\X,\Q} $-modules est surholonome après tout changement de base si et seulement si ses espaces de cohomologie le sont (voir \ref{coro2-stab-surcoh-dansX}). 

\medskip 

Décrivons maintenant l'organisation  de ce papier. 
Il s'articule en trois parties. Les deux premiers chapitres établissent des résultats préliminaires qui permettront d'obtenir les deux résultats principaux de ce travail donnés dans le dernier. Précisons à présent leur contenu.  

Dans le premier chapitre, dans le cas d'une immersion fermée de schémas affines, lisses sur une base noethérienne
et munis de coordonnées locales,
nous calculons explicitement les morphismes d'adjonction entre l'image directe et l'image inverse extraordinaire dans les catégories
de $\D$-modules. Nous vérifions ensuite que l'image directe par une immersion fermée commute à l'image inverse extraordinaire
par une immersion fermée. 

Dans le cas d'une immersion fermée de $\V$-schémas formels affines, lisses et munis de coordonnées locales, 
nous étendons dans le deuxième chapitre le calcul des morphismes d'adjonction par passage à la limite projective de la situation du premier chapitre. 
Nous y expliquons pourquoi, afin d'obtenir des catégories stables par image directe et image inverse extraordinaire d'une immersion fermée, il s'agit de travailler cette fois-ci exclusivement avec des modules sur les sections globales de faisceaux d'opérateurs différentiels de niveau fixé, 
et non plus avec des modules sur des faisceaux d'opérateurs différentiels de niveau fixé. 

Enfin, dans la dernière partie, après quelques compléments sur la notion d'holonomie sans structure de Frobenius développée dans \cite{caro-holo-sansFrob}, grâce aux isomorphismes de changement de base du premier chapitre et aux morphismes d'adjonction du deuxième chapitre, nous établissons le résultat principal de ce papier : un $\D ^\dag _{\X} (\hdag D) _{\Q}$-module surcohérent dans $\X$ après tout changement de base
est $\D ^\dag _{\X} (\hdag D) _{\Q}$-holonome. On vérifie ensuite que cela implique que la notion de {\og surholonomie après tout changement de base\fg} est stable pour tout entier $j$ par le foncteur $j$-ième espace de cohomologie, i.e. par $\mathcal{H} ^{j}$.

\section*{Notations}
Soient $\V$ un anneau de valuations discrètes complet d'inégales caractéristiques $(0,p)$, 
de corps résiduels $k$ supposé parfait, d'uniformisant $\pi$, de corps des fractions $K$. 
Lorsque l'on parlera de structure de Frobenius (seulement dans ce cas-là), 
on suppose qu'il existe un isomorphism $\sigma \colon \V \riso \V$ relevant l'endomorphisme de Frobenius de $k$.

Les faisceaux seront notés avec des lettres calligraphiques, 
les lettres droites correspondantes signifiant alors leur section globale. 
Le symbole chapeau désigne toujours la complétion $p$-adique. 
Si $f \,:\, \ZZ \to \X$ est un morphisme de $\V$-schémas formels lisses, 
on notera 
pour tout $i \in \N$ par 
$f _i \,:\, Z _i \to X _i $ le morphisme de schémas induit par réduction modulo $\pi ^{i+1}$
(et de même pour d'autres lettres de l'alphabet). On notera simplement $X $ ou $Z$ à la place de $X _0$ ou $Z _0$.

Nous reprendrons les notations standards concernant les faisceaux d'opérateurs différentiels apparaissant dans la théorie des $\D$-modules arithmétiques de Berthelot 
(e.g. voir \cite{Be1} pour les constructions détaillées de ces faisceaux et \cite{Beintro2} pour les constructions des opérations cohomologiques telles que l'image directe ou l'image inverse extraordinaire).
\medskip

\section*{Remerciement}

Je remercie Tomoyuki Abe pour une erreur décelée dans une version précédente de ce papier.

\section{Immersion fermée de schémas lisses et $\D$-modules quasi-cohérents}

Dans ce chapitre, on désigne par $m$ un entier positif, $S$ un schéma noethérien. Sauf mention explicite du contraire, 
nous travaillerons dans la catégorie des $S$-schémas (lisses) et 
nous omettrons alors d'indiquer $S$ dans les notations concernant les opérateurs différentiels.

\subsection{Préliminaires: quelques calculs en coordonnées locales}

\label{u+-sch}

Soit $u\,:\, Z \hookrightarrow X$ une immersion fermée de $S$-schémas lisses. 
On note $\omega _{X}$ le faisceau des formes différentielles sur $X$ de degré maximal, de même pour $Z$.
On dispose du $(\D ^{(m)} _{Z}, u ^{-1} \D ^{(m)} _{X})$-bimodule
$\D ^{(m)} _{Z \hookrightarrow X}:= u ^{*} \D ^{(m)} _X$
et du 
$(u ^{-1} \D ^{(m)} _{X}, \D ^{(m)} _{Z})$-bimodule
$\D ^{(m)} _{X \hookleftarrow Z}:= u ^{*} _{\mathrm{d}} (\D ^{(m)} _X \otimes _{\O _X} \omega _X ^{-1}) \otimes _{\O _Z} \omega _{Z}$, 
l'indice $d$ signifiant que 
l'on choisit la structure droite du $\D ^{(m)} _X$-bimodule à droite $\D ^{(m)} _X \otimes _{\O _X} \omega _X ^{-1}$.
D'après \cite[3.4.1]{Be2}, 
via l'isomorphisme de transposition échangeant les structures droite et gauche de $\D ^{(m)} _X \otimes _{\O _X} \omega _X ^{-1}$,
on rappelle que $\D ^{(m)} _{X \hookleftarrow Z}$ est isomorphe 
à $\omega _{Z} \otimes _{\O _Z} u ^{*} _{\mathrm{g}} (\D ^{(m)} _X \otimes _{\O _X} \omega _X ^{-1}) $, 
l'indice $g$ signifiant que 
l'on choisit la structure gauche cette fois-ci.
Cependant, par souci de précisions (e.g. pour les calculs locaux), nous prendrons exclusivement la définition
$\D ^{(m)} _{X \hookleftarrow Z}:= u ^{*} _{\mathrm{d}} (\D ^{(m)} _X \otimes _{\O _X} \omega _X ^{-1}) \otimes _{\O _Z} \omega _{Z}$.

Dans cette section, supposons $S$ affine, $X$ affine et muni de coordonnées locales $t _1, \dots, t _n$ relativement à $S$
telles que $Z = V ( t _1, \dots, t _r)$. Notons $\I$ l'idéal de l'immersion fermée $u$. 
Conformément aux notations adoptées dans le préambule de ce papier,
on note $O _S$, $O _X$ et $O _Z$ les sections globales de 
$\O _S$, $\O _X$ et de $\O _Z$ ; de même pour les sections globales d'autres faisceaux. 
On notera alors $\partial  _i ^{< k> _{(m)}}$ pour $i =1,\dots, d$ et $k \in \N$ les opérateurs de $D ^{(m)} _{X/S}$ associés canoniquement à ces coordonnées locales.
Dans ce cas, les éléments de $D  ^{(m)} _{X}$ s'écrivent de manière unique sous la forme
$\sum _{\underline{k} \in \N ^n} a _{\underline{k}} \underline{\partial} ^{<\underline{k}> _{(m)}}$
(resp. $\sum _{\underline{k} \in \N ^n} \underline{\partial} ^{<\underline{k}> _{(m)}} a _{\underline{k}} $),  
la somme n'ayant qu'un nombre fini de termes 
$a _{\underline{k}} \in O _X$ non nuls
et où 
$\underline{\partial} ^{<\underline{k}> _{(m)}} := \prod _{i=1} ^{n} \partial _i ^{<k _i> _{(m)}}$.

\begin{nota}
\label{nota-OSpartial-i}
La sous-$O _S$-algèbre de $D  ^{(m)} _{X}$ engendrée par les éléments
$\{ \partial _1 ^{< k _1> _{(m)}}, \partial _2 ^{< k _2> _{(m)}}, \dots, \partial _r ^{< k _r> _{(m)}} \, |\, k _1, \dots, k _r \in \N\}$
sera notée 
$O _S \{\partial _1,\dots, \partial _r \} ^{(m)}$. 
S'il n'y a pas de risque de confusion (avec les éléments de $D ^{(m)} _{X'}$ où $X ' = V ( t _{r+1},\dots, t _n)$), 
les éléments 
$\underline{\partial} ^{<(\underline{i},\underline{0})> _{(m)}} $, avec
$\underline{i} \in \N ^{r}$ et $\underline{0} \in \N ^{n-r}$ de composantes nulles,
constituant
la $O _S$-base canonique de cette $O _S$-algèbre commutative seront simplement notés
$\underline{\partial} ^{<\underline{i}> _{(m)}}$.

On ne confondra cette $O _S$-algèbre commutative avec le polynôme à puissances divisées de niveau $m$ en $r$ variables à coefficients dans $O _S$ de \cite[1.5.3]{Be1}
et noté
$O _S <\partial _1,\dots, \partial _r > _{(m)}$.
On pourrait d'ailleurs appeler la $O_S$-algèbre $O _S \{\partial _1,\dots, \partial _r \} ^{(m)}$, le polynôme à puissances {\og symétriquement \fg} divisées de niveau $m$ en $r$ variables à coefficients dans $O _S$.
\end{nota}

\begin{nota}
Dans ce qui suit, plusieurs notations et isomorphismes dépendront du choix des coordonnées locales mais
pour alléger les notations nous ne les indiquerons pas.
 
On dispose de l'égalité
$D ^{(m)} _{Z \hookrightarrow X}:= \Gamma (Z, \D ^{(m)} _{Z \hookrightarrow X})
=
D ^{(m)} _{X} /ID ^{(m)} _{X}$.
Cette égalité est compatible avec les structures canoniques 
de $D ^{(m)}_X$-modules à gauche respectives. 
Pour tout $P \in D ^{(m)} _{X}$, on note $[P ] _Z$ son image dans 
$D ^{(m)} _{X} / I D ^{(m)} _{X} $. En particulier, si $a \in O _X$, 
alors $[a  ] _Z$ désigne l'image de 
$a$ dans $O _Z$.

La structure de $D ^{(m)} _Z$-module à gauche sur $D ^{(m)} _{X} /ID ^{(m)} _{X}$
se décrira via l'isomorphisme \ref{D_(Z->X)} ci-dessous. 
\end{nota}

\begin{nota} 
Le module $D ^{(m)} _{Z}   \otimes _{O _S} O _S  \{\partial _1,\dots, \partial _r \} ^{(m)}$
est muni canoniquement d'une
structure de  
$D ^{(m)} _{Z} $-module à gauche.
On bénéficie du morphisme canonique 
$\sigma ^{(m)}   _{Z,X/S}\,:\,
D ^{(m)} _{X}
\to
D ^{(m)} _{Z}   \otimes _{O _S} O _S  \{\partial _1,\dots, \partial _r \} ^{(m)}$
défini en posant
$$\sigma ^{(m)}  _{Z,X/S}
\left (
\sum _{\underline{k} \in \N ^n} a _{\underline{k}} \underline{\partial} ^{<\underline{k}> _{(m)}}
\right)
:=
\sum _{\underline{i} \in \N ^{r}} 
\left (
\sum _{\underline{j} \in \N ^{n-r}}
[ a _{(\underline{i},\underline{j})}
] _Z 
\,
\underline{\partial} ^{<\underline{j}> _{(m)}} \right )
\otimes 
\underline{\partial} ^{<\underline{i}> _{(m)}},$$
les $a _{\underline{k}}\in O _X $ étant nuls sauf pour un nombre fini.
S'il n'y a pas de doute sur la base ou sur le niveau, on notera simplement $\sigma _{Z,X}$.

En considérant la structure de 
$O _X$-module induite par la structure de 
$D ^{(m)} _{X} $-module à gauche
sur $D ^{(m)} _{X} $ (resp. induite par la structure de $D ^{(m)} _{Z} $-module à gauche sur $D ^{(m)} _{Z}   \otimes _{O _S} O _S  \{\partial _1,\dots, \partial _r \} ^{(m)}$), 
on vérifie que 
$\sigma _{Z,X}$ est $O _X$-linéaire  (pour les structures gauches, à ne pas confondre avec
la linéarité comme $D ^{(m)} _{X} $-modules à droite de \ref{sigma-bil}).
De plus, $\sigma _{Z,X}$ est surjective et son noyau est $I D ^{(m)} _{X}$.
On notera alors par 
$\overline{\sigma} ^{(m)} _{Z,X/S}\,:\,D ^{(m)} _{X}/ I D ^{(m)} _{X}
\riso
D ^{(m)} _{Z}   \otimes _{O _S} O _S  \{\partial _1,\dots, \partial _r \} ^{(m)}$
ou simplement $\overline{\sigma} _{Z,X}$
l'isomorphisme $O _Z$-linéaire induit caractérisé par la formule
$$\overline{\sigma} _{Z,X}
\left(
[\sum _{\underline{k} \in \N ^n} a _{\underline{k}} \underline{\partial} ^{<\underline{k}> _{(m)}}] _Z
\right)
:=
\sigma _{Z,X}
\left(
\sum _{\underline{k} \in \N ^n} a _{\underline{k}} \underline{\partial} ^{<\underline{k}> _{(m)}}
\right)
=
\sum _{\underline{i} \in \N ^{r}} 
\left (
\sum _{\underline{j} \in \N ^{n-r}}
[a _{(\underline{i},\underline{j})}] _Z 
\underline{\partial} ^{<\underline{j}> _{(m)}} \right )
\otimes 
\underline{\partial} ^{<\underline{i}> _{(m)}}.$$
\end{nota}

\begin{vide}
[Changement de base et de niveau]
\label{nota-chgt-base-niv}
Soient $m' \geq m$ un entier, 
$S'\to S$ un morphisme de schémas noethériens affines (pour simplifier). 
On note $X' := X \times _S S'$, $Z' := Z \times _S S'$, 
$u'\,:\,Z' \hookrightarrow X'$ l'immersion fermée induite
et
$f \,:\, X' \to X$ 
la projection canonique. 
D'après \cite[2.2.2.2]{Be1} et \cite[2.2.16]{Be1}, on dispose des morphismes d'anneaux 
$f ^{-1} \D ^{(m)} _{X/S} \to f ^{*} \D ^{(m)} _{X/S} \liso \D ^{(m)} _{X'/S'}
\to \D ^{(m')} _{X'/S'}$. 
On dispose ainsi des homomorphismes d'anneaux
$D ^{(m)} _{X/S} \to D ^{(m')} _{X'/S'}$,
$D ^{(m)} _{Z/S} \to D ^{(m')} _{Z'/S'}$.

Les coordonnées locales $t _1, \dots, t _n$ sur $X$ relativement à $S$
induisent canoniquement des coordonnées locales $t '_1, \dots, t '_n$ sur $X'$ relativement à $S'$.
On notera alors $\partial  _i ^{\prime< k> _{(m)}}$ pour $i =1,\dots, d$ et $k \in \N$ les opérateurs de $D ^{(m')} _{X'/S'}$ associés. 
Notons  $O _{S'}  \{\partial ' _1,\dots, \partial '_r \} ^{(m')}$
la sous-$O _{S'}$-algèbre de $D  ^{(m')} _{X'/S'}$ engendrée par les éléments
$\partial _i ^{\prime<  k> _{(m)}}$ pour $1\leq i\leq r$ et $k\in \N$.
Comme l'homomorphisme $D ^{(m)} _{X/S} \to D ^{(m')} _{X'/S'}$ envoie
$\partial  _i ^{< k> _{(m)}}$ sur $\partial  _i ^{\prime< k> _{(m)}}$, il induit la factorisation
$O _S  \{\partial _1,\dots, \partial _r \} ^{(m)}
\to O _{S'}  \{\partial ' _1,\dots, \partial '_r \} ^{(m')}$,
Avec de plus la formule \cite[2.2.3.1]{Be1}, on calcule que les morphismes
$\sigma ^{(m)}  _{Z,X/S}$ commutent aux changements de base et aux changements de niveau, i.e. 
le diagramme canonique commutatif : 
\begin{equation}
\label{sigma-chgt-niv}
\xymatrix{
{D ^{(m)} _{X/S}} 
\ar[r] ^-{\sigma ^{(m)}   _{Z,X/S}} \ar[d] ^-{}
& 
{D ^{(m)} _{Z/S}   \otimes _{O _S} O _S  \{\partial _1,\dots, \partial _r \} ^{(m)} } 
\ar[d] ^-{}
\\ 
{D ^{(m')} _{X'/S'}} 
\ar[r] ^-{\sigma ^{(m')}   _{Z',X'/S'}} 
& 
{ D ^{(m')} _{Z'/S'}   \otimes _{O _{S'}} O _{S'}  \{\partial '_1,\dots, \partial '_r \} ^{(m')}} }
\end{equation}
est commutatif.
\end{vide}

\begin{nota}
\label{nota-ident-DZX-DZ}
Soit $\underline{j} \in \N ^{n-r}$. 
On ne confondra pas l'élément 
$\underline{\partial} ^{<(\underline{0},\underline{j})> _{(m)}} 
\in D ^{(m)} _{X}$
avec l'élément 
$\underline{\partial} ^{<\underline{j}> _{(m)}} 
\in D ^{(m)} _{Z}$. 
Pour éviter les risques de confusion (e.g. pour les formules \ref{partial-Z-X-actionpartial} ou \ref{u_+u^!toif-form}), 
contrairement à la situation de 
\ref{nota-OSpartial-i},
on ne simplifiera pas dans cette situation l'écriture de 
$\underline{\partial} ^{<(\underline{0},\underline{j})> _{(m)}} $
via
$\underline{\partial} ^{<\underline{j}> _{(m)}} $.

On note $D ^{(m)} _{Z,X}$ 
la sous-$O _X$-algèbre
de $D ^{(m)} _{X}$ engendrée par les éléments 
de la forme 
$\underline{\partial} ^{<(\underline{0},\underline{j})> _{(m)}}$,
avec 
$\underline{j} \in \N ^{n-r}$.
Via la formule \cite[2.2.4.(iv)]{Be1}, 
les éléments de
$D ^{(m)} _{Z,X}$
s'écrivent (de manière unique) de la forme
$\sum _{\underline{j} \in \N ^{n-r}}
b _{\underline{j}} 
\underline{\partial} ^{<(\underline{0},\underline{j})> _{(m)}}$ (ou de la forme 
$\sum _{\underline{j} \in \N ^{n-r}}
\underline{\partial} ^{<(\underline{0},\underline{j})> _{(m)}}
b _{\underline{j}} $)
avec $b _{\underline{j}}  \in O _X$ et nuls sauf pour un nombre fini.

Comme $I$ est engendré comme $O _X$-module par $t _1,\dots, t _r$ qui sont dans le centre de 
$D ^{(m)} _{Z,X}$, on obtient 
$D ^{(m)} _{Z,X} I =I D ^{(m)} _{Z,X} $.
Par unicité des écritures décrites ci-dessus,
on dispose de plus
de l'égalité
$I D ^{(m)} _{X}  \cap D ^{(m)} _{Z,X} = I D ^{(m)} _{Z,X}$.
En identifiant $D ^{(m)} _{Z}$ aux éléments de 
$D ^{(m)} _{Z}   \otimes _{O _S} O _S  \{\partial _1,\dots, \partial _r \} ^{(m)}$
dont les coefficients de $\underline{\partial} ^{<\underline{i}> _{(m)}}$ sont nuls pour 
$\underline{i}\not = \underline{0}$, on vérifie la formule
$ \sigma _{Z,X} ^{-1} (D ^{(m)} _{Z}) = D ^{(m)} _{Z,X}$.
L'isomorphisme $\overline{\sigma} _{Z,X}$ se factorise ainsi en 
l'isomorphisme $O _Z$-linéaire (pour les structures gauches)
\begin{equation}
\label{DZXmodI}
\overline{\sigma} _{Z,X}\,:\, 
D ^{(m)} _{Z,X} /D ^{(m)} _{Z,X} I = D ^{(m)} _{Z,X} / I D ^{(m)} _{Z,X} \riso D ^{(m)} _{Z},
\end{equation}
toujours noté $\overline{\sigma} _{Z,X}$ 
si aucune confusion n'est à craindre.
\end{nota}

\begin{vide}
Soient $Z' \hookrightarrow Z$ une seconde immersion fermée, 
$P \in D ^{(m)} _{Z',X}$. Avec les identifications 
du type de \ref{nota-ident-DZX-DZ} (e.g. on voit $\sigma ^{(m)}   _{Z,X} (P)$ comme un élément de
$D ^{(m)} _{Z}$),
on dispose alors de l'égalité de transitivité dans $D ^{(m)} _{Z'}$: 
\begin{equation}
\label{Eg-trans-sigma}
\sigma ^{(m)}   _{Z',Z} \circ \sigma ^{(m)}   _{Z,X} (P)
=
\sigma ^{(m)}   _{Z',X} (P ).
\end{equation}
\end{vide}

\begin{lemm}
\label{Z-X-actionpartial}
Soient $\E$ un $\D ^{(m)} _X$-module à gauche, $x \in E$,
$P \in D ^{(m)} _{Z,X}$, $[ -] _Z$ la surjection canonique $E \to E /IE$.
En se rappelant que $\sigma _{Z,X} (P   ) =\overline{\sigma} _{Z,X} ([P ] _Z )\in D ^{(m)} _{Z}$ (voir \ref{DZXmodI}),
on dispose de la formule 
\begin{equation}
\label{fcoro-Z-X-actionpartial}
\sigma _{Z,X} (P   )\cdot [x] _Z
 = [P \cdot x ] _Z.
\end{equation}

\end{lemm}

\begin{proof}
Comme la formule \ref{fcoro-Z-X-actionpartial} est $O _X$-linéaire par rapport à $P$ (pour la structure gauche de $\O _X$-module
sur $D ^{(m)} _{Z,X}$),
comme le cas où $P \in \O _X$ est clair,
il suffit de prouver que, pour tout $\underline{j} \in \N ^{n-r}$, on dispose de la formule :
\begin{equation}
\label{partial-Z-X-actionpartial}
\underline{\partial} ^{<\underline{j}> _{(m)}} \cdot [ x ] _Z
= 
[ \underline{\partial} ^{<(\underline{0},\underline{j})> _{(m)}} \cdot x ] _Z,
\end{equation}
où $\underline{0} \in \N ^{r}$ est de composantes nulles, 
$\underline{\partial} ^{<\underline{j}> _{(m)}} \in D _Z ^{(m)}$
et
$\underline{\partial} ^{<(\underline{0},\underline{j})> _{(m)}} 
\in 
D _X ^{(m)}$.
Cela découle d'un calcul aisé en se rappelant de la définition 
de la structure de $\D ^{(m)} _Z$-module à gauche sur $u ^* (\E)$
construite fonctoriellement à partir des stratifications correspondant à la structure 
de $\D ^{(m)} _X$-module à gauche sur $\E$ (voir \cite{Be2}).
\end{proof}

\begin{lemm}
\label{DZXmodI-ann}
L'isomorphisme \ref{DZXmodI} 
\begin{equation}
\overline{\sigma} _{Z,X} \,:\,
D ^{(m)} _{Z,X} /D ^{(m)} _{Z,X} I = D ^{(m)} _{Z,X} / I D ^{(m)} _{Z,X} \riso D ^{(m)} _{Z}
\end{equation}
est un isomorphisme de $O _Z$-algèbres.
On dispose ainsi du morphisme de $O _X$-algèbres
$\sigma _{Z,X} \,:\,
D ^{(m)} _{Z,X} \to D ^{(m)} _{Z}$.
\end{lemm}

\begin{proof}
On sait déjà que $\overline{\sigma} _{Z,X}$ est un isomorphisme $O _Z$-linéaire (pour les structures gauches).
Soient $P,Q \in D ^{(m)} _{Z,X}$.
Il reste à établir la formule 
$\overline{\sigma} _{Z,X} ([PQ] _Z)= 
\overline{\sigma} _{Z,X} ([P] _Z)
\overline{\sigma} _{Z,X} ([Q] _Z)$.
Par $O _Z$-linéarité de $\overline{\sigma} _{Z,X}$ (pour les structures gauches), 
il suffit de l'établir pour 
$P = \underline{\partial} ^{<(\underline{0},\underline{j})> _{(m)}}$
et $Q = b\underline{\partial} ^{<(\underline{0},\underline{l})> _{(m)}}$, où $b \in O _X$
et
$\underline{j}, \underline{l} \in \N ^{n-r}$.
D'après la formule 
\cite[2.2.4.(iv)]{Be1}, on calcule dans $D ^{(m)}  _X$:
$$\underline{\partial} ^{<(\underline{0},\underline{j})> _{(m)}}
b\underline{\partial} ^{<(\underline{0},\underline{l})> _{(m)}}
=
\sum _{\underline{j}'\leq \underline{j}}
\left \{ \begin{smallmatrix}   \underline{j} \\   \underline{j}' \\ \end{smallmatrix} \right \}
\underline{\partial} ^{<(\underline{0},\underline{j}-\underline{j}')> _{(m)}}
(b)
\underline{\partial} ^{<(\underline{0},\underline{j}')> _{(m)}}
\underline{\partial} ^{<(\underline{0},\underline{l})> _{(m)}}.
$$
Avec \cite[2.2.4.(iii)]{Be1} utilisé sur $D ^{(m)}  _X$ puis sur $D ^{(m)}  _Z$, on en déduit la formule \ref{1.1.5.2}:
\begin{gather}
\label{1.1.5.2}
\overline{\sigma} _{Z,X}  ([\underline{\partial} ^{<(\underline{0},\underline{j})> _{(m)}}b\underline{\partial} ^{<(\underline{0},\underline{l})> _{(m)}}] _Z)
=
\sum _{\underline{j}'\leq \underline{j}}
\left \{ \begin{smallmatrix}   \underline{j} \\   \underline{j}' \\ \end{smallmatrix} \right \}
[\underline{\partial} ^{<(\underline{0},\underline{j}-\underline{j}')> _{(m)}}
(b)] _Z
\underline{\partial} ^{<\underline{j}'> _{(m)}}
\underline{\partial} ^{<\underline{l}> _{(m)}}\\
\label{1.1.5.3}
=
\sum _{\underline{j}'\leq \underline{j}}
\left \{ \begin{smallmatrix}   \underline{j} \\   \underline{j}' \\ \end{smallmatrix} \right \}
\underline{\partial} ^{<(\underline{j}-\underline{j}'> _{(m)}}
([b] _Z)
\underline{\partial} ^{<\underline{j}'> _{(m)}}
\underline{\partial} ^{<\underline{l}> _{(m)}}
=\underline{\partial} ^{<(\underline{j})> _{(m)}}[b] _Z \underline{\partial} ^{<\underline{l}> _{(m)}}.
\end{gather}
D'après \ref{fcoro-Z-X-actionpartial} utilisé pour $\E = \O_X$, pour tout $\underline{j}''\leq \underline{j}$, on dispose de la formule 
$\underline{\partial} ^{<\underline{j}''> _{(m)}} ([b] _Z)
=
[\underline{\partial} ^{<(\underline{0},\underline{j}'')> _{(m)}}(b)] _Z$, ce qui donne la première égalité de \ref{1.1.5.3}.
La dernière résulte de 
\cite[2.2.4.(iv)]{Be1} sur $D ^{(m)}  _Z$.
D'où le résultat.
\end{proof}

\begin{prop}
\label{sigma-bil}
Les éléments de
$D ^{(m)} _{X}$ s'écrivent de manière unique de la forme
$\sum _{\underline{i} \in \N ^{r}} 
P _{\underline{i}}
\underline{\partial} ^{<(\underline{i},\underline{0})> _{(m)}}
$,
avec 
$P _{\underline{i}} \in D ^{(m)} _{Z,X}$ nuls sauf pour un nombre fini.
L'isomorphisme $\overline{\sigma} _{Z,X}$ est en fait $D ^{(m)} _{Z}$-linéaire et 
se caractérise par la formule:
\begin{align}
\notag
D ^{(m)} _{Z \hookrightarrow X}
=D ^{(m)} _{X} /I D ^{(m)} _{X} 
& \riso
D ^{(m)} _{Z}  \otimes _{O _S} O _S  \{\partial _1,\dots, \partial _r \} ^{(m)}
\\
\label{D_(Z->X)}
[\sum _{\underline{i} \in \N ^{r}} 
P _{\underline{i}}
\underline{\partial} ^{<(\underline{i},\underline{0})> _{(m)}}] _Z&
\mapsto
\sum _{\underline{i} \in \N ^{r}} 
\overline{\sigma} _{Z,X} ([P _{\underline{i}}] _Z )
\otimes
\underline{\partial} ^{<\underline{i}> _{(m)}}.
\end{align}

Via cet isomorphisme $\overline{\sigma} _{Z,X}$,
on munit alors 
$D ^{(m)} _{Z}  \otimes _{O _S} O _S  \{\partial _1,\dots, \partial _r \} ^{(m)}$ 
d'une structure canonique de 
$(D ^{(m)} _{Z}, D ^{(m)} _{X})$-bimodule
prolongeant sa structure de $D ^{(m)} _{Z}$-module à gauche.
Le morphisme $\sigma _{Z,X}$ est alors $D ^{(m)} _{X}$-linéaire à droite.

\end{prop}

\begin{proof}
La $D ^{(m)} _{Z}$-linéarité résulte de l'isomorphisme d'anneaux du lemme \ref{DZXmodI-ann}
et de la formule \ref{fcoro-Z-X-actionpartial} utilisée pour 
$\E = \D ^{(m)} _{X}$.
\end{proof}

\begin{nota}
On identifie $\omega _X$ et $\O _X$, 
$\omega _Z$ et $\O _Z$ (ces identifications dépendent du choix des coordonnées locales que l'on a fixées une bonne fois pour toute dans cette section).
Modulo ces identifications, on obtient l'égalité
$D ^{(m)} _{X \hookleftarrow Z} = D ^{(m)} _{X} / D ^{(m)} _{X} I$.
On vérifie par construction que la structure de $ D ^{(m)} _{X}$-module à gauche sur $D ^{(m)} _{X \hookleftarrow Z} $
correspond à la structure de $ D ^{(m)} _{X}$-module à gauche sur $D ^{(m)} _{X} / D ^{(m)} _{X}I$.
On munit ainsi, via cette identification,
$D ^{(m)} _{X} /D ^{(m)} _{X} I$
d'une structure de 
$(D ^{(m)} _{X}, D ^{(m)} _{Z})$-bimodule prolongeant sa structure canonique de $D ^{(m)} _{X}$-module à gauche.
Pour tout $P \in D ^{(m)} _{X}$, on note $[P ] ' _Z$ son image dans 
$D ^{(m)} _{X} / D ^{(m)} _{X} I$.

Comme $I D ^{(m)} _{X}  \cap D ^{(m)} _{Z,X} = I D ^{(m)} _{Z,X}=  D ^{(m)} _{Z,X} I= D ^{(m)} _{X} I \cap D ^{(m)} _{Z,X}$,
on obtient alors pour tout $P \in D ^{(m)} _{Z,X}$ la formule
$[P ]  _Z= [P ] ' _Z$.
\end{nota}

\begin{nota}
\label{tau(m)ZXS}
Considérons $O _S \{\partial _1,\dots, \partial _r \} ^{(m)} \otimes _{O _S} D ^{(m)} _{Z}   $
muni de sa structure canonique de $D ^{(m)} _{Z}   $-module à droite.
On construit le morphisme 
$\tau ^{(m)} _{Z,X/S}\,:\, D ^{(m)} _{X}
\to
O _S \{\partial _1,\dots, \partial _r \} ^{(m)} \otimes _{O _S} D ^{(m)} _{Z}   $
en posant
$$\sum _{\underline{k} \in \N ^n} \underline{\partial} ^{<\underline{k}> _{(m)}} a _{\underline{k}} 
\mapsto
\sum _{\underline{i} \in \N ^{r}} 
\left (  \underline{\partial} ^{<\underline{i}> _{(m)}}
\otimes 
\sum _{\underline{j} \in \N ^{n-r}}
\underline{\partial} ^{<\underline{j}> _{(m)}} 
[a _{(\underline{i},\underline{j})}] _Z
\right ),
$$
les $a _{\underline{k}}\in O _X $ étant nuls sauf pour un nombre fini.
S'il n'y a pas de doute sur la base et le niveau, on notera simplement $\tau _{Z,X}$.
De plus, les morphismes
$\tau ^{(m)}  _{Z,X/S}$ commutent aux changements de base et de niveau, i.e., 
on dispose du diagramme commutatif analogue à 
\ref{sigma-chgt-niv}.

En considérant la structure de 
$O _X$-module induite par la structure de 
$D ^{(m)} _{X} $-module à droite
sur $D ^{(m)} _{X} $ (resp. induite par la structure de $D ^{(m)} _{Z} $-module à droite sur $O _S \{\partial _1,\dots, \partial _r \} ^{(m)} \otimes _{O _S} D ^{(m)} _{Z} $), 
on vérifie que 
$\tau _{Z,X}$ est $O _X$-linéaire (pour les structures droites, à ne pas confondre avec
la linéarité comme $D ^{(m)} _{X} $-modules à gauche de \ref{tau-DZlin}).
Le morphisme $\tau _{Z,X}$ est surjective, 
de noyau $D ^{(m)} _{X}I$.
On notera 
$\overline{\tau} ^{(m)}_{Z,X/S}\,:\,D ^{(m)} _{X}/ D ^{(m)} _{X} I
\riso
O _S \{\partial _1,\dots, \partial _r \} ^{(m)} \otimes _{O _S} D ^{(m)} _{Z}   $
ou simplement $\overline{\tau} _{Z,X}$
l'isomorphisme $O _Z$-linéaire induit.

\end{nota}

\begin{lemm}
\label{tau=sigma}
L'isomorphisme de la forme 
$D ^{(m)} _{Z,X} /D ^{(m)} _{Z,X} I \riso D ^{(m)} _{Z}$
induit par $\tau _{Z,X}$ est identique à celui induit par $\sigma _{Z,X}$.
\end{lemm}

\begin{proof}
Soit $P \in D ^{(m)} _{Z,X}$. 
Il s'agit de prouver $\sigma _{Z,X} (P ) =\tau _{Z,X} (P ) $.
Par $\O _S$-linéarité de $\sigma _{Z,X}$ et $\tau _{Z,X}$, il suffit de considérer le cas où
$P=\underline{\partial} ^{<(\underline{0},\underline{j})> _{(m)}}b$, avec $\underline{j} \in \N ^{n-r}$ et $b \in O _X$.
Or, il résulte de \ref{DZXmodI-ann} la première égalité: 
$\sigma _{Z,X} (P )  =\sigma _{Z,X} (\underline{\partial} ^{<(\underline{0},\underline{j})> _{(m)}} )  \sigma _{Z,X} (b ) 
=
\underline{\partial} ^{<\underline{j}> _{(m)}} [b] _Z=
\tau _{Z,X} (P )  $.
\end{proof}

\begin{nota}
\label{rappel-adjoint}
Soient $\E$ un $\D _X$-module à gauche, $\M$ un $\D _X$-module à droite, 
$x \in E$, $y \in \M$, $P\in D _X$.
On rappelle que, via l'identification entre $\O _X$ et $\omega _X$ (resp. entre $\O _X$ et  $\omega _X ^{-1}$) 
l'action de $P$ sur $x$ (resp. $y$) pour la structure de 
$\D _X$-module à droite (resp. à gauche) tordue de $\E \otimes _{\O _X} \omega _X$ 
(resp. de $\M \otimes _{\O _X} \omega _X ^{-1}$) 
est $\overset{^t}{} P \cdot x$ (resp. $y \cdot \overset{^t}{} P)$,
où {\og $ \overset{^t}{} P$\fg} désigne l'adjoint de $P$ (voir \cite[1.2]{Be2}).
Il est immédiat que l'on dispose pour tout $P \in D ^{(m)} _{Z,X}$ de la formule 
$ \overset{^t}{} \overline{\sigma} _{Z,X} ([P ] _Z )  = \overline{\tau} _{Z,X} ( [\overset{^t}{} P] _Z ) $.
Avec \ref{tau=sigma}, on obtient ainsi
\begin{equation}
\label{adjoint-comm-sigma}
 \overset{^t}{} \overline{\sigma} _{Z,X} ([P ] _Z )  = \overline{\sigma} _{Z,X} ( [\overset{^t}{} P] _Z ).
\end{equation}
\end{nota}

\begin{lemm}
\label{calcul-local-action-transp}
Soient $\M$ un $\D ^{(m)} _X$-module à droite, 
 $y\in M$,
$P \in D ^{(m)} _{Z,X}$.
Via l'identification de 
$\Gamma (Z, u ^{*}  (\M \otimes _{\O _X} \omega _X ^{-1}) \otimes _{\O _Z} \omega _{Z})$
avec $M /M I$, 
ce dernier est muni d'une structure canonique 
de $D ^{(m)} _Z$-module à droite. 
Notons 
$[ -] '_Z$ la surjection canonique $M \to M /MI$.
On a la formule  
\begin{equation}
\label{fcoro-Z-X-actionpartial-trans}
[y] '_Z \cdot \overline{\sigma} _{Z,X} ([P ] _Z )			    = 				[y \cdot P ] '_Z.
\end{equation}
\end{lemm}

\begin{proof}

Posons $\E := \M \otimes _{\O _X} \omega _X ^{-1}$.
Afin de clarifier la preuve, 
les identifications entre $\omega _X ^{-1}$ et $\O _X$ ou 
entre $\omega _Z$ et $\O _Z$ seront indiquées explicitement
par le symbole $\leftrightarrow$. 
On dispose ainsi du diagramme commutatif par définition:
\begin{equation}
\label{diag-[]'[]}
\xymatrix{
{M}                         \ar[d] ^-{[-] ' _Z}  \ar[r] ^-{}				& { E} \ar[l] ^-{} \ar[d] ^-{[-]  _Z}
\\ 
{M/MI}                    \ar[r] ^-{}				& { E/IE.} \ar[l] ^-{}
}
\end{equation}
Nous allons maintenant traduire (puis vérifier) l'égalité \ref{fcoro-Z-X-actionpartial-trans}
correspondant dans $E /IE$.  
Notons $x \in E$,
l'élément correspondant à $y \in M$. 
En regardant le diagramme \ref{diag-[]'[]}, 
l'élément de $E /IE$ correspondant à 
$[y] '_Z $ est $[x ] _Z$. 
Avec les rappels de \ref{rappel-adjoint}, cela implique que l'élément de $E /IE$ qui correspond à 
$ [y] '_Z \cdot  \overline{\sigma} _{Z,X} ([P ] _Z )$
 est
$\overset{^t}{} ( \overline{\sigma} _{Z,X} ([P ] _Z )) \cdot [x] _Z$.
D'un autre côté, avec \ref{rappel-adjoint}, 
l'élément de $E$ correspond à 
$y \cdot P $ est $\overset{^t}{}  P \cdot x$. 
Via le diagramme commutatif \ref{diag-[]'[]},
il en résulte que l'élément de $E /IE$ qui correspond à 
$[y \cdot P ] '_Z$ est $[\overset{^t}{}  P \cdot x ] _Z $.

Il s'agit ainsi de vérifier dans $E /IE$ l'égalité : 
$\overset{^t}{} ( \overline{\sigma} _{Z,X} ([P ] _Z )) \cdot [x] _Z
=
[\overset{^t}{}  P \cdot x ] _Z $.
Cela découle de \ref{adjoint-comm-sigma} puis \ref{Z-X-actionpartial}
qui induisent les égalités
$ \overset{^t}{} \overline{\sigma} _{Z,X} ([P ] _Z ) \cdot [x] _Z
=
\overline{\sigma} _{Z,X} ( [\overset{^t}{} P] _Z )\cdot [x] _Z
 = 
 [\overset{^t}{}  P \cdot x ] _Z $.
\end{proof}

\begin{prop}
\label{tau-DZlin}
Les éléments de
$D ^{(m)} _{X}$ s'écrivent de manière unique de la forme
$\sum _{\underline{i} \in \N ^{r}} 
\underline{\partial} ^{<(\underline{i},\underline{0})> _{(m)}}
P _{\underline{i}}
$,
avec 
$P _{\underline{i}} \in D ^{(m)} _{Z,X}$ nuls sauf pour un nombre fini.
L'isomorphisme $\overline{\tau} _{Z,X}$ est en fait $D ^{(m)} _{Z}   $-linéaire et est caractérisé par la formule:
\begin{align}
\notag
\overline{\tau} _{Z,X}\,:\,
 D ^{(m)} _{X \hookleftarrow Z} = D ^{(m)} _{X}/D ^{(m)} _{X} I
&
\riso
O _S \{\partial _1,\dots, \partial _r \} ^{(m)} \otimes _{O _S} D ^{(m)} _{Z}  
\\
\label{D_(X<-Z)}
[\sum _{\underline{i} \in \N ^{r}} 
\underline{\partial} ^{<(\underline{i},\underline{0})> _{(m)}} P _{\underline{i}}] ^\prime _Z&
\mapsto
\sum _{\underline{i} \in \N ^{r}} 
\underline{\partial} ^{<(\underline{i})> _{(m)}}
\otimes
\overline{\sigma} _{Z,X} ([P _{\underline{i}}] _Z ).
\end{align}
Via cet isomorphisme $\overline{\tau} _{Z,X}$,
on munit alors 
$O _S \{\partial _1,\dots, \partial _r \} ^{(m)} \otimes _{O _S} D ^{(m)} _{Z}  $ 
d'une structure canonique de 
$(D ^{(m)} _{X},D ^{(m)} _{Z})$-bimodule
prolongeant sa structure de $D ^{(m)} _{Z}$-module à droite.
Le morphisme $\tau _{Z,X}$ est alors $D ^{(m)} _{X}$-linéaire.
\end{prop}

 \begin{proof}
 Pour la validité de la formule, cela découle du lemme \ref{tau=sigma}.
 La $D ^{(m)} _{Z}$-linéarité se déduit de l'isomorphisme d'anneaux du lemme \ref{DZXmodI-ann}
et du lemme
 \ref{calcul-local-action-transp} utilisé pour 
 $\M=\D ^{(m)} _X$.
\end{proof}

\subsection{Adjonction: images directes, images inverses extraordinaire par une immersion fermée}

On reprend les notations et hypothèses de la section \ref{u+-sch}.

\begin{vide}
Soit $\FF$ un  $\D ^{(m)} _{Z}$-module à gauche quasi-cohérent. L'image directe par $u$ de niveau $m$ de $\FF$
est définie dans \cite{Beintro2} en posant $u _+ ^{(m)} (\FF) := u _* ( \D ^{(m)} _{X \hookleftarrow Z} 
\otimes ^{\L}  _{\D ^{(m)} _{Z}} 
\FF)$. 
Comme $\D ^{(m)} _{X \hookleftarrow Z} $ est un $\D ^{(m)} _{Z}$-module localement libre (voir \ref{tau-DZlin}),
le symbole $\L$ est inutile. 
On obtient de plus les isomorphismes
\begin{equation}
\label{u+-descpt} 
\Gamma (X, u ^{(m)}_+ (\FF) )
= 
D ^{(m)} _{X \hookleftarrow Z} 
\otimes   _{D ^{(m)} _{Z}} 
F
\underset{\overline{\tau} _{Z,X} \otimes Id}{\riso}
(O _S \{\partial _1,\dots, \partial _r \} ^{(m)} \otimes _{O _S} D ^{(m)} _{Z} )
\otimes  _{D ^{(m)} _{Z}} 
F
\riso 
O _S \{\partial _1,\dots, \partial _r \} ^{(m)} \otimes _{O _S} F.
\end{equation}
\end{vide}

\begin{vide}
\label{u+-chgt-base-niv}
On reprend les notations de \ref{nota-chgt-base-niv}.
Par abus de notations, on écrira $\O _S$ le faisceau déduit de $\O _S$ par image inverse pour les catégories de faisceaux d'ensemble 
(e.g. $u ^{-1}$).
De même pour $\O _{S'}$.
On dispose ainsi de l'isomorphisme 
$u ^{(m)}_+ (\FF)  \riso 
\O _S \{\partial _1,\dots, \partial _r \} ^{(m)} \otimes _{\O _S} \FF$.
Comme 
$\O _{S'}\otimes _{\O _S} (\D ^{(m)} _{X \hookleftarrow Z/S}) \riso \D ^{(m)} _{X' \hookleftarrow Z'/S'}$
et 
$\O _{S'}\otimes _{\O _S}\D ^{(m)} _{Z/S} 
\riso \D ^{(m)} _{Z'/S'}$, 
on vérifie facilement l'isomorphisme de changement
\begin{equation}
\label{Iso-u+-chgt-base-niv}
\O _{S'}\otimes _{\O _S}  u ^{(m)}_+ (\FF)
\riso 
u ^{\prime (m)}_+ (\O _{S'}\otimes _{\O _S} \FF).
\end{equation}
Pour le cas général, on pourra d'ailleurs consulter \cite[2.4.2]{Beintro2}.
Comme le morphisme $\overline{\tau} ^{(m)} _{Z,X/S}$ 
est compatible aux changements de base et de niveau, 
on vérifie qu'il en est de même de l'isomorphisme \ref{u+-descpt}, i.e., 
pour tout morphisme $\D ^{(m)} _{Z'/S'}$-linéaire de la forme
$\O _{S'}\otimes _{\O _S}\FF \to \FF'$ avec $\FF'$ un $\D ^{(m')} _{Z'/S'}$-module à gauche quasi-cohérent, 
le diagramme canonique
\begin{equation}
\label{u+-descpt-chgt-de-base-niv}
\xymatrix{
{\O _{S'}\otimes _{\O _S} u ^{(m)}_+ (\FF) } 
\ar[d] ^-{\overline{\tau} ^{(m)} _{Z,X/S}} _-{\sim}
\ar[r] _-{\sim}
& 
{u ^{\prime (m)}_+ (\O _{S'}\otimes _{\O _S} \FF) }  
\ar[r] 
\ar[d] ^-{\overline{\tau} ^{(m)} _{Z',X'/S'}} _-{\sim}
& 
{u ^{\prime (m')}_+ (\FF') }  
\ar[d] ^-{\overline{\tau} ^{(m')} _{Z',X'/S'}} _-{\sim}
\\ 
{\O _{S'}\otimes _{\O _S} (\O _S \{\partial _1,\dots, \partial _r \} ^{(m)} \otimes _{\O _S} \FF) }
\ar[r]_-{\sim}
& 
{\O _{S'} \{\partial '_1,\dots, \partial '_r \} ^{(m)} \otimes _{\O _{S'}} (\O _{S'}\otimes _{\O _S} \FF)   } 
\ar[r]
& 
{\O _{S'} \{\partial '_1,\dots, \partial '_r \} ^{(m')} \otimes _{\O _{S'}} (\FF')   } 
} 
\end{equation}
est commutatif.

\end{vide}

\begin{vide}
\label{vide-inv!-nota}
Soit $\E$ un $\D ^{(m)} _{X}$-module à gauche quasi-cohérent.
On dispose de la formule $\mathcal{H} ^0 u ^{!} ( \E)= \cap _{s=1} ^r \ker ( \E \overset{t _s}{\longrightarrow} \E)$
 (on devrait plutôt écrire
$u ^{-1} ( \cap _{s=1} ^r \ker ( \E \overset{t _s}{\longrightarrow} \E))$).
On notera $\mathcal{H} ^0 u ^! (E): = \Gamma (Z, \mathcal{H} ^0 u ^! (\E))
= \cap _{s=1} ^r \ker ( E \overset{t _s}{\longrightarrow} E)$.
On munit 
$\mathcal{H} ^0 u ^{!} ( \E)$
d'une structure canonique de $\D  ^{(m)} _{Z}$-module via l'isomorphisme
$\overline{\sigma} _{Z,X}$
de \ref{DZXmodI-ann} (égal à celui induit par $\tau _{Z,X}$ d'après \ref{tau=sigma}) de la manière suivante: 
pour tout $P \in D ^{(m)} _{Z,X} $, pour tout $x \in \mathcal{H} ^0 u ^! (E)$, 
on pose 
\begin{equation}
\label{def-D-modu!}
\sigma _{Z,X} (P )\cdot x : = P\cdot x.
\end{equation}
\end{vide}

\begin{lemm}
\label{lemm-sigma-tau-act-O-ind}
Soient $\underline{i} \in \N ^{r}$, 
$P \in D ^{(m)} _{Z}$, $Q \in  D ^{(m)} _{Z,X}$. 
Si on dispose dans $D ^{(m)} _{X}$ 
de l'égalité $\underline{\partial} ^{<(\underline{i},\underline{0})> _{(m)}} Q=
Q \underline{\partial} ^{<(\underline{i},\underline{0})> _{(m)}}$, 
on obtient alors dans respectivement 
$D ^{(m)} _{Z}   \otimes _{O _S} O _S \{\partial _1,\dots, \partial _r \} ^{(m)}$
et 
$O _S \{\partial _1,\dots, \partial _r \} ^{(m)}\otimes _{O _S} D ^{(m)} _{Z} $
les formules
\begin{equation}
\label{sigma-tau-act-O-ind}
(P \otimes \underline{\partial} ^{<\underline{i}>} )\cdot  Q
= 
(P  \sigma _{Z,X} (Q) )\otimes \underline{\partial} ^{<\underline{i}> }, 
\hspace{1cm}
Q\cdot (\underline{\partial} ^{<\underline{i}>} \otimes P)
 =
 \underline{\partial} ^{<\underline{i}> } \otimes  ( \sigma _{Z,X} (Q) P ). 
\end{equation}

\end{lemm}

\begin{proof}
Soit $Q ' \in D ^{(m)} _{Z,X}$ tel que 
$P= \sigma _{Z,X} (Q')$. D'après \ref{D_(Z->X)} et \ref{DZXmodI-ann},
on obtient dans $D ^{(m)} _{Z}   \otimes _{O _S} O _S \{\partial _1,\dots, \partial _r \} ^{(m)}$ la première égalité: 
\begin{equation}
\notag
P \sigma _{Z,X} (Q) \otimes \underline{\partial} ^{<\underline{i}> }
=\sigma _{Z,X} (Q 'Q \underline{\partial} ^{<(\underline{i},\underline{0})> })
=
\sigma _{Z,X} (Q' \underline{\partial} ^{<(\underline{i},\underline{0})>} Q)
=
\sigma _{Z,X} (Q' \underline{\partial} ^{<(\underline{i},\underline{0})> } )
\cdot  Q =
( P \otimes \underline{\partial} ^{<\underline{i}>} )
\cdot  Q, 
\end{equation}
l'avant dernière égalité résultant de la $D ^{(m)} _{X}$-linéarité à droite de 
$\sigma _{Z,X}$ (par définition de \ref{sigma-bil}).
On procède de même pour établir la seconde égalité. 
\end{proof}

\begin{vide}
[Morphisme d'adjonction I]

Soit $\FF$ un  $\D ^{(m)} _{Z}$-module à gauche quasi-cohérent.
On construit le morphisme canonique d'adjonction 
\begin{gather}
\label{idtou+u!iso}
\mathrm{adj}\,:\,
\FF 
\to
\mathcal{H} ^{0} u ^{ !} \left (u ^{(m)} _{+}  (\FF) \right ),
\end{gather}
en posant, pour toute section $x\in F$, 
$\mathrm{adj} (x)=1 \otimes x$ (i.e., via l'isomorphisme \ref{u+-descpt}, $\mathrm{adj} (x)=\sum _{\underline{k}\in \N ^{r}} \underline{\partial} ^{<\underline{k}>} \otimes x _{\underline{k}}$ 
avec $x _{\underline{k}} =x$ si $\underline{k}=0$ et $x _{\underline{k}} =0$ sinon).
Comme $1$ est dans le centre de $\D ^{(m)} _{X}$, la formule de droite de \ref{sigma-tau-act-O-ind}
pour $\underline{i}=\underline{0}$ est valable pour tous 
$P \in D ^{(m)} _{Z}$, $Q \in  D ^{(m)} _{Z,X}$.
Via la formule \ref{def-D-modu!},
on en déduit 
que le morphisme \ref{idtou+u!iso} est $\D ^{(m)} _{Z}$-linéaire.
Si $S$ est annulé par une puissance de $p >0$, 
ce morphisme est injectif mais pas forcément surjectif (e.g., considérer $\partial _1 ^{<p ^{N}>} \otimes x$ pour $N$ assez grand).

\end{vide}

\begin{prop}
[Morphisme d'adjonction II]
\label{prop-u_+u^!toif}
Pour tout $\D ^{(m)} _{X}$-module à gauche quasi-cohérent $\E$, on dispose du morphisme canonique de
$\D ^{(m)} _{X}$-modules à gauche :
\begin{equation}
\label{u_+u^!toif}
\mathrm{adj}\,:\,u _+ ^{(m)} (\mathcal{H} ^0 u ^! (\E)) 
\to \E.
\end{equation}
Ce morphisme est caractérisé par la formule: 
pour tout élément de $O _S \{\partial _1,\dots, \partial _r \} ^{(m)} \otimes _{O _S} \mathcal{H} ^0 u ^! (E)$
de la forme
$\sum _{\underline{i}\in \N ^{r}} \underline{\partial} ^{<\underline{i}> _{(m)}} \otimes x _{\underline{i}}$,
avec $x _{\underline{i}} \in \mathcal{H} ^{0} u ^! (\E)$ nuls sauf pour un nombre fini de termes, 
on a modulo l'isomorphisme \ref{u+-descpt} 
\begin{equation}
\label{u_+u^!toif-form}
\mathrm{adj} 
(\sum _{\underline{i}\in \N ^{r}} \underline{\partial} ^{<\underline{i}> _{(m)}} \otimes x _{\underline{i}})
=
\sum _{\underline{i}\in \N ^{r}} \underline{\partial} ^{<(\underline{i},\underline{0})> _{(m)}}  \cdot x _{\underline{i}}.
\end{equation}

\end{prop}

\begin{proof}
Il s'agit de prouver que le morphisme noté $\mathrm{adj}$ et défini par la formule
\ref{u_+u^!toif-form} est $D ^{(m)} _X$-linéaire.
Soient $P \in D ^{(m)} _X$, $x \in \mathcal{H} ^{0} u ^! (\E)$, $\underline{i}\in \N ^{r}$. 
Par $\O _S$-linéarité de $\mathrm{adj}$, il suffit de vérifier 
$\mathrm{adj} (P \cdot (\underline{\partial} ^{<\underline{i}> _{(m)}}\otimes x ))
=
P\underline{\partial} ^{<(\underline{i},\underline{0})> _{(m)}} \cdot x $.
Or, 
comme $\tau _{Z,X}$ est $D ^{(m)} _{X}$-linéaire (voir \ref{tau-DZlin}), 
on a dans $(O _S \{\partial _1,\dots, \partial _r \} ^{(m)} \otimes _{O _S} D ^{(m)} _{Z} )
\otimes _{D ^{(m)} _{Z}} E$ : 
$$P \cdot (\underline{\partial} ^{<\underline{i}> _{(m)}} \otimes 1 ) \otimes x )
= \tau _{Z,X} (P \underline{\partial} ^{<(\underline{i},\underline{0})> _{(m)}} ) \otimes x.$$
Par $\O _S$-linéarité en $P$ de la formule à établir, on se ramène à l'un des trois cas ci-dessous traités séparément. 

\noindent $\bullet$ Si $P= \underline{\partial} ^{<(\underline{0},\underline{j})> _{(m)}}$ avec $\underline{j} \in \N ^{n-r}$,
on calcule 
$\tau _{Z,X} (P \underline{\partial} ^{<(\underline{i},\underline{0})> _{(m)}} ) \otimes x
= ( \underline{\partial} ^{<\underline{i}> _{(m)}} \otimes 1)
\otimes \underline{\partial} ^{<\underline{j}> _{(m)}} \cdot x$.
Comme d'après \ref{def-D-modu!}
$\underline{\partial} ^{<\underline{j}> _{(m)}} \cdot x= 
\underline{\partial} ^{<(\underline{0},\underline{j})> _{(m)}} \cdot x$,
on en tire:
$$\mathrm{adj} (P \cdot (\underline{\partial} ^{<\underline{i}> _{(m)}}\otimes x ))=
\underline{\partial} ^{<(\underline{i},\underline{0})> _{(m)}}  \underline{\partial} ^{<(\underline{0},\underline{j})> _{(m)}} \cdot x
=
 \underline{\partial} ^{<(\underline{0},\underline{j})> _{(m)}}  \underline{\partial} ^{<(\underline{i},\underline{0})> _{(m)}} \cdot x
 =
 P \underline{\partial} ^{<(\underline{i},\underline{0})> _{(m)}} \cdot x.$$
 
 \noindent $\bullet$ Si $P= \underline{\partial} ^{<(\underline{i}',\underline{0})> _{(m)}}$ avec $\underline{i}' \in \N ^{r}$,
 cela découle d'un calcul analogue en utilisant de plus la formule \cite[2.2.4.(iii)]{Be1}.
 
\noindent $\bullet$ Il reste à traiter le cas où
$P =a \in O _X$.
En prenant l'adjoint de la formule
la formule 
\cite[2.2.4.(iv)]{Be1}, on calcule dans $D ^{(m)}  _X$:
\begin{equation}
\label{adj-2.2.4iv}
a\underline{\partial} ^{<(\underline{i}, \underline{0})> _{(m)}}
=
\sum _{\underline{i}'\leq \underline{i}}
(-1) ^{| \underline{i}-\underline{i}'| }
\left \{ \begin{smallmatrix}   \underline{i} \\   \underline{i}' \\ \end{smallmatrix} \right \}
\underline{\partial} ^{<(\underline{i}', \underline{0})> _{(m)}}
\underline{\partial} ^{<(\underline{i}-\underline{i}', \underline{0})> _{(m)}}
(a).
\end{equation}
D'où 
$\tau _{Z,X} (a \underline{\partial} ^{<(\underline{i},\underline{0})> _{(m)}} ) \otimes x
=
\sum _{\underline{i}'\leq \underline{i}}
(-1) ^{| \underline{i}-\underline{i}'| }
\left \{ \begin{smallmatrix}   \underline{i} \\   \underline{i}' \\ \end{smallmatrix} \right \}
(\underline{\partial} ^{<\underline{i}'> _{(m)}} \otimes 1)
\otimes 
[\underline{\partial} ^{<(\underline{i}-\underline{i}', \underline{0})> _{(m)}} 
(a)] _{Z} x.$
Il en résulte 
$$\mathrm{adj} (a \cdot (\underline{\partial} ^{<\underline{i}> _{(m)}}\otimes x ))
=
\sum _{\underline{i}'\leq \underline{i}}
(-1) ^{| \underline{i}-\underline{i}'| }
\left \{ \begin{smallmatrix}   \underline{i} \\   \underline{i}' \\ \end{smallmatrix} \right \}
\underline{\partial} ^{<(\underline{i}', \underline{0})> _{(m)}}
\cdot 
([\underline{\partial} ^{<(\underline{i}-\underline{i}', \underline{0})> _{(m)}}  (a)] _Z x).$$
D'après \ref{def-D-modu!}, on a par définition:
$[\underline{\partial} ^{<(\underline{i}-\underline{i}', \underline{0})> _{(m)}}  (a)] _Z x
=
\underline{\partial} ^{<(\underline{i}-\underline{i}', \underline{0})> _{(m)}}  (a) x$.
En réutilisant la formule \ref{adj-2.2.4iv} (à l'envers), on termine la vérification de l'égalité.
\end{proof}

\begin{rema}
L'isomorphisme de transposition $\gamma \,:\, 
D ^{(m)} _{X} \to D ^{(m)} _{X}$ se factorise par l'isomorphisme
$\gamma \,:\, 
D ^{(m)} _{X} / I D ^{(m)} _{X}
\riso D ^{(m)} _{X} /D ^{(m)} _{X} I$.
\end{rema}

\subsection{Changement de base pour les faisceaux quasi-cohérents sur les schémas}

\begin{nota}
\label{nota-diag-chgt-base}
Soit $T\to S$ un morphisme de schémas noethériens affines.
Dans toute cette section, nous considérons les  deux carrés cartésiens 
\begin{equation}
\xymatrix{
{Z} 
\ar@{^{(}->}[r] ^-{u}
& 
{X } 
\\ 
{Z'} 
\ar@{^{(}->}[r] ^-{u'}
\ar@{^{(}->}[u] ^-{b}
& 
{X',} 
\ar@{^{(}->}[u] ^-{a}
}
\
\
\
\xymatrix{
{Z _T} 
\ar@{^{(}->}[r] ^-{u _T}
& 
{X _T} 
\\ 
{Z' _T} 
\ar@{^{(}->}[r] ^-{u' _T}
\ar@{^{(}->}[u] ^-{b _T}
& 
{X' _T,} 
\ar@{^{(}->}[u] ^-{a _T}
}
\end{equation}
dont les morphismes du carré de gauche sont des immersions fermées de $S$-schémas lisses, 
dont le carré de droite se déduit du carré de gauche par changement de base via $T\to S$.

On suppose de plus $X$ affine, muni de coordonnées locales $t _1, \dots, t _d$ telles qu'il existe des entiers $1\leq r \leq d-1$, $1\leq s \leq d-r$ vérifiant
$X'= V ( t_1, \dots, t _r)$ et $Z=V ( t _{r+1},\dots, t _{r+s})$.
On note $I$ l'idéal engendré par  $t_1, \dots, t _r$
et
$J$ l'idéal engendré par $t _{r+1},\dots, t _{r+s}$.

L'image inverse (en tant que $\O_X$-module, i.e., à ne pas confondre avec l'image inverse en tant que 
$\D ^{(m)} _{X}$-module) par $a$ d'un $\D ^{(m)} _{X}$-module $\E $ se notera
\begin{equation}
\label{def-a^*-sch}
\L a ^{*} (\E) := \D  ^{(m)} _{Z \hookrightarrow X} \otimes ^{\L} _{a ^{-1}\D  ^{(m)} _{X}}
a ^{-1} \E \in D ^{\mathrm{b}} (\D ^{(m)} _{Z}) .
\end{equation}
De même, l'image inverse (en tant que $\O_Z$-module) par $b$ d'un $\D ^{(m)} _{Z}$-module $\FF$ se notera 
$\L b ^{*} (\FF)$. 
\end{nota}

\begin{lemm}
\label{formules-I-J-ind}
On dispose des égalités 
\begin{gather}
\label{formules-I-J-ind1}
(D ^{(m)} _{X'} \otimes _{O _S} O _S \{\partial _1,\dots, \partial _r \} ^{(m)} ) J 
=
D ^{(m)} _{X'} J \otimes _{O _S} O _S \{\partial _1,\dots, \partial _r \} ^{(m)} 
\\
\label{formules-I-J-ind2}
I(O _S \{\partial _{r+1},\dots, \partial _{r+s} \} ^{(m)} \otimes _{O _S} D ^{(m)} _{Z}  )
=
O _S \{\partial _{r+1},\dots, \partial _{r+s} \} ^{(m)} \otimes _{O _S} I D ^{(m)} _{Z}  .
\end{gather}
\end{lemm}

\begin{proof}
L'égalité \ref{formules-I-J-ind1} découle de la formule \ref{sigma-tau-act-O-ind} (en remplaçant dans les notations le symbole $Z$ par $X'$)
et du fait que $J$ est engendré par $t _{r+1},\dots, t _{r+s}$ qui commutent aux 
éléments de la forme 
$\underline{\partial} ^{<(\underline{i},\underline{0})> _{(m)}}$,
où $\underline{i} \in \N ^{r}$.
De même, on obtient l'égalité \ref{formules-I-J-ind2} à partir  de la seconde formule de \ref{sigma-tau-act-O-ind}.
\end{proof}

\begin{lemm}
\label{lemm-iso-chgt-de-basebis}
Avec les notations de \ref{nota-diag-chgt-base}, 
soit
$\FF $ un $\D ^{(m)} _{Z}$-module quasi-cohérent. 
Dans ce cas, on dispose de l'isomorphisme de $\D ^{(m)} _{X'}$-modules:
\begin{equation}
\label{iso-chgt-de-basebis}
a ^{*} \circ u ^{(m)} _+ (\FF) 
\riso 
u ^{\prime^{(m)}} _+  \circ b ^{*} (\FF) .
\end{equation}
En outre, celui-ci commute aux morphismes de changement de niveau et de base, i.e., avec les notations de \ref{u+-chgt-base-niv}, pour tout $m'\geq m$, pour
tout $\D ^{(m')} _{Z _T /T}$-module quasi-cohérent $\FF '$, pour tout morphisme $\D ^{(m)} _{Z/S}$-linéaire de la forme
$\O _T \otimes _{\O_S}\FF \to \FF'$, 
le diagramme canonique 
\begin{equation}
\label{iso-chgt-de-basebis-chgtbase}
\xymatrix{
{\O _T \otimes _{\O_S} ( a ^{*} \circ u ^{(m)} _+ (\FF) )} 
\ar[r] ^-{\sim}
\ar[d] ^-{\sim}
& 
{a _T ^{*} \circ u  ^{(m)} _{T+} (\O _T \otimes _{\O_S}\FF) } 
\ar[d] ^-{\sim}
\ar[r] ^-{}
&
{a _T ^{*} \circ u  ^{(m')} _{T+} (\FF') } 
\ar[d] ^-{\sim}
\\ 
{\O _T \otimes _{\O_S} (u ^{\prime^{(m)}} _+  \circ b ^{*} (\FF) )} 
\ar[r] ^-{\sim}
& 
{u  ^{\prime^{(m)}} _{T+}  \circ b _T ^{*} (\O _T \otimes _{\O_S}\FF) } 
\ar[r] ^-{}
&
{u  ^{\prime^{(m')}} _{T+}  \circ b _T ^{*} (\FF') ,} 
}
\end{equation}
où les isomorphismes verticaux sont ceux de la forme \ref{iso-chgt-de-basebis} et 
où les isomorphismes horizontaux de gauche se déduisent de \ref{Iso-u+-chgt-base-niv}, 
est commutatif. 
\end{lemm}

\begin{proof}
Via les théorèmes de type $A$, il suffit de le vérifier pour les sections globales. 
On obtient 
\begin{align}
a ^{*} \circ u _+ ^{(m)} (F) &=
(D ^{(m)} _{X}/ I D ^{(m)} _{X} ) \otimes _{D ^{(m)} _{X}}(D ^{(m)} _{X}/ D ^{(m)} _{X} J )\otimes _{D ^{(m)} _{Z}} F\\
u ^{\prime(m)} _+   \circ b ^{*} (F) & =
(D ^{(m)} _{X'}/ D ^{(m)} _{X'} J )\otimes _{D ^{(m)} _{Z'}} (D ^{(m)} _{Z}/ I D ^{(m)} _{Z} ) \otimes _{D ^{(m)} _{Z}} F .
\end{align}

On bénéficie des isomorphismes $D ^{(m)} _{X'}$-linéaires :
\begin{align}
\notag
(D ^{(m)} _{X}/ I D ^{(m)} _{X} ) \otimes _{D ^{(m)} _{X}}(D ^{(m)} _{X}/ D ^{(m)} _{X} J )
& \underset{\overline{\sigma} _{X',X}\otimes Id}{\riso} 
(D ^{(m)} _{X'} \otimes _{O _S} O _S \{\partial _1,\dots, \partial _r \} ^{(m)} )
\otimes _{D ^{(m)} _{X}}(D ^{(m)} _{X}/ D ^{(m)} _{X} J )
\\
\notag
&
\riso 
(D ^{(m)} _{X'} \otimes _{O _S} O _S \{\partial _1,\dots, \partial _r \} ^{(m)} ) 
/
(D ^{(m)} _{X'} \otimes _{O _S} O _S \{\partial _1,\dots, \partial _r \} ^{(m)} ) J\\
\label{1225}
& \riso 
(D ^{(m)} _{X'}/ D ^{(m)} _{X'} J )\otimes _{O _S} O _S \{\partial _1,\dots, \partial _r \} ^{(m)},
\end{align}
le dernier isomorphisme résultant de la première égalité de \ref{formules-I-J-ind}.
On dispose de plus des isomorphismes $D ^{(m)} _{Z}$-linéaires :
\begin{align}
\notag
(D ^{(m)} _{X}/ I D ^{(m)} _{X} ) \otimes _{D ^{(m)} _{X}}(D ^{(m)} _{X}/ D ^{(m)} _{X} J )
& 
\underset{Id \otimes \overline{\tau} _{Z,X}}{\riso} 
(D ^{(m)} _{X}/ I D ^{(m)} _{X} ) \otimes _{D ^{(m)} _{X}}
(O _S \{\partial _{r+1},\dots, \partial _{r+s} \} ^{(m)} \otimes _{O _S} D ^{(m)} _{Z}  )
\\
\label{1226}
& \riso 
O _S \{\partial _{r+1},\dots, \partial _{r+s} \} ^{(m)} \otimes _{O _S} D ^{(m)} _{Z} / I D ^{(m)} _{Z},
\end{align}
dont le dernier isomorphisme découle de la seconde égalité de \ref{formules-I-J-ind}.
Il en dérive le carré suivant:
\begin{equation}
\label{carre-but-bili}
\xymatrix{
{(D ^{(m)} _{X}/ I D ^{(m)} _{X} ) \otimes _{D ^{(m)} _{X}}(D ^{(m)} _{X}/ D ^{(m)} _{X} J )} 
\ar[r] ^-{\sim} _-{\ref{1225}}
\ar[d] ^-{\sim} _-{\ref{1226}}
& 
{(D ^{(m)} _{X'}/ D ^{(m)} _{X'} J )\otimes _{O _S} O _S \{\partial _1,\dots, \partial _r \} ^{(m)} } 
\ar[d] ^-{\sim} _-{\overline{\tau} _{Z',X'}\otimes Id}
\\ 
{O _S \{\partial _{r+1},\dots, \partial _{r+s} \} ^{(m)} \otimes _{O _S} D ^{(m)} _{Z}/  I D ^{(m)} _{Z}} 
\ar[r] ^-{\sim} _-{Id \otimes \overline{\sigma} _{Z',Z}}
& 
{O _S \{\partial _{r+1},\dots, \partial _{r+s} \} ^{(m)} \otimes _{O _S}
D ^{(m)} _{Z'}
\otimes _{O _S} O _S \{\partial _1,\dots, \partial _r \} ^{(m)}. } 
}
\end{equation}
Le carré \ref{carre-but-bili} est commutatif. En effet, soient 
$P \in D ^{(m)} _{Z',X}$,
$\underline{i} \in \N ^r$, $\underline{j} \in \N ^s$.  
Les éléments du terme en haut à gauche $(D ^{(m)} _{X}/ I D ^{(m)} _{X} ) \otimes _{D ^{(m)} _{X}}(D ^{(m)} _{X}/ D ^{(m)} _{X} J )$
sont des sommes d'éléments de la forme 
$Q:=[\underline{\partial} ^{<(\underline{0},\underline{j}, \underline{0})> _{(m)}}
P
\underline{\partial} ^{<(\underline{i},\underline{0}, \underline{0})> _{(m)}}] _{X'} \otimes [1] ^\prime  _Z$.
On calcule que cet élément 
$Q$
s'envoie via \ref{1225} sur 
$[\underline{\partial} ^{<(\underline{j}, \underline{0})> _{(m)}}
\sigma ^{(m)}   _{X',X}  (P )  ] _{Z'}
\otimes 
\underline{\partial} ^{<\underline{i}> _{(m)}}$.
Avec la formule de transitivité \ref{Eg-trans-sigma},
ce dernier s'envoie via la flèche de droite du carré \ref{carre-but-bili}
sur 
$\underline{\partial} ^{<\underline{j}> _{(m)}} \otimes  \sigma ^{(m)}   _{Z',X}  (P )  \otimes \underline{\partial} ^{<\underline{i}> _{(m)}}$.
D'un autre côté, 
comme 
$Q
=
[1] _{X'} \otimes [\underline{\partial} ^{<(\underline{0},\underline{j}, \underline{0})> _{(m)}}
P
\underline{\partial} ^{<(\underline{i},\underline{0}, \underline{0})> _{(m)}}] ^\prime _Z$,
la flèche de gauche du carré de \ref{carre-but-bili}
envoie 
$Q$ sur 
$\underline{\partial} ^{<\underline{j}> _{(m)}} \otimes [\sigma ^{(m)}   _{Z,X}  (P )\underline{\partial} ^{<(\underline{i},\underline{0})> _{(m)}} ] _{Z'}$.
Ce dernier s'envoie via la flèche du bas sur
$\underline{\partial} ^{<\underline{j}> _{(m)}} \otimes \sigma ^{(m)}   _{Z',X}  (P ) \otimes \underline{\partial} ^{<\underline{i}> _{(m)}}$.
On en déduit la commutativité du diagramme \ref{carre-but-bili}.

Considérons à présent le diagramme canonique: 
\begin{equation}
\label{carre-but-bili-bis}
\xymatrix{
{(D ^{(m)} _{X'}/ D ^{(m)} _{X'} J )\otimes _{D ^{(m)} _{Z'}} (D ^{(m)} _{Z}/ I D ^{(m)} _{Z} )} 
\ar[r] ^-{\sim} _-{Id \otimes \overline{\sigma} _{Z',Z}}
\ar[d] ^-{\sim}_-{\overline{\tau} _{Z',X'}\otimes Id}
& 
{(D ^{(m)} _{X'}/ D ^{(m)} _{X'} J )\otimes _{O _S} O _S \{\partial _1,\dots, \partial _r \} ^{(m)} } 
\ar[d] ^-{\sim} _-{\overline{\tau} _{Z',X'}\otimes Id}
\\ 
{O _S \{\partial _{r+1},\dots, \partial _{r+s} \} ^{(m)} \otimes _{O _S} D ^{(m)} _{Z} / I D ^{(m)} _{Z}} 
\ar[r] ^-{\sim} _-{Id \otimes \overline{\sigma} _{Z',Z}}
& 
{O _S \{\partial _{r+1},\dots, \partial _{r+s} \} ^{(m)} \otimes _{O _S}
D ^{(m)} _{Z'}
\otimes _{O _S} O _S \{\partial _1,\dots, \partial _r \} ^{(m)}.} 
}
\end{equation}
Ce carré \ref{carre-but-bili-bis} est commutatif. 
En effet, pour tous $\underline{i} \in \N ^r$, $\underline{j} \in \N ^s$,
$P \in  D ^{(m)} _{Z',X'}$, $Q\in  D ^{(m)} _{Z',Z}$, on calcule que
$[\underline{\partial} ^{<(\underline{j},\underline{0})> _{(m)}} P ] ' _{Z'}
\otimes 
[ Q \underline{\partial} ^{<(\underline{i},\underline{0})> _{(m)}}  ] _{Z'}$
s'envoie via les deux chemins possibles de \ref{carre-but-bili-bis} sur 
$\underline{\partial} ^{<\underline{j}> _{(m)}} \otimes \sigma _{Z',X'}  (P )
\sigma   _{Z',Z}  (Q)  \otimes \underline{\partial} ^{<\underline{i}> _{(m)}}$.

Comme les flèches de droite (resp. du bas) des carrés 
\ref{carre-but-bili} et \ref{carre-but-bili-bis}
sont identiques, 
comme les flèches de gauche (resp. du haut) des carrés 
\ref{carre-but-bili} et \ref{carre-but-bili-bis}
sont $D ^{(m)} _{Z}$-linéaires (resp. $D ^{(m)} _{X'}$-linéaires), 
il en résulte l'isomorphisme canonique
de $(D ^{(m)} _{X'}, D ^{(m)} _{Z})$-bimodules
\begin{equation}
(D ^{(m)} _{X}/ I D ^{(m)} _{X} ) \otimes _{D ^{(m)} _{X}}(D ^{(m)} _{X}/ D ^{(m)} _{X} J )
\riso 
(D ^{(m)} _{X'}/ D ^{(m)} _{X'} J )\otimes _{D ^{(m)} _{Z'}} (D ^{(m)} _{Z}/ I D ^{(m)} _{Z} ).
\end{equation}
D'où l'isomorphisme \ref{iso-chgt-de-basebis}.
Vérifions à présent sa commutation au changement de niveaux, i.e. supposons $S=T$. 
Par fonctorialité de l'isomorphisme \ref{iso-chgt-de-basebis}, 
il s'agit de vérifier la commutativité du diagramme 
\begin{equation}
\label{iso-chgt-de-basebis-chgtbase-2}
\xymatrix{
{a ^{*} \circ u ^{(m)} _+ (\FF') } 
\ar[r] ^-{}
\ar[d] ^-{\sim}
& 
{a ^{*} \circ u ^{(m')} _+ (\FF') } 
\ar[d] ^-{\sim}
\\ 
{u ^{\prime^{(m)}} _+  \circ b ^{*} (\FF') } 
\ar[r] ^-{}
& 
{u ^{\prime^{(m')}} _+  \circ b ^{*} (\FF') .} 
}
\end{equation}
Pour cela, il suffit de constater que les isomorphismes des carrés \ref{carre-but-bili} et \ref{carre-but-bili-bis} 
commutent au changement de niveaux.
Cette vérification résulte facilement de la commutation aux changements de niveau
des isomorphismes de la forme $\overline{\sigma}$ (i.e., on dispose du carré commutatif \ref{sigma-chgt-niv})
et de ceux de la forme $\overline{\tau}$ de 
\ref{tau-DZlin}. De même, on vérifie qu'il commute aux changements de base, i.e. on valide la commutativité du carré de gauche de 
\ref{iso-chgt-de-basebis-chgtbase}.
\end{proof}

\begin{lemm}
\label{acycl-u+}
Soit $\FF $ un $\D ^{(m)} _{Z}$-module quasi-cohérent plat.
On dispose alors de l'annulation: 
\begin{equation}
\forall\, n \not = 0,\?  \mathcal{H} ^{n} \L a ^{*} u _+ (\FF) = 0.
\end{equation}
\end{lemm}

\begin{proof}
On procède par récurrence sur $r \geq 1$.
Dans un premier temps, supposons
$X ' = V (t _1)$. Dans ce cas, pour $n \not \in \{-1,0 \}$, $\mathcal{H} ^{n} \L a ^{*} u _+ (\FF)=0$.
De plus, $\mathcal{H} ^{-1} \L a ^{*} u _+ (\FF)  = \ker (u _+ (\FF) \underset{ t_1}{\longrightarrow}
u _+ (\FF))$.
Or, d'après \ref{u+-descpt},
$ u _+ (\FF) \riso O _S \{\partial _{r+1},\dots, \partial _{r+s} \} ^{(m)} \otimes _{O _S} \FF$.
Comme $\FF$ n'a pas de $t _1$ torsion, 
le lemme \ref{lemm-sigma-tau-act-O-ind} nous permet de conclure 
(en effet $t _1$ commute aux 
$\underline{\partial} ^{<(\underline{0},\underline{j}, \underline{0})> _{(m)}}$,
avec $\underline{j} \in \N ^s$).

Supposons à présent $r \geq 2$ et posons
$X''= V ( t_1, \dots, t _{r-1})$ et $Z''=Z \cap X''$, 
$b ' \,:\, Z '' \hookrightarrow Z$, 
$a ' \,:\, X '' \hookrightarrow X$, 
$a '' \,:\, X ' \hookrightarrow X''$, 
$u'' \,:\, Z '' \hookrightarrow X''$
les immersions fermées induites.
Par hypothèse de récurrence et d'après \ref{iso-chgt-de-basebis},
on obtient:
$\L a ^{\prime *} \circ u _+ (\FF) 
\riso 
a ^{\prime *} \circ u _+ (\FF) 
\riso 
u ^{\prime \prime} _+  \circ b ^{\prime *} (\FF)$. D'après le cas $r =1$,
on vérifie aussi 
$\L a ^{\prime \prime *} \circ u ^{\prime \prime} _+  \circ b ^{\prime *} (\FF) 
\riso 
a ^{\prime \prime *} \circ u ^{\prime \prime} _+  \circ b ^{\prime *} (\FF) $.
Il en dérive: 
$\L a ^{\prime\prime *} \circ \L a ^{\prime  *} \circ u _+ (\FF)
\riso a ^{\prime\prime *} \circ  a ^{\prime  *} \circ u _+ (\FF) \riso a ^{*} \circ u _+ (\FF)$.
L'isomorphisme $\L a ^{*} \riso \L a ^{\prime\prime *} \circ \L a ^{\prime  *}$ nous permet de conclure. 
\end{proof}

\begin{prop}
\label{prop-iso-chgt-de-base}
On bénéficie pour tout $\D ^{(m)} _{Z}$-quasi-cohérent $\FF $ de l'isomorphisme 
\begin{equation}
\label{iso-chgt-de-base}
a ^{!} \circ u _+ (\FF) \riso u ^{\prime} _+  \circ b ^{!} (\FF) .
\end{equation}
\end{prop}

\begin{proof}
Soit $\PP$ une résolution gauche de $\FF$ par des $\D ^{(m)} _{Z}$-modules quasi-cohérents plats 
(il suffit de prendre une résolution de $F$ par des $D ^{(m)} _{Z}$-modules plats puis on applique le foncteur 
$\D ^{(m)} _{Z} \otimes _{D ^{(m)} _{Z}}-)$. 
Comme les foncteurs $u _+$ et $u ^{\prime} _+$ sont exactes, 
par \ref{acycl-u+}, on obtient 
$\L a ^{*} u _+ (\FF) \riso a ^{*} u _+ (\PP)$
et
$u ^{\prime} _+  \circ \L b ^{*} (\FF) \riso 
u ^{\prime} _+  \circ b ^{*} (\PP)$.
Le lemme \ref{lemm-iso-chgt-de-basebis}
nous permet de conclure.
\end{proof}

\section{Immersion fermée de $\V$-schémas formels lisses: adjonction pour les $D$-modules}
\label{section2.1}
Soient $m$ un entier positif, $u\,:\,\ZZ \hookrightarrow \X$ une immersion fermée de $\V$-schémas formels lisses d'idéal $\I$. 
Berthelot a défini la catégorie dérivée des complexes de $\widehat{\D}  ^{(m)} _{\X}$-modules quasi-cohérents à cohomologie bornée,
notée $D ^\mathrm{b} _{\mathrm{qc}} (\widehat{\D}  ^{(m)} _{\X})$
(voir \cite{Beintro2}). De même sur $\ZZ$. Il a de plus construit les foncteurs image directe par 
$u$ de niveau $m$ noté
$ u _+ ^{(m)}\,:\, D ^\mathrm{b} _{\mathrm{qc}} (\widehat{\D}  ^{(m)} _{\ZZ}) \to D ^\mathrm{b} _{\mathrm{qc}} (\widehat{\D}  ^{(m)} _{\X})$
et le foncteur image inverse extraordinaire par 
$u$ de niveau $m$ noté
$ u  ^{!(m)}\,:\, D ^\mathrm{b} _{\mathrm{qc}} (\widehat{\D}  ^{(m)} _{\X}) \to D ^\mathrm{b} _{\mathrm{qc}} (\widehat{\D}  ^{(m)} _{\ZZ})$ (e.g., voir \cite{Beintro2}).
Pour alléger, on pourra simplement noter $ u _+ $ ou $ u  ^{!}$.

\subsection{Manque de stabilité des faisceaux quasi-cohérents sans $p$-torsion sur un schéma formel}
La notion de faisceau quasi-cohérent sans $p$-torsion apparaît naturellement dans la théorie des $\D$-modules arithmétiques. 
Malheureusement, cette catégorie n'est pas stable a priori par le foncteur $\mathcal{H} ^{0} u ^{!}$ (voir 
\ref{pas-stab-qc}). C'est pour cette raison que nous travaillerons par la suite (e.g. dans la section \ref{section2.2}) avec leur section globale
et non avec le faisceau tout entier.

\begin{defi}
\label{def-qcsstor}
\begin{itemize}
\item Le séparé complété (pour la topologie $p$-adique) d'un $\D  ^{(m)} _{\X}$-module ${\E}$ est le $\widehat{\D}  ^{(m)} _{\X}$-module défini en posant 
$\widehat{\E}:=  \underset{\underset{i}{\longleftarrow}}{\lim\,} 
\left ( 
\D  ^{(m)} _{X _i} \otimes _{\D  ^{(m)} _{\X} }{\E}
\right )$.

\item Un $\D  ^{(m)} _{\X}$-module ${\E}$ est séparé complet  si le morphisme canonique 
\begin{equation}
\label{def-qcmod}
{\E}
\to
\underset{\underset{i}{\longleftarrow}}{\lim\,} 
\left ( 
\D  ^{(m)} _{X _i} \otimes _{\D  ^{(m)} _{\X} }{\E}
\right )=\widehat{\E}
\end{equation}
est un isomorphisme. 

\item Un $\D  ^{(m)} _{\X}$-module ${\E}$ est sans $p$-torsion si pour tout ouvert $\U$ de $\X$, 
$\Gamma (\U, {\E}) $ est sans $p$-torsion.

\item Un $\D  ^{(m)} _{\X}$-module ${\E}$  
est pré-pseudo-quasi-cohérent   
si, pour tout entier $i\in \N$, le $\O _{X _i}$-module
$\D  ^{(m)} _{X _i} \otimes _{\D  ^{(m)} _{\X} }{\E}$ est quasi-cohérent.

\item Un $\widehat{\D}  ^{(m)} _{\X}$-module ${\E}$  
est pseudo-quasi-cohérent s'il est séparé complet et pré-pseudo-quasi-cohérent.  

\item Un $\widehat{\D}  ^{(m)} _{\X}$-module ${\E}$ sans $p$-torsion 
est quasi-cohérent s'il est pseudo-quasi-cohérent.

\end{itemize}
\end{defi}

\begin{lemm}
\label{completion-adjoint}
Si $\E$ est un $\D  ^{(m)} _{\X}$-module pré-pseudo-quasi-cohérent, 
alors 
$\widehat{\E}$ est
un $\widehat{\D}  ^{(m)} _{\X}$-module pseudo-quasi-cohérent. 
\end{lemm}

\begin{proof}
Soit $\E$ un $\D  ^{(m)} _{\X}$-module pré-pseudo-quasi-cohérent.
Il résulte de \cite[3.3.1.2]{Be1} que le morphisme canonique
$\widehat{\E}/ \pi ^{i+1} \widehat{\E}= \D  ^{(m)} _{X _i}  \otimes _{\D  ^{(m)} _{\X} } \widehat{\E} 
\to
\D  ^{(m)} _{X _i}  \otimes _{\D  ^{(m)} _{\X} } \E $ est un isomorphisme (on remarque que l'hypothèse de pré-pseudo-quasi-cohérence est a priori nécessaire).
Il en résulte que 
$\widehat{\E}$ est
un $\widehat{\D}  ^{(m)} _{\X}$-module pseudo-quasi-cohérent. 
\end{proof}

\begin{vide}
\label{const-pseudo-qc}
Supposons $\X$ affine.
Soit 
${E}$ un 
$\widehat{D}  ^{(m)} _{\X} $-module séparé et complet pour la topologie $p$-adique. 
On définit canoniquement un $\widehat{\D}  ^{(m)} _{\X}$-module en posant 
$\widehat{\D}  ^{(m)} _{\X}  \widehat{\otimes} _{\widehat{D}  ^{(m)} _{\X} } {E}:=
\underset{\underset{i}{\longleftarrow}}{\lim\,}  \D  ^{(m)} _{X _i}  \otimes _{\widehat{D}  ^{(m)} _{\X} } {E}$, i.e., le séparé complété de
$\widehat{\D}  ^{(m)} _{\X} \otimes_{\widehat{D}  ^{(m)} _{\X} } {E}$.
Comme $\widehat{\D}  ^{(m)} _{\X} \otimes_{\widehat{D}  ^{(m)} _{\X} } {E}$
est pré-pseudo-quasi-cohérent, il résulte du lemme \ref{completion-adjoint} que 
${\E}:= 
\widehat{\D}  ^{(m)} _{\X}  \widehat{\otimes} _{\widehat{D}  ^{(m)} _{\X} } {E}$
est
un $\widehat{\D}  ^{(m)} _{\X}$-module pseudo-quasi-cohérent. 

Comme $\Gamma( \X, \D  ^{(m)} _{X _i}  \otimes _{\widehat{D}  ^{(m)} _{\X} } {E} ) \riso  D  ^{(m)} _{X _i}  \otimes _{\widehat{D}  ^{(m)} _{\X} } {E}$,
comme le foncteur $\Gamma( \X,-)$ commute aux limites projectives, 
on obtient alors
$\Gamma( \X,{\E}) \riso {E}$.
\end{vide}

\begin{vide}
\label{ss-section-ThB-coro}
Soient ${\E}$ un $\widehat{\D}  ^{(m)} _{\X}$-module 
pseudo-quasi-cohérent
et $\U$ un ouvert affine de $\X$.
Il résulte de \cite[3.3.2.1]{Be1} que pour tout entier $n \geq 1$ on ait
l'annulation
$H ^n ( \U, {\E}) =0$.
En appliquant le foncteur $\Gamma (\U,-)$ à la suite exacte
$0\to {\E} \to {\E} \to {\E} _i \to 0$, 
on en déduit que le morphisme canonique
\begin{equation}
\label{ThB-coro}
D  ^{(m)} _{U _i} \otimes _{\widehat{D}  ^{(m)} _{\U} } \Gamma ( \U, {\E} )
\to 
\Gamma (\U, {\E} _i) 
\end{equation}
est un isomorphisme. 
On en déduit les trois points suivants: 

$\bullet$ Le morphisme canonique
$\widehat{\Gamma ( \U, {\E} )}
\to
\underset{\underset{i}{\longleftarrow}}{\lim\,} 
\Gamma (\U, {\E} _i)$
est un isomorphisme.
En appliquant le foncteur 
$\Gamma (\U, -)$ à \ref{def-qcmod}, 
on obtient
l'isomorphisme
$\Gamma (\U, {\E} )
\riso
\underset{\underset{i}{\longleftarrow}}{\lim\,} 
\Gamma (\U, {\E} _i)$.
Ainsi, $\Gamma (\U, {\E} )$
est séparé complet pour la topologie $p$-adique. 

$\bullet$ On déduit de \ref{ThB-coro} que le morphisme canonique
$\D  ^{(m)} _{U _i} \otimes _{\widehat{D}  ^{(m)} _{\U} } \Gamma ( \U, {\E} )
\to 
{\E} _i |\U$
est un isomorphisme.
En passant à la limite projective, on en déduit le morphisme canonique
$\widehat{\D}  ^{(m)} _{\U}  \widehat{\otimes} _{\widehat{D}  ^{(m)} _{\U} }
\Gamma ( \U, {\E} )
\riso
{\E}|\U$ (avec les notations de \ref{const-pseudo-qc}  pour le premier terme)
est un isomorphisme.

$\bullet$ Comme le système $(\Gamma (\U, {\E} _i)) _{i\in \N}$ satisfait la condition
de Mittag-Leffler, 
comme $H ^n ( \U, {\E}) =0$ pour tout $n \geq 1$,
il résulte alors de  \cite[7.20]{Berthelot-Ogus-cristalline} l'isomorphisme
canonique
$\underset{\underset{i}{\longleftarrow}}{\lim\,} 
\left ( 
\D  ^{(m)} _{X _i} \otimes _{\widehat{\D}  ^{(m)} _{\X} }{\E}
\right )
\riso 
\R \underset{\underset{i}{\longleftarrow}}{\lim\,} 
\left ( 
\D  ^{(m)} _{X _i} \otimes _{\widehat{\D}  ^{(m)} _{\X} }{\E}
\right ).$ 
\end{vide}

\begin{rema}
Soit ${\E}$ un $\widehat{\D}  ^{(m)} _{\X}$-module sans $p$-torsion et quasi-cohérent.
Il découle de troisième point de \ref{ss-section-ThB-coro}, 
l'isomorphisme canonique
$${\E} 
\riso 
\R \underset{\underset{i}{\longleftarrow}}{\lim\,} 
\left ( 
\D  ^{(m)} _{X _i} \otimes ^\L _{\widehat{\D}  ^{(m)} _{\X} }{\E}
\right ).$$
On obtient ${\E}  \in 
D ^\mathrm{b} _{\mathrm{qc}} (\widehat{\D}  ^{(m)} _{\X})$ (définition de Berthelot), ce qui justifie notre terminologie de quasi-cohérence sans $p$-torsion.

\end{rema}

\begin{vide}
On remarque que le fait qu'un $\widehat{\D}  ^{(m)} _{\X} $-module soit 
pseudo-quasi-cohérent
(resp. sans $p$-torsion et quasi-cohérent)
est local en $\X$. Avec \ref{const-pseudo-qc} et \ref{ss-section-ThB-coro},
on obtient alors la description locale suivante de ces modules: 
Supposons que $\X$ soit affine. 
Les foncteurs $\widehat{\D}  ^{(m)} _{\X}  \widehat{\otimes} _{\widehat{D}  ^{(m)} _{\X} } -$ 
et $\Gamma (\X, -)$ induisent des équivalences quasi-inverses entre la catégorie 
des $\widehat{D}  ^{(m)} _{\X} $-modules séparés et complets pour la topologie $p$-adique
(resp. sans $p$-torsion et séparés et complets)
et la catégorie des $\widehat{\D}  ^{(m)} _{\X} $-module 
pseudo-quasi-cohérents
(resp. sans $p$-torsion et quasi-cohérents).

\end{vide}

\begin{vide}
[Image directe par $u$ d'un $\D$-module]
\label{compar-u+-qc-mod}

Soit ${\FF}$ un 
$\widehat{\D}  ^{(m)} _{\ZZ}$-module sans $p$-torsion et quasi-cohérent (voir \ref{def-qcsstor}).
Comme ${\FF}$ est sans $p$-torsion, 
alors l'image directe de niveau $m$ de 
${\FF}$ par $u$ telle que définie par Berthelot  (voir \cite[3.5.1]{Beintro2}) est donnée par la formule 
$u ^{ (m)} _{+} (\smash{{\FF}} ):= 
\R \underset{\underset{i}{\longleftarrow}}{\lim\,} 
u  _{i+} \smash{{\FF_i}} ^{(m)}$.
Or, d'après les isomorphismes de changement de base (\ref{u+-chgt-base-niv}), pour tous entiers $i' \geq i$, on a l'isomorphisme
$\D  ^{(m)} _{X _i} \otimes _{\D  ^{(m)} _{X _{i'}}} u  _{i'+} (\smash{{\FF_{i'} }} ^{(m)}) \riso u  _{i+} (\smash{{\FF_i}} ^{(m)})$. D'après la version \cite[7.20]{Berthelot-Ogus-cristalline} de Mittag-Leffler, on en déduit le premier isomorphisme
$$u ^{ (m)} _{+} (\smash{{\FF}} )
\liso
\underset{\underset{i}{\longleftarrow}}{\lim\,} 
u  _{i+} \smash{{\FF_i}} ^{(m)}
\liso 
u _* \left (\widehat{\D}  ^{(m)} _{\X \hookleftarrow \ZZ} \widehat{\otimes} _{\widehat{\D}  ^{(m)} _{\ZZ}} {\FF} \right ),
$$
le dernier résultant du fait que $u _* $ commute aux limites projectives.

Supposons de plus que $\X$ possède des coordonnées locales $t _1, \dots, t _n$  telles que $\ZZ = V ( t _1, \dots, t _r)$. 
La sous-$\V$-algèbre de $D  ^{(m)} _{\X}$ engendrée par la famille d'éléments
$\{ \partial _1 ^{< k _1> _{(m)}}, \partial _2 ^{< k _2> _{(m)}}, \dots, \partial _r ^{< k _r> _{(m)}} \, |\, k _1, \dots, k _r \in \N\}$
sera notée 
$\V \{\partial _1,\dots, \partial _r \} ^{(m)}$. 
On note $D ^{(m)} _{\ZZ,\X}$ 
la sous-$O _\X$-algèbre
de $D ^{(m)} _{\X}$ engendrée par les éléments 
de la forme 
$\underline{\partial} ^{<(\underline{0},\underline{j})> _{(m)}}$,
avec 
$\underline{j} \in \N ^{n-r}$.
Comme pour \ref{tau(m)ZXS},
on construit le morphisme canonique
$\tau ^{(m)} _{\ZZ,\X/\V}\,:\, D ^{(m)} _{\X}
\to
\V \{\partial _1,\dots, \partial _r \} ^{(m)} \otimes _{\V} D ^{(m)} _{\ZZ}   $.
Comme pour \ref{tau-DZlin}, 
on vérifie qu'il induit l'isomorphisme d'anneaux 
$\overline{\tau} ^{(m)} _{\ZZ,\X/\V}\,:\, D ^{(m)} _{\ZZ,\X}/ID ^{(m)} _{\ZZ,\X}
\riso
D ^{(m)} _{\ZZ}   $
et l'isomorphisme $(D ^{(m)} _{\X},D ^{(m)} _{\ZZ})$-bilinéaire
$\overline{\tau} _{\ZZ,\X}\,:\,
 D ^{(m)} _{\X \hookleftarrow \ZZ} = D ^{(m)} _{\X}/D ^{(m)} _{\X} I
\riso
\V \{\partial _1,\dots, \partial _r \} ^{(m)} \otimes _{\V} D ^{(m)} _{\ZZ} $. 

Si aucune ambiguïté n'est à craindre (sur le niveau par exemple), 
on notera simplement $\widehat{\tau} _{\ZZ,\X}$
la complétion $p$-adique de $\overline{\tau} ^{(m)} _{\ZZ,\X/\V}$, i.e.
$\underset{\underset{i}{\longleftarrow}}{\lim\,} \overline{\tau} ^{(m)}_{Z_i,X_i/S _i}$, où $\S:= \Spf \V$ et $S _i$ sont les réductions modulo $\pi ^{i+1}$. 
On notera de même l'isomorphisme d'anneaux 
$\widehat{\tau} _{\ZZ,\X}\,:\, \widehat{D} ^{(m)} _{\ZZ,\X}/I\widehat{D} ^{(m)} _{\ZZ,\X}
\riso
\widehat{D} ^{(m)} _{\ZZ}   $ induit.

Les éléments de
$\widehat{D} ^{(m)} _{\X}$ s'écrivent de manière unique de la forme
$\sum _{\underline{i} \in \N ^{r}} 
\underline{\partial} ^{<(\underline{i},\underline{0})> _{(m)}}
P _{\underline{i}}
$,
avec 
$P _{\underline{i}} \in \widehat{D} ^{(m)} _{\ZZ,\X}$, les $P _{\underline{i}} $
convergeant vers $0$ pour la topologie $p$-adique.
L'isomorphisme $(\widehat{D} ^{(m)} _{\X},\widehat{D} ^{(m)} _{\ZZ})$-bilinéaire $\widehat{\tau} _{\ZZ,\X}$ 
est donné par la formule:
\begin{align}
\notag
\widehat{\tau} _{\ZZ,\X}\,:\,
 \widehat{D} ^{(m)} _{\X \hookleftarrow \ZZ} = \widehat{D} ^{(m)} _{\X}/\widehat{D} ^{(m)} _{\X} I
&
\riso
\V\{\partial _1,\dots, \partial _r \} ^{(m)} \widehat{\otimes} _{\V} \widehat{D} ^{(m)} _{\ZZ}  
\\
\label{widehatD_(X<-Z)}
[\sum _{\underline{i} \in \N ^{r}} 
\underline{\partial} ^{<(\underline{i},\underline{0})> _{(m)}} P _{\underline{i}}] ^\prime _Z&
\mapsto
\sum _{\underline{i} \in \N ^{r}} 
\underline{\partial} ^{<(\underline{i})> _{(m)}}
\otimes
\widehat{\tau} _{\ZZ,\X} ([P _{\underline{i}}]' _{\ZZ} ),
\end{align}
où 
$[-]' _{\ZZ}$ désigne le morphisme canonique
$\widehat{D} ^{(m)} _{\X}
\to 
\widehat{D} ^{(m)} _{\X}/\widehat{D} ^{(m)} _{\X} I$.
Il en résulte l'isomorphisme
\begin{align}
\label{u+-comm-Gamma}
\Gamma (\X, u ^{(m)} _+ ( {\FF}))
\riso
\widehat{D}  ^{(m)} _{\X \hookleftarrow \ZZ} \widehat{\otimes} _{\widehat{D}  ^{(m)} _{\ZZ}} {F}
\underset{\widehat{\tau} _{\ZZ,\X}}{\riso}
\V \{\partial _1,\dots, \partial _r \} ^{(m)} \widehat{\otimes} _\V  {F}.
\end{align}
\end{vide}
Par stabilité de la quasi-cohérence par image directe par $u$ (voir \cite{Beintro2}), 
on en déduit que $u ^{ (m)} _{+} (\smash{{\FF}} )$  est  un 
$\widehat{\D}  ^{(m)} _{\X}$-module sans $p$-torsion et quasi-cohérent.

\begin{rema}
\label{pas-stab-qc}
Soit  ${\E}$ un $\widehat{\D}  ^{(m)} _{\X}$-module sans $p$-torsion et quasi-cohérent.
Comme ${\E}  \in 
D ^\mathrm{b} _{\mathrm{qc}} (\widehat{\D}  ^{(m)} _{\X})$, par stabilité de la quasi-cohérence (voir \cite{Beintro2}),
on obtient 
$u ^{!} ( {\E}) \in D ^\mathrm{b} _{\mathrm{qc}} (\widehat{\D}  ^{(m)} _{\ZZ})$.
On remarque que 
$\mathcal{H} ^0 u ^{!} ( {\E}) $ est sans $p$-torsion
et donc
$(\mathcal{H} ^0 u ^{!} ( {\E}) )_i := 
\D ^{(m)} _{Z _i} \otimes ^\L  _{\widehat{\D}  ^{(m)} _{\ZZ}}\mathcal{H} ^0 u ^{!} ( {\E})
\liso \D ^{(m)} _{Z _i} \otimes  _{\widehat{\D}  ^{(m)} _{\ZZ}}\mathcal{H} ^0 u ^{!} ( {\E})$.
Par contre, il paraît faux que 
$(\mathcal{H} ^0 u ^{!} ( {\E}) )_i$ soit un $\D ^{(m)} _{Z _i} $-module quasi-cohérent. 
Ainsi, on a a priori
$\mathcal{H} ^0 u ^{!} ( {\E}) \not \in D ^\mathrm{b} _{\mathrm{qc}} (\widehat{\D}  ^{(m)} _{\ZZ})$.
Par contre,
on vérifie par un calcul que le morphisme canonique
\begin{equation}
\label{H0u!-limproj}
\mathcal{H} ^0 u ^{!} ( {\E}) 
\to
\underset{\underset{i}{\longleftarrow}}{\lim\,} 
\mathcal{H} ^0 u _i ^! ({\E} _i) 
\end{equation}
est un isomorphisme ( 
cela résulte aussi de l'isomorphisme
$u ^{!}  ({\E})
\riso
\R \underset{\underset{i}{\longleftarrow}}{\lim\,} 
u _i ^! ({\E} _i) $,
auquel on applique le foncteur 
$\mathcal{H} ^0 $).
Comme 
$u _i ^{!} ( {\E} _i )
\riso 
\D ^{(m)} _{Z _i} \otimes ^\L  _{\widehat{\D}  ^{(m)} _{\ZZ}}u ^{!} ( {\E})$,
on obtient l'isomorphisme
$\mathcal{H} ^0  u _i ^{!} ( {\E} _i )
\riso 
\mathcal{H} ^0  (\D ^{(m)} _{Z _i} \otimes ^\L  _{\widehat{\D}  ^{(m)} _{\ZZ}}u ^{!} ( {\E}))$. 
En particulier, les $\D ^{(m)} _{Z _i} $-modules apparaissant dans la limite projective 
\ref{H0u!-limproj} sont quasi-cohérents. 
\end{rema}

\subsection{Images directes et images inverses extraordinaires de $D$-modules par une immersion fermée}
\label{section2.2}

Même si cela n'est pas toujours nécessaire, supposons dans cette section que $\X$ est affine et possède des coordonnées locales $t _1, \dots, t _n$ 
telles que $\ZZ = V ( t _1, \dots, t _r)$. 
Comme cela est expliqué au début de la section précédente, nous travaillerons ici 
dans les catégories de $\widehat{D}  ^{(m)} _{\X}$-modules (ou $\widehat{D}  ^{(m)} _{\ZZ}$-modules) sans $p$-torsion, 
séparés et complets pour la topologie $p$-adique. 
Nous définirons sur ces catégories la notion d'image directe (resp. d'image inverse extraordinaire)
de niveau $m$ par $u$ notée $u _+ ^{(m)}$ et 
(resp. $\mathcal{H} ^0 u ^{!}$).

\begin{vide}
[Image directe par $u$ d'un $\widehat{D}  ^{(m)} _{\ZZ}$-module]
\label{desc-imag-direct}
Soit $\smash{{F}} $ un $\widehat{D}  ^{(m)} _{\ZZ}$-module sans $p$-torsion, 
séparé et complet pour la topologie $p$-adique. 
On définit l'image directe de ${F}$ par $u$ de niveau $m$
en posant
$$u ^{ (m)} _{+} (\smash{{F}}):=
\underset{\underset{i}{\longleftarrow}}{\lim\,} 
u  _{i+}( {F _i}) 
=
\widehat{D}  ^{(m)} _{\X \hookleftarrow \ZZ} \widehat{\otimes} _{\widehat{D}  ^{(m)} _{\ZZ}} {F}
\underset{\widehat{\tau} _{\ZZ,\X}}{\riso}
\V \{\partial _1,\dots, \partial _r \} ^{(m)} \widehat{\otimes} _\V  {F}.
$$
La formule \ref{u+-comm-Gamma} justifie la notation.
On obtient ainsi
un $\widehat{D}  ^{(m)} _{\X}$-module sans $p$-torsion, 
séparé et complet pour la topologie $p$-adique. 
Via cet isomorphisme, les éléments de 
$u ^{(m)} _{+}  (\smash{{F}} )$
s'écrivent de manière unique sous la forme
$\sum _{\underline{k}\in \N ^{r}} \underline{\partial} ^{<\underline{k}>} \otimes x _{\underline{k}}$,
avec $x _{\underline{k}} \in\smash{{F}} $ 
et $\lim _{|\underline{k} |\to \infty} x _{\underline{k}} =0$ (pour la topologie $p$-adique de $\smash{{F}}$). 
De plus, les éléments de 
$\left(u ^{(m)} _{+}  (\smash{{F}} ) \right )_\Q$
s'écrivent de manière unique sous la forme
$\sum _{\underline{k}\in \N ^{r}} \underline{\partial} ^{<\underline{k}>} \otimes x _{\underline{k}}$,
avec $x _{\underline{k}} \in\smash{{F}} _\Q$ 
et $\lim _{|\underline{k} |\to \infty} x _{\underline{k}} =0$ (pour la topologie de $\smash{{F}} _\Q$ induite par la topologie $p$-adique de
$\smash{{F}} $). Pour ne pas alourdir les notations nous avons noter 
$\underline{\partial} ^{<\underline{k}>}$ la classe de 
$\underline{\partial} ^{<\underline{k}>}$ dans
$\widehat{D} ^{(m)} _{\X\leftarrow \ZZ}$. 

\end{vide}

\begin{lemm}
\label{inclusion-u+}
Soit ${F} \hookrightarrow {G} $
un monomorphisme de $\widehat{D} ^{(m)} _{\ZZ}$-modules sans $p$-torsion, 
séparés et complets pour la topologie $p$-adique. 
\begin{enumerate}
\item 
Le morphisme 
$u ^{(m)} _{+}  ( {F} )
\to 
u ^{(m)} _{+}  (  {G})$
(resp. $ (u ^{(m)} _{+}  ( {F} )) _\Q
\to 
(u ^{(m)} _{+}  ( {G} )) _\Q$)
est alors un monomorphisme. 

\item 
Le morphisme canonique
$u ^{(m)} _{+}  ( {F} )
\to 
u ^{(m)} _{+}  ( {G} )$
(resp. $ (u ^{(m)} _{+}  ( {F} )) _\Q
\to 
(u ^{(m)} _{+}  ( {G} )) _\Q$)
 est alors un isomorphisme si et seulement si
le morphisme canonique 
${F} \hookrightarrow {G} $
(resp. ${F} _\Q\hookrightarrow {G} _\Q$)
est un isomorphisme.
\end{enumerate}
\end{lemm}

\begin{proof}
Cela découle des descriptions de \ref{desc-imag-direct} (ou alors de \ref{prop-idtou+u!isohat}).
\end{proof}

\begin{lemm}
\label{coro-inclusion-u+}
Soit ${F}$
un $\widehat{D} ^{(m)} _{\ZZ}$-module sans $p$-torsion, 
séparé et complet pour la topologie $p$-adique. 
Alors ${F}$ est 
un $\widehat{D} ^{(m)} _{\ZZ}$-module cohérent si et seulement si 
$u ^{(m)} _{+}  ( {F} )$ est 
un $\widehat{D} ^{(m)} _{\X}$-module cohérent.
\end{lemm}

\begin{proof}
Comme le foncteur $u ^{(m)} _{+} $ défini par Berthelot pour les faisceaux 
préserve la cohérence, comme on dispose des théorèmes de type $A$, 
l'isomorphisme \ref{u+-comm-Gamma} nous permet de conclure la nécessité de l'assertion du lemme. 
Réciproquement, supposons que ${F}$ n'est pas
un $\widehat{D} ^{(m)} _{\ZZ}$-module cohérent.
Il existe alors une suite strictement croissante $({F_n}) _{n\in\N}$
de sous-$\widehat{D} ^{(m)} _{\ZZ}$-module cohérent de 
${F}$. Il résulte alors du lemme \ref{inclusion-u+} que l'on obtient la
suite strictement croissante $(u ^{(m)} _{+}  ({F_n})) _{n\in\N}$
de sous-$\widehat{D} ^{(m)} _{\X}$-module cohérent
de $u ^{(m)} _{+}  ( {F} )$. 
Comme $\widehat{D} ^{(m)} _{\X}$ est noethérien, 
ce dernier n'est pas
un $\widehat{D} ^{(m)} _{\X}$-module cohérent. 
\end{proof}

\begin{rema}
Soit $H ^{(m)}$ un $\widehat{D}  ^{(m)} _{\ZZ,\Q}$-module tel qu'il existe un 
$\widehat{D}  ^{(m)} _{\ZZ}$-module 
$\smash{\overset{_\circ}{H}} ^{(m)}$ 
sans $p$-torsion, séparé et complet (pour la topologie $p$-adique), 
tel que 
$\smash{\overset{_\circ}{H}} ^{(m)} _{\Q} \riso H ^{(m)}.$ 
On remarque que la topologie sur $H ^{(m)}$ induite par la base de voisinage 
de $p ^{n}\smash{\overset{_\circ}{H}} ^{(m)}$ 
fait de $H ^{(m)}$ un $K$-espace de Banach. 

En outre, si $H ^{(m)}$ est un $\widehat{D}  ^{(m)} _{\ZZ,\Q}$-module cohérent, alors
cette topologie est indépendant du choix d'un tel 
$\widehat{D}  ^{(m)} _{\ZZ}$-module 
$\smash{\overset{_\circ}{H}} ^{(m)}$ et ce dernier est forcément $\widehat{D}  ^{(m)} _{\ZZ}$-cohérent.
En effet, cela résulte du théorème de Banach (e.g. voir \cite[2.8.1]{bosch})
et de la  noethérianité de $\widehat{D}  ^{(m)} _{\ZZ}$.

On en déduit que si $H ^{(m)}$ est un $\widehat{D}  ^{(m)} _{\ZZ,\Q}$-module cohérent alors 
$\left(u ^{(m)} _{+}  (\smash{\overset{_\circ}{H}} ^{(m)} ) \right )_\Q$
 est un $\widehat{D}  ^{(m)} _{\X,\Q}$-module cohérent et 
ne dépend pas du 
$\widehat{D}  ^{(m)} _{\ZZ}$-module séparé complet $\smash{\overset{_\circ}{H}} ^{(m)} $ sans $p$-torsion, 
tel que 
$\smash{\overset{_\circ}{H}} ^{(m)} _{\Q} \riso H  ^{(m)}.$ 
On le note alors $u ^{(m)} _{+}  (H  ^{(m)})$.
\end{rema}

\begin{vide}
[Image inverse extraordinaire par $u$ d'un $\widehat{D}  ^{(m)} _{\X}$-module]
\label{im-inv-D-mod}
Soient  ${\E}$ un $\widehat{\D}  ^{(m)} _{\X}$-module sans $p$-torsion et quasi-cohérent
et 
$\smash{{E}} := \Gamma ( \X, {\E})$ 
le $\widehat{D}  ^{(m)} _{\X}$-module sans $p$-torsion, 
séparé et complet pour la topologie $p$-adique correspondant. 
On pose ${\E} _i := \D ^{(m)} _{X _i} \otimes _{\widehat{\D}  ^{(m)} _{\X}} {\E}
\liso \D ^{(m)} _{X _i} \otimes ^\L  _{\widehat{\D}  ^{(m)} _{\X}} {\E}$
et 
${E} _i := D ^{(m)} _{X _i} \otimes _{\widehat{D}  ^{(m)} _{\X}} {E}
\liso 
D ^{(m)} _{X _i} \otimes ^\L  _{\widehat{D}  ^{(m)} _{\X}} {E}$.
On rappelle la notation 
$\mathcal{H} ^0 u _i ^{!} ( {E}  _i)= \cap _{s=1} ^r \ker ( {E} _i\overset{t _s}{\longrightarrow} {E}  _i)$
(voir \ref{vide-inv!-nota}).
On définit l'image inverse extraordinaire de ${E}$ par $u$
en posant
\begin{equation}
\label{H0u!-limproj-mod}
\mathcal{H} ^0 u ^{!} ( {E}) 
:=
\underset{\underset{i}{\longleftarrow}}{\lim\,} 
\mathcal{H} ^0 u _i ^! ({E} _i) 
\subset 
\underset{\underset{i}{\longleftarrow}}{\lim\,} 
{E} _i ={E}.
\end{equation}
Il découle de \ref{H0u!-limproj} que l'on a 
l'isomorphisme 
$\Gamma (\ZZ, \mathcal{H} ^0 u ^{!} ( {\E}) )
=
\mathcal{H} ^0 u ^{!} ( {E})$, ce qui justifie la notation.

On calcule
$\mathcal{H} ^0 u ^{!} ( {E})= \cap _{s=1} ^r \ker ( {E} \overset{t _s}{\longrightarrow} {E})$.
On munit 
$\mathcal{H} ^0 u ^{!} ({E})$
d'une structure canonique de $\widehat{D} ^{(m)} _{\ZZ} $-module à gauche via l'isomorphisme
$\widehat{\tau} _{\ZZ,\X}\,:\, \widehat{D} ^{(m)} _{\ZZ,\X}/\widehat{D} ^{(m)} _{\ZZ,\X} I
\riso
\widehat{D} ^{(m)} _{\ZZ}   $ de la manière suivante: 
pour tout $P \in\widehat{D} ^{(m)} _{\ZZ,\X}$, pour tout $x \in \mathcal{H} ^0 u ^{!} ( {E})$, 
on pose 
\begin{equation}
\label{def-D-modu!hat}
\widehat{\tau} _{\ZZ,\X} ([P ] ' _{\ZZ})\cdot x : = P\cdot x,
\end{equation}
où $[-]' _{\ZZ}$ désigne le morphisme canonique
$\widehat{D} ^{(m)} _{\ZZ,\X}
\to 
\widehat{D} ^{(m)} _{\ZZ,\X}/\widehat{D} ^{(m)} _{\ZZ,\X} I$.

Comme pour $s =1, \dots, r$ les applications 
$t _s \,:\, {E} \to {E}$
sont continues, 
comme 
$\mathcal{H} ^0 u ^{!} ( {E})= \cap _{s=1} ^r \ker ( {E} \overset{t _s}{\longrightarrow} {E})$,
cela entraîne que
$\mathcal{H} ^0 u ^{!} ( {E})$ est séparé complet pour la topologie induite par la topologie $p$-adique 
de $ {E}$.
Comme $ {E}$ est sans $p$-torsion, on déduit aussi
du calcul 
$\mathcal{H} ^0 u ^{!} ( {E})= \cap _{s=1} ^r \ker ( {E} \overset{t _s}{\longrightarrow} {E})$
l'égalité
$\mathcal{H} ^0 u ^{!} ( {E})  \cap p ^{i+1} {E} = 
p ^{i+1} \mathcal{H} ^0 u ^{!} ( {E}) $. 
En particulier, la topologie $p$-adique de $\mathcal{H} ^0 u ^{!} ( {E}) $
et la topologie de $\mathcal{H} ^0 u ^{!} ( {E}) $ induite par la topologie $p$-adique 
de $ {E}$ sont identiques.
Il en résulte que 
$\mathcal{H} ^0 u ^{!} ( {E}) $ 
est un $\widehat{D}  ^{(m)} _{\ZZ}$-module sans $p$-torsion, 
séparé et complet pour la topologie $p$-adique.

Terminons par une remarque : lorsque $\mathcal{H} ^0 u ^{!} ( {\E})$ est de surcroît quasi-cohérent,
$\mathcal{H} ^0 u ^{!} ( {E})$ est alors 
le $\widehat{D}  ^{(m)} _{\ZZ}$-module associé à 
$\mathcal{H} ^0 u ^{!} ( {\E})$.

\end{vide}

\begin{rema}
\label{u+limproj-pre}
Soit $\smash{{E}} $ un $\widehat{D}  ^{(m)} _{\X}$-module sans $p$-torsion, 
séparé et complet pour la topologie $p$-adique. 
En posant $\left (\mathcal{H} ^0 u ^{!} ( {E})  \right ) _i:=
D ^{(m)} _{Z _i} \otimes _{\widehat{D}  ^{(m)} _{\ZZ}} 
\mathcal{H} ^0 u ^{!} ( {E})$,
comme 
$\mathcal{H} ^0 u ^{!} ( {E})  \cap p ^{i+1} {E} = 
p ^{i+1} \mathcal{H} ^0 u ^{!} ( {E}) $,
on obtient l'inclusion
\begin{equation}
\label{u!isubsetui!}
\left (\mathcal{H} ^0 u ^{!} ( {E})  \right ) _i
\subset 
\mathcal{H} ^0 u _i ^! ({E} _i).
\end{equation}
En prenant la structure de $D ^{(m)} _{Z _i}$-module sur 
$\left (\mathcal{H} ^0 u ^{!} ( {E})  \right ) _i$ induite par la structure
de $\widehat{D}  ^{(m)} _{\ZZ}$-module de $\mathcal{H} ^0 u ^{!} ( {E})$ définie par la formule \ref{def-D-modu!hat}
et la structure de $D ^{(m)} _{Z _i}$-module de $\mathcal{H} ^0 u _i ^! ({E} _i)$
donnée par l'égalité \ref{def-D-modu!},
on vérifie que 
l'inclusion \ref{u!isubsetui!} est $D ^{(m)} _{Z _i}$-linéaire.
En général, celle-ci n'est pas bijective.  

\end{rema}

\begin{rema}
\label{u+limproj}
Avec les notations de \ref{u+limproj-pre},
le morphisme canonique 
$ u _+ \mathcal{H} ^0 u ^{!} ( {E})
=
\underset{\underset{i}{\longleftarrow}}{\lim\,} 
u  _{i+}\left (\mathcal{H} ^0 u ^{!} ( {E})  \right ) _i
\to 
\underset{\underset{i}{\longleftarrow}}{\lim\,} 
 u _{i+} \mathcal{H} ^0 u _i ^{!} ( {E} _i)$ est un isomorphisme, i.e. 
 que le morphisme canonique 
$$\V \{\partial _1,\dots, \partial _r \} ^{(m)} \widehat{\otimes} _\V \mathcal{H} ^0 u ^{!} ( {E})
\to 
\underset{\underset{i}{\longleftarrow}}{\lim\,} 
(O _{S_i} \{\partial _1,\dots, \partial _r \} ^{(m)} \otimes _{O _{S_i}}
 \mathcal{H} ^{0}   u  _i ^! (E _i))$$
 est un isomorphisme.
En effet, 
les éléments $x \in \V \{\partial _1,\dots, \partial _r \} ^{(m)} \widehat{\otimes} _\V \mathcal{H} ^0 u ^{!} ( {E})$
s'écrivent de manière unique sous la forme
$x = \sum _{\underline{l} \in \N ^{r}}  \underline{\partial} ^{<\underline{l}>} \otimes x _{\underline{l}}$,
avec $x _{\underline{l}} \in \mathcal{H} ^0 u ^{!} ( {E})$ (et convergeant vers $0$).
Soit $(x _i) _i \in \underset{\underset{i}{\longleftarrow}}{\lim\,} 
(O _{S_i} \{\partial _1,\dots, \partial _r \} ^{(m)} \otimes _{O _{S_i}}
 \mathcal{H} ^{0}   u  _i ^! (E _i))$.
On écrit de manière unique $x _i$ sous la forme
$x _i= \sum _{\underline{l} \in \N ^{r}}  \underline{\partial} ^{<\underline{l}>}  \otimes  x _{\underline{l},i}$,
avec $x _{\underline{l},i} \in \mathcal{H} ^0 u ^{!} ( {E}_i)$ et la somme étant cette fois-ci finie.
Par unicité, on obtient que $(x _{\underline{l},i}) _i \in \underset{\underset{i}{\longleftarrow}}{\lim\,} 
\mathcal{H} ^0 u _i ^! ({E} _i)$.
L'égalité
$\mathcal{H} ^0 u^! ({E})=
\underset{\underset{i}{\longleftarrow}}{\lim\,} \mathcal{H} ^0 u _i ^! ({E} _i)$ 
nous permet alors de conclure.
\end{rema}

\subsection{Adjonction}
Nous reprenons les notations et hypothèses de la section \ref{section2.2}.
Nous vérifions que le foncteur
$u _+ ^{(m)}$ est adjoint à gauche de
$\mathcal{H} ^0 u ^{!}$.

\begin{prop}
[Morphisme d'adjonction I]
\label{prop-idtou+u!isohat}
Soit $\smash{{F}} ^{(m)}$ un $\widehat{D}  ^{(m)} _{\ZZ}$-module sans $p$-torsion, 
séparé et complet pour la topologie $p$-adique. 
On définit le morphisme canonique d'adjonction
\begin{gather}
\label{idtou+u!isohat}
\mathrm{adj}\,:\,
\smash{{F}} ^{(m)}
\to
\mathcal{H} ^{0} u ^{ !} \left (u ^{(m)} _{+}  (\smash{{F}} ^{(m)})  \right) ,
\end{gather}
en posant
$\mathrm{adj} (x )= 1 \otimes x$ (i.e. $\mathrm{adj} (x )=  \sum _{\underline{k}\in \N ^{r}} \underline{\partial} ^{<\underline{k}>} \otimes x _{\underline{k}}$ 
avec $x _{\underline{k}} =x$ si $\underline{k}=0$ et $x _{\underline{k}} =\underline{0}$ sinon),
 pour tout $x \in \smash{{F}} ^{(m)}$.
Le morphisme $\mathrm{adj}$ est 
 un isomorphisme de $\widehat{D}  ^{(m)} _{\ZZ}$-modules. 
\end{prop}

\begin{proof}
Établissons d'abord sa $\widehat{D}  ^{(m)} _{\ZZ}$-linéarité.
Pour tout $\widehat{D}  ^{(m)} _{\ZZ}$-module à gauche
$\smash{{G}}$, on note 
$[-] _i$ le morphisme canonique 
$\smash{{G}} \to \smash{{G}}/ \pi ^{i+1} \smash{{G}}$.
On dispose du diagramme commutatif: 
\begin{equation}
\label{diag1-idtou+u!isohat}
\xymatrix{
{\smash{{F}} ^{(m)}} 
\ar[r] ^-{\mathrm{adj}}
\ar[d] ^-{[-] _i}
& 
{\mathcal{H} ^{0} u ^{ !} \circ u ^{(m)} _{+}  (\smash{{F}} ^{(m)})  } 
\ar[d] ^-{[-] _i}
\\ 
{\smash{{F_i}} ^{(m)}} 
\ar[r] 	^-{\mathrm{adj}}
& 
{\left ( \mathcal{H} ^{0} u ^{ !} \circ u ^{(m)} _{+}  (\smash{{F}} ^{(m)}) \right ) _i } 
\ar@{^{(}->}[r] ^-{}
& 
{\mathcal{H} ^{0} u _i ^{ !} \circ u ^{(m)} _{i +}  (\smash{{F_i}} ^{(m)}) , } 
}
\end{equation}
dont la flèche en bas à droite découle de \ref{u!isubsetui!}.
On calcule que le morphisme composé du bas est le morphisme 
$\mathrm{adj}$ de \ref{idtou+u!iso} qui est  $D ^{(m)} _{Z _i}$-linéaire. 
Comme l'inclusion  \ref{u!isubsetui!} est $D ^{(m)} _{Z _i}$-linéaire, il en résulte que la flèche en bas
à gauche de \ref{diag1-idtou+u!isohat} est aussi $D ^{(m)} _{Z _i}$-linéaire.
Par passage à la limite projective, il en résulte que le morphisme  du haut de
\ref{diag1-idtou+u!isohat} est $\widehat{D}  ^{(m)} _{\ZZ}$-linéaire.

Vérifions à présent sa bijectivité. L'injectivité est évidente. 
Or, 
via la description \ref{desc-imag-direct} de l'image directe et avec ses notations, 
pour $1\leq i\leq r$, on obtient dans $u ^{(m)} _{+}  (\smash{{F}} ^{(m)}) $
la relation
$$t _i \cdot \sum _{\underline{k}\in \N ^{r}} (\underline{\partial} ^{<\underline{k}>} \otimes x _{\underline{k}})
= -\sum _{\underline{k} \geq 1 _i} \left\{
\begin{smallmatrix}
k _i   \\
1   
\end{smallmatrix}
\right \}
\underline{\partial} ^{<\underline{k}-1 _i>} \otimes x _{\underline{k}},$$ 
où
$1 _i= (0, \dots, 0,1,0, \dots, 0)$ avec $1$ à la $i$ème place (en effet, il suffit de vérifier l'égalité modulo $\pi ^{i+1}$, ce qui découle
des formules de \cite[2.2.4]{Be1} et \ref{D_(X<-Z)}).
On en déduit aussitôt que tous les éléments de 
$\mathcal{H} ^{0} u ^{ !} \circ u ^{(m)} _{+}  (\smash{{F}} ^{(m)})  $
sont de la forme 
$1\otimes x$ avec $x \in \smash{{F}} ^{(m)}$.
\end{proof}

\begin{prop}
[Morphisme d'adjonction II]
\label{prop-u_+u^!toif-limproj}
Soit $\smash{{E}} ^{(m)}$
un $\widehat{D} ^{(m)} _{\X}$-module sans $p$-torsion, séparé et complet pour la topologie $p$-adique. 
Considérons la flèche
\begin{equation}
\label{u_+u^!toif-limproj}
\mathrm{adj}\,:\,u ^{(m)} _{+} \circ \mathcal{H} ^{0} u ^{ !} (\smash{{E}} ^{(m)})
\to 
\smash{{E}} ^{(m)}
\end{equation}
donnée par 
$\sum _{\underline{k}\in \N ^{r}} \underline{\partial} ^{<\underline{k}>} \otimes x _{\underline{k}}
\to 
\sum _{\underline{k}\in \N ^{r}} \underline{\partial} ^{<\underline{k}>}  \cdot x _{\underline{k}}$,
avec $x _{\underline{k}} \in \mathcal{H} ^{0} u ^{ !} (\smash{{E}} ^{(m)})$ 
et $\lim _{|\underline{k} |\to \infty} x _{\underline{k}} =0$. 
L'application $\mathrm{adj}$ est 
 un morphisme de $\widehat{D}  ^{(m)} _{\X}$-modules. 
\end{prop}

\begin{proof}
Considérons le diagramme canonique suivant
\begin{equation}
\label{diag1-u_+u^!toif-limproj}
\xymatrix{
{u ^{(m)} _{+} \circ \mathcal{H} ^{0} u ^{ !} (\smash{{E}} ^{(m)})} 
\ar[rr] ^-{\mathrm{adj}} _-{\ref{u_+u^!toif-limproj}}
\ar[d] ^-{[-] _i}
& &
{\smash{{E}} ^{(m)} } 
\ar[d] ^-{[-] _i}
\\ 
{u ^{(m)} _{i+} \left ( \mathcal{H} ^{0} u ^{ !} (\smash{{E}} ^{(m)}) \right) _i  } 
\ar@{^{(}->}[r] ^-{}
& 
{u ^{(m)} _{i+} \circ \mathcal{H} ^{0} u _i ^{ !} (\smash{{E _i}} ^{(m)}) } 
\ar[r] ^-{\mathrm{adj}} _-{\ref{u_+u^!toif}}
& 
{\smash{{E_i}} ^{(m)} } 
}
\end{equation}
dont la flèche en bas à gauche se construit via l'inclusion \ref{u!isubsetui!}.
Avec \ref{prop-u_+u^!toif}, le morphisme composé du bas de \ref{diag1-u_+u^!toif-limproj}
est $D ^{(m)} _{X _i}$-linéaire. 
Or, on calcule que les deux chemins possible du diagramme 
envoie un élément de la forme 
$\sum _{\underline{k}\in \N ^{r}} \underline{\partial} ^{<\underline{k}>} \otimes x _{\underline{k}}$ 
sur 
$\sum _{\underline{k}\in \N ^{r}} \underline{\partial} ^{<\underline{k}>} \cdot [x _{\underline{k}}] _i$.
D'où la commutativité de \ref{diag1-u_+u^!toif-limproj}.
Ainsi la réduction modulo $\pi ^{i+1}$ du morphisme $\mathrm{adj}$
de \ref{u_+u^!toif-limproj} 
est $D ^{(m)} _{X _i}$-linéaire. 
On obtient le résultat par passage à la limite projective. 
\end{proof}

\begin{vide}
Soient $\smash{{E}}$ (resp. $\smash{{F}}$)
un $\widehat{D} ^{(m)} _{\X}$-module (resp. $\widehat{D} ^{(m)} _{\ZZ}$-module)
sans $p$-torsion, séparé et complet pour la topologie $p$-adique. 
Via le morphisme d'adjonction 
\ref{idtou+u!isohat}, le foncteur 
$\mathcal{H} ^{0} u ^{ !} $
induit l'application
\begin{equation}
\label{Hom-adj-I}
\mathrm{Hom} _{\widehat{D} ^{(m)} _{\X}}
(u ^{(m)} _{+}  (\smash{{F}} ), 
\smash{{E}} ) 
\to 
\mathrm{Hom} _{\widehat{D} ^{(m)} _{\ZZ}}
(\smash{{F}}, 
\mathcal{H} ^{0} u ^{ !}  \smash{{E}} ).
\end{equation}

Via le morphisme d'adjonction \ref{u_+u^!toif-limproj},
le foncteur 
$u ^{(m) } _+$
induit l'application
\begin{equation}
\label{Hom-adj-II}
\mathrm{Hom} _{\widehat{D} ^{(m)} _{\ZZ}}
(\smash{{F}}, 
\mathcal{H} ^{0} u ^{ !}  \smash{{E}} )
\to
\mathrm{Hom} _{\widehat{D} ^{(m)} _{\X}}
(u ^{(m)} _{+}  (\smash{{F}} ), 
\smash{{E}} ) .
\end{equation}

On calcule que le morphisme composé
\begin{equation}
\label{comp-adj=id1}
u ^{(m)} _{+}  (\smash{{F}} )
\overset{\ref{idtou+u!isohat}}{\longrightarrow}
u ^{(m)} _{+}   \circ \mathcal{H} ^{0} u ^{ !}  \circ u ^{(m)} _{+}  (\smash{{F}} )
\overset{\ref{u_+u^!toif-limproj}}{\longrightarrow}
u ^{(m)} _{+}  (\smash{{F}} ),
\end{equation}
dont la première flèche est l'image par $u ^{(m)} _{+}$
de \ref{idtou+u!isohat} et dont la deuxième
est \ref{u_+u^!toif-limproj} utilisé pour $\smash{{E}}= u ^{(m)} _{+}  (\smash{{F}} )$,
est l'identité. Comme la première flèche est un isomorphisme, la seconde aussi. 
Par un calcul, on vérifie de plus que le morphisme composé
$$\mathcal{H} ^{0} u ^{ !}  (\smash{{E}} )
\overset{\ref{u_+u^!toif-limproj}}{\longrightarrow}
\mathcal{H} ^{0} u ^{ !}  \circ u ^{(m)} _{+} \circ \mathcal{H} ^{0} u ^{ !}  (\smash{{E}} )
\overset{\ref{idtou+u!isohat}}{\longrightarrow}
\mathcal{H} ^{0} u ^{ !}   (\smash{{E}} )$$
est l'identité. Il en résulte que les applications \ref{Hom-adj-I} et \ref{Hom-adj-II} ci-dessus sont réciproques l'une de l'autre.

\end{vide}

\begin{rema}
Par manque de stabilité de la catégorie des faisceaux quasi-cohérents mise en exergue dans la section \ref{section2.1},
on ne dispose pas a priori d'un analogue dans ce contexte. 
\end{rema}

\section{La surcohérence implique l'holonomie}

\subsection{Compléments sur l'holonomie sans structure de Frobenius}
On suppose pour simplifier que $\X$ est intègre. On désigne par $m$ un entier positif. 
\begin{nota}
\label{Nota-suite-spectrale}

Soit $\E$ un $\smash{\widehat{\D}} ^{(m)} _{\X,\,\Q}$-module cohérent. 
Notons $(F ^{i}  _{(m)}(\E) ) _{i \geq 0}$ la filtration décroissante définie par la suite spectrale 
\begin{equation}
\label{StSp}
E _2 ^{r,s}=\mathcal{E} xt ^{r}  _{\smash{\widehat{\D}} ^{(m)} _{\X,\,\Q}} (\mathcal{E} xt ^{-s}  _{\smash{\widehat{\D}} ^{(m)} _{\X,\,\Q}} ( \E,\smash{\widehat{\D}} ^{(m)} _{\X,\,\Q}),\smash{\widehat{\D}} ^{(m)} _{\X,\,\Q})
\Rightarrow 
\mathcal{H} ^{n}  ( \E)
\end{equation}
pour $n=0$, i.e.,  $F ^{i}  _{(m)}(\E) /F ^{i+1}  _{(m)}(\E) = E _{\infty} ^{i,-i}$.
Lorsque le niveau ne fait pas de doute, on écrira $F ^{i}  (\E)$ au lieu de $F ^{i}  _{(m)}(\E)$.
Il découle aussitôt de \cite[1.3]{caro-holo-sansFrob}  et avec ses notations que, 
pour tout $0 \leq i<  \mathrm{codim} ^{(m)} (\E )$,
on obtient $E _2 ^{i,i}=0$.
Il en résulte,
pour tout $0 \leq i\leq  \mathrm{codim} ^{(m)} (\E )$, l'égalité $F ^{i} (\E) =\E$.

\end{nota}

\begin{lemm}
\label{formule-codim}
Soit $0\to \E' \underset{f}{\longrightarrow} \E \underset{g}{\longrightarrow} \E'' \to 0$ une suite exacte de $\smash{\widehat{\D}} ^{(m)} _{\X,\,\Q}$-modules cohérents.
Alors, on dispose de la formule 
$\mathrm{codim} ^{(m)} (\E ) = \min \{ \mathrm{codim} ^{(m)} (\E ') , \mathrm{codim} ^{(m)} (\E '')\}$.
Autrement dit, pour tout entier $i \in \N$ fixé, 
l'égalité $\mathrm{codim} ^{(m)} (\E ) \geq i$ est vérifiée si et seulement si 
les deux égalités $\mathrm{codim} ^{(m)} (\E ') \geq i$ et $\mathrm{codim} ^{(m)} (\E'' ) \geq i$
sont satisfaites.
\end{lemm}

\begin{proof}
Il ne coûte pas grand chose de supposer $\X$ affine.
Soient $ \smash{\overset{_\circ}{E}}$ un $\smash{\widehat{D}} ^{(m)} _{\X}$-module cohérent
sans $p$-torsion tel que 
$\smash{\overset{_\circ}{E }} _\Q \riso E$.
Par nothérianité de $\smash{\widehat{D}} ^{(m)} _{\X}$, 
on obtient des $\smash{\widehat{D}} ^{(m)} _{\X}$-modules cohérents
sans $p$-torsion en posant 
$\smash{\overset{_\circ}{E '}}:= f ^{-1}(\smash{\overset{_\circ}{E }})$, 
$\smash{\overset{_\circ}{E'' }}:= g (\smash{\overset{_\circ}{E }}) $.
On vérifie que la suite canoniquement induite
$0\to \smash{\overset{_\circ}{E }}' \to \smash{\overset{_\circ}{E }} \to \smash{\overset{_\circ}{E }}'' \to 0$
est exacte et que l'on dispose des isomorphismes canoniques $\smash{\overset{_\circ}{E ' }} _\Q \riso E'$
et $\smash{\overset{_\circ}{E '' }} _\Q \riso E''$.
La suite exacte obtenue modulo $\pi$, combinée avec
\cite[5.2.4.(i)]{Beintro2} donne la formule souhaitée.
\end{proof}

\begin{prop}
\label{Vir-III.3.6}
Avec les notations de \ref{Nota-suite-spectrale}, 
pour tout $0 \leq i\leq d$, 
$F ^{i} (\E)$ est le plus grand sous $\smash{\widehat{\D}} ^{(m)} _{\X,\,\Q}$-module cohérent de codimension supérieure ou égale à $i$ de $\E$.
\end{prop}

\begin{proof} 
D'après \cite[1.3]{caro-holo-sansFrob}, 
$\mathrm{codim} ^{(m)} (E _{2} ^{i,i}) \geq i$.
Avec le lemme \ref{formule-codim}, on en déduit que $\mathrm{codim} ^{(m)} (E _{\infty} ^{i,i}) \geq i$.
Comme $F ^{i} (\E)  /F ^{i+1} (\E) = E _{\infty} ^{i,i}$, 
il en résulte que $F ^{i} (\E) $ est de codimension supérieure ou égale à $i$. 
Soit $\E ^{i}$ un sous-$\smash{\widehat{\D}} ^{(m)} _{\X,\,\Q}$-module cohérent de $\E$ de codimension supérieure ou égale à $i$.
Par fonctorialité de la suite spectrale de \ref{StSp}, 
on obtient le diagramme commutatif: 
$$\xymatrix{
{F ^{i} (\E ^{i}) } 
\ar@{=}[r] ^-{}
\ar@{^{(}->}[d] ^-{}
& {\E ^{i}} 
\ar@{^{(}->}[d] ^-{}
\\ 
{F ^{i} (\E) } 
\ar@{^{(}->}[r] ^-{}
& {\E .} 
}$$
La flèche du haut est bien une égalité d'après \ref{Nota-suite-spectrale} et parce que 
$\E ^{i}$ est de codimension supérieure ou égale à $i$. Il en résulte l'inclusion  
$\E ^{i} \hookrightarrow F ^{i} (\E ) $ factorisant ce diagramme commutatif. 
D'où le résultat.
\end{proof}

\begin{prop}
\label{Filtration-extension}
Soient $m ' \geq m \geq$ deux entiers, $\E  ^{(m)}$ un $\smash{\widehat{\D}} ^{(m)} _{\X,\,\Q}$-module cohérent.
On note $\E  ^{(m')}:= \smash{\widehat{\D}} ^{(m')} _{\X,\,\Q} \otimes _{\smash{\widehat{\D}} ^{(m)} _{\X,\,\Q}}\E  ^{(m)}$, 
$(F ^{i} _{(m)} (\E^{(m)}) ) _{i \geq 0}$ et $(F ^{i} _{(m')} (\E^{(m')}) ) _{i \geq 0}$ les filtrations décroissantes de \ref{Nota-suite-spectrale}.
On dispose de l'isomorphisme canonique 
$$\smash{\widehat{\D}} ^{(m')} _{\X,\,\Q} \otimes _{\smash{\widehat{\D}} ^{(m)} _{\X,\,\Q}}F ^{i} _{(m)} (\E^{(m)}) \riso F ^{i} _{(m')} (\E^{(m')}).$$
\end{prop}

\begin{proof}
Cela résulte du fait que l'extension
$\smash{\widehat{\D}} ^{(m)} _{\X,\,\Q} \to 
\smash{\widehat{\D}} ^{(m')} _{\X,\,\Q}$ est plate (voir \cite{Be1}) et
que le foncteur dual commute aux extensions (voir \cite{virrion}).
\end{proof}

\begin{vide}
\label{passage-m-dag}
Les assertions de \ref{Nota-suite-spectrale}, \ref{formule-codim}
et \ref{Vir-III.3.6} (resp. \ref{Filtration-extension})
restent valables en remplaçant $\smash{\widehat{\D}} ^{(m)} _{\X,\,\Q} $ (resp. $\smash{\widehat{\D}} ^{(m')} _{\X,\,\Q} $) 
par $\D ^\dag _{\X, \Q}$.
\end{vide}

\begin{nota}
\label{Nota-hol}
Soit $\E$ un $\D ^\dag _{\X, \Q}$-module cohérent. On note $\E ^{*}:= \mathcal{H} ^{0} \DD (\E)$. 
On pose aussi $\E ^\mathrm{hol}:= (\E ^{*})^{*}$. 
On dispose ainsi du foncteur 
$(- )^\mathrm{hol}\,:\, 
\mathrm{Coh} ( \D ^\dag _{\X, \Q}) 
\to 
\mathrm{Hol} ( \D ^\dag _{\X, \Q})$ défini par 
$\E \mapsto \E ^\mathrm{hol}$. 
En considérant la suite spectrale de \ref{StSp} 
(avec $\smash{\widehat{\D}} ^{(m)} _{\X,\,\Q} $ remplacé 
par $\D ^\dag _{\X, \Q}$),
on vérifie l'égalité $F ^{d} (\E) =\E ^\mathrm{hol} $.
D'après \ref{Vir-III.3.6} (avec la remarque \ref{passage-m-dag}), 
$\E ^\mathrm{hol} $ est le plus grand sous-module holonome de $\E$. 
On remarque de plus que 
le foncteur oubli
$\mathrm{Hol} ( \D ^\dag _{\X, \Q}) 
\to 
\mathrm{Coh} ( \D ^\dag _{\X, \Q})$
est adjoint à gauche du foncteur $(- )^\mathrm{hol}$.
En particulier, le foncteur $(- )^\mathrm{hol}$ est exact à gauche et commute aux limites projectives.
\end{nota}

\begin{lemm}
\label{stab-hol-f_+fet}
Soient $f \,:\, \X \to \Y$ un morphisme fini étale, $\E$ un $\D ^\dag _{\X, \Q}$-module cohérent.
\begin{enumerate}
\item \label{stab-hol-f_+fet1} On dispose des isomorphismes
$f _+ (\E ^\mathrm{hol} ) 
\riso 
f _+ (\E ) ^\mathrm{hol} $ et $f _+ (\E /\E ^\mathrm{hol} ) 
\riso 
f _+ (\E ) /f _+ (\E ) ^\mathrm{hol} $. 
\item $\E$ est holonome si et seulement si $f _+ (\E)$ est holonome. 
\item Si $\E$ est surcohérent dans $\X$ (voir la définition \ref{defi-surcoh})
alors
$f _+ (\E)$ est un $\D ^\dag _{\Y, \Q}$-module surcohérent dans $\Y$.
\end{enumerate}

\end{lemm}

\begin{proof}
La première assertion \ref{stab-hol-f_+fet1} résulte du théorème de dualité relative (voir \cite{Vir04}) 
et de l'exactitude du foncteur $f _+$ lorsque $f$ est fini et étale.
La seconde résulte du premier isomorphisme de \ref{stab-hol-f_+fet1} via le fait qu'un monomorphisme 
$\E' \hookrightarrow \E$ de $\D ^\dag _{\X, \Q}$-modules cohérents est un isomorphisme si et seulement si 
$f _+ (\E') \hookrightarrow  f _+ ( \E)$ est un isomorphisme. 
La dernière assertion découle de l'isomorphisme de commutation du foncteur localisation aux images directes (voir
\cite[2.2.18.2]{caro_surcoherent}) et du fait que le foncteur $f _+$ est acyclique car $f$ est fini et étale.

\end{proof}

\begin{lemm}
\label{hol-cart}
Soit $\E' \subset \E$ un monomorphisme de $\D ^\dag _{\X, \Q}$-modules cohérents.
Le morphisme $\E ^{\prime \mathrm{hol}}\to \E ^{\mathrm{hol}}$ induit par 
le foncteur $(- )^\mathrm{hol}$ (voir \ref{Nota-hol}) est alors un monomorphisme.
De plus, $\E ^{\prime \mathrm{hol}}=\E ' \cap  \E ^{\mathrm{hol}}$.
\end{lemm}

\begin{proof}
La première assertion résulte de l'exactitude à gauche du foncteur $(- )^\mathrm{hol}$ (voir \ref{Nota-hol}).
Comme $\E ' \cap  \E ^{\mathrm{hol}}$ est un $\D ^\dag _{\X, \Q}$-module holonome (car inclus dans un module $\D ^\dag _{\X, \Q}$-holonome) inclus dans
$\E'$, comme $\E ^{\prime \mathrm{hol}}$ est
le plus grand sous-$\D ^\dag _{\X, \Q}$-module holonome de $\E'$, 
on en déduit que $\E ^{\prime \mathrm{hol}}=\E ' \cap  \E ^{\mathrm{hol}}$.
\end{proof}

\begin{lemm}
\label{E1*=0}
Soit $\E$ un $\D ^\dag _{\X, \Q}$-module cohérent. 
Le morphisme canonique 
$\E ^* \to (\E ^\mathrm{hol}) ^{*} $
est un isomorphisme, i.e.,
$(\E / \E ^{\mathrm{hol}}) ^{*}=0$.
En particulier, $(\E / \E ^{\mathrm{hol}}) ^{\mathrm{hol}}=0$.
\end{lemm}

\begin{proof}
Posons $\E _\mathrm{n\text{-}hol} := \E / \E ^{\mathrm{hol}}$.
Comme $\E ^\mathrm{hol}$ est holonome, alors $\mathcal{H} ^{-1} \DD (\E  ^\mathrm{hol})=0$.
En appliquant le foncteur $\mathcal{H} ^{0} \DD$ à la suite exacte 
$0 \to \E ^\mathrm{hol} \underset{\mathrm{adj}}{\longrightarrow} \E \longrightarrow \E _\mathrm{n\text{-}hol} \to 0$, il en résulte que l'on obtient 
la suite exacte 
$0 \to \E _\mathrm{n\text{-}hol} ^* \to \E ^* \to (\E ^\mathrm{hol}) ^{*}  \to 0$.
Comme c'est un morphisme de $\D ^\dag _{\X, \Q}$-modules holonomes, quitte à utiliser l'isomorphisme de bidualité, il suffit de vérifier qu'en appliquant le foncteur 
$\mathcal{H} ^{0} \DD $ au morphisme $\E ^* \to (\E ^\mathrm{hol}) ^{*} $, on obtient un isomorphisme. 
Or, ce dernier
est le morphisme $(-)^\mathrm{hol} (\mathrm{adj})\,:\,(\E ^\mathrm{hol}) ^\mathrm{hol} \to \E ^\mathrm{hol}$.
Comme le morphisme d'adjonction de la forme $\mathrm{adj}\,:\,\FF ^\mathrm{hol} \subset \FF $
est fonctoriel en $\FF$, on dispose 
du diagramme commutatif
\begin{equation}
\notag
\xymatrix@ C=2cm {
{\E ^\mathrm{hol}} 
\ar@{^{(}->}[r] ^-{\mathrm{adj}}
& {  \E } 
\\
 {(\E ^\mathrm{hol}) ^\mathrm{hol} } 
 \ar@{^{(}->}[r] ^-{(-)^\mathrm{hol} (\mathrm{adj})}
 \ar@{^{(}->}[u] ^-{\mathrm{adj}}
 & 
 {\E ^\mathrm{hol} .} 
 \ar@{^{(}->}[u] ^-{\mathrm{adj}}
 }
\end{equation}
Comme $\E ^\mathrm{hol}$ est holonome, la flèche de gauche est bijective. 
De plus, les flèches du haut et de droite sont identiques. Il en résulte que la flèche du bas est un isomorphisme.

\end{proof}

\subsection{Surcohérence dans un $\V$-schéma formel lisse}
\label{section3.2}
On désignera par $\X$ un $\V$-schéma formel lisse et $D$ un diviseur de sa fibre spéciale $X$.
Pour tout sous-schéma fermé $Z$ de $X$,
on note $(\hdag Z)$ le foncteur localisation en dehors de $Z$ (voir \cite[2]{caro_surcoherent}). 
L'objectif de cette section est de valider l'isomorphisme \ref{pre-a*Gamma}. Pour cela, introduisons d'abord la notion 
suivante.

\begin{defi}\label{defi-surcoh}
  Soit $\E$ un objet de ($F$-)$D (\D ^\dag _{\X} (\hdag D) _{\Q})$. 
\begin{itemize}
\item  Le complexe $\E$ est {\og $\D ^\dag _{\X} (\hdag D) _{\Q}$-surcohérent dans $\X$\fg} si  
$\E \in (F$-$)D _\mathrm{coh} ^\mathrm{b} (\D ^\dag _{\X} (\hdag D) _{\Q})$ et si,
    pour tout diviseur $T $ de $X$, on ait 
    $(\hdag T )  (\E) \in (F$-$)D _\mathrm{coh} ^\mathrm{b} (\D ^\dag _{\X} (\hdag D) _{\Q})$. 

\item     
On dit que $\E$ est {\og $\D ^\dag _{\X} (\hdag D) _{\Q}$-surcohérent dans $\X$ et après tout changement de base\fg} si, pour tout morphisme d'anneaux de valuation discrète complets
    d'inégales caractéristiques $(0,p)$ de corps résiduels parfaits de la forme le morphisme canonique $\V \to \V '$,
    $\S := \Spf (\V)$, 
    $\S ':= \Spf (\V')$, $f\,:\, \X ':= \X \times _{\S} \S'\to \X$ le morphisme canonique, 
    $D ' := f ^{-1} (D)$, le complexe induit par changement de base par $\V \to \V'$ noté
    $f ^* (\E):= \D ^\dag _{\X'/\S'} (\hdag D') _{\Q} \otimes _{f ^{-1} \D ^\dag _{\X/\S} (\hdag D) _{\Q}} f ^{-1} \E$ 
    est $\D ^\dag _{\X'} (\hdag D') _{\Q}$-surcohérent dans $\X '$ (pour les opérateurs différentiels d'ordre fini, 
    le changement de base a été défini et étudié dans 
    \cite[3.2]{Be2} ; on en déduit la version $\D ^{\dag}$ de manière usuelle). 
    De même, on ajoutera {\og après tout changement de base\fg} pour d'autres notions, e.g. la surholonomie. 

\item   Enfin un ($F$-)$\D ^\dag _{\X} (\hdag D) _{\Q}$-module est $\D ^\dag _{\X} (\hdag D) _{\Q}$-surcohérent dans $\X$ s'il l'est en tant
  qu'objet de ($F$-)$D ^\mathrm{b} (\D ^\dag _{\X} (\hdag D) _{\Q})$. Lorsque $D$ est vide ou s'il n'y a aucune ambiguïté sur $D$, on dira plus simplement
surcohérent dans $\X$.

\end{itemize}
\end{defi}

%\newpage

%Yes, we can keep the first version of overholonomicity after any base change, i.e. any base change is allowed (but at least with perfect residue fields). 
%Below is the answer to the question of why overconvergent $F$-isocrystals are overholonomic after any base change:

%\bigskip 

%\newpage

\begin{exem}
\label{holo-chtdebase}
\begin{enumerate}
\item Les propriétés de cohérence et d'holonomie sont stables par changement de base.
Plus précisément, soit un morphisme d'anneaux de valuation discrète complets
    d'inégales caractéristiques $(0,p)$ de corps résiduels parfaits de la forme le morphisme canonique $\V \to \V '$,
$\S := \Spf (\V)$, 
    $\S ':= \Spf (\V')$, 
 soient $f\,:\, \X ':= \X \times _{\S} \S' \to \X$ le morphisme canonique, 
    $D ' := f ^{-1} (D)$. Si $\E \in D ^{\mathrm{b}} _{\mathrm{coh}} (\D ^\dag _{\X} (\hdag D) _{\Q})$, 
    alors $f ^* (\E) \in D ^{\mathrm{b}} _{\mathrm{coh}} (\D ^\dag _{\X '} (\hdag D') _{\Q})$.
    De plus, si $\E$ est un $\D ^\dag _{\X} (\hdag D) _{\Q}$-module holonome (voir \cite[2.10]{caro-holo-sansFrob} ou \cite{Beintro2} pour une définition 
    lorsque $\E$ est muni d'une structure de Frobenius), alors 
    $f ^* (\E)$ est un $\D ^\dag _{\X'} (\hdag D') _{\Q}$-module holonome. En effet, cela se déduit du fait que l'extension 
    $f ^{-1} \D ^\dag _{\X} (\hdag D) _{\Q} \to \D ^\dag _{\X'} (\hdag D') _{\Q}$ est plate (en effet, avec \cite{Be1}, on se ramène au cas des opérateurs différentiels d'ordre fini, cas qui résulte alors
    du fait que $\V \to \V'$ est plate).

\item  Lorsque $\X$ est un $\V$-schéma formel propre et lisse, 
d'après \cite[2.3.17]{caro-Tsuzuki}, les notions 
%d'après \cite{caro-stab-holo}, les notions de
%    $F \text{-}\D ^\dag _{\X, \Q}$-modules holonomes,
   de $F \text{-}\D ^\dag _{\X,\Q}$-complexes surcohérents, 
   de $F \text{-}\D ^\dag _{\X,\Q}$-complexes surholonomes, 
   de $F$-complexes dévissables en $F$-isocristaux surcohérents sont équivalentes. 
   D'après le théorème de la réduction semistable de Kedlaya (\cite{kedlaya-semistableIV}),
   les $F$-isocristaux surconvergents sont {\og potentiellement unipotent avec des résidues nilpotents\fg}, 
   i.e. après une altération génériquement étale ils deviennent unipotents avec des résidues nilpotents. 
   Il découle du théorème  \cite[2.3.13]{caro-Tsuzuki} (et en calquant la preuve de \cite[2.3.15]{caro-Tsuzuki}) 
   que les isocristaux potentiellement unipotents avec des résidues nilpotents
   sont surholonomes. Comme le fait qu'un isocristal surconvergent soit potentiellement unipotent avec des résidues nilpotents
   est stable par n'importe quel changement de base, 
   il en résulte que les $F$-isocristaux surcohérents sont surholonomes et après tout changement de base. 
   Par dévissage, il en découle qu'un $F \text{-}\D ^\dag _{\X,\Q}$-complexe surholonome
   est un $\D ^\dag _{\X,\Q}$-complexe surholonome et après tout changement de base.

\end{enumerate}

\end{exem}

\begin{lemm}
\label{lemm-surcoh-4}
\begin{enumerate}
\item Soit $(\X _i ) _i $ un recouvrement ouvert de $\X$.
Un complexe $\E\in D ^\mathrm{b} (\D ^\dag _{\X} (\hdag D) _{\Q})$ est surcohérent dans $\X$ si et seulement si $\E |\X _i$ est 
surcohérent dans $\X _i$ pour tout $i$.

\item Si deux termes d'un triangle distingué de $D ^\mathrm{b} (\D ^\dag _{\X} (\hdag D) _{\Q})$ sont surcohérents dans $\X$,
alors le troisième l'est.

\item Un complexe $\E\in D ^\mathrm{b} (\D ^\dag _{\X} (\hdag D) _{\Q})$ est surcohérent dans $\X$ si et seulement si pour tout entier $j \in \Z$ les modules
$\mathcal{H} ^{j} (\E) $ sont surcohérents dans $\X$.

\item Pour tout sous-schéma fermé $Z$ de $X$, pour tout complexe $\E\in D ^\mathrm{b} (\D ^\dag _{\X} (\hdag D) _{\Q})$ surcohérent dans $\X$,
le complexe $(\hdag Z) (\E)$ est surcohérent dans $\X$.

\end{enumerate}

\end{lemm}

\begin{proof}
Les deux premières assertions sont triviales. La troisième résulte du fait que le foncteur $(\hdag T)$ avec $T$ un diviseur de $X$ est exact sur la catégorie des $\D ^\dag _{\X} (\hdag D) _{\Q}$-modules cohérents. 
Traitons à présent le quatrièmement. Lorsque $Z$ est un diviseur, cela découle de la formule \cite[2.2.14]{caro_surcoherent}: pour tout diviseur $T$ de $X$, on a $(\hdag T ) \circ (\hdag Z) (\E) \riso (\hdag Z \cup T) (\E)$.
On termine la vérification par récurrence sur le nombre de générateurs de l'idéal de $\O _X$ induit par $Z\hookrightarrow X$, 
en utilisant des triangles distingués de Mayer-Vietoris (voir \cite[2.2.16]{caro_surcoherent}).

\end{proof}

\begin{lemm}
[Berthelot-Kashiwara/surcohérence dans un $\V$-schéma formel lisse]
\label{B-K-surcoh}
Soit $u\,:\, \ZZ \hookrightarrow \X$ une immersion fermée 
de $\V$-schémas formels lisses telle que $D \cap Z $ soit un diviseur de $Z$. 
Les foncteurs $u _+$ est $u ^!$ induisent des équivalences quasi-inverses entre la catégorie 
des $\D ^\dag _{\ZZ} (\hdag D \cap Z )_{\Q}$-modules (resp. des complexes de $D ^\mathrm{b} (\D ^\dag _{\ZZ} (\hdag D \cap Z )_{\Q})$)
surcohérents dans $\ZZ$
et celle des $\D ^\dag _{\X} (\hdag D) _{\Q}$-modules (resp. des complexes de $D ^\mathrm{b} (\D ^\dag _{\X} (\hdag D) _{\Q})$)
à support dans $Z$ surcohérents dans $\X$.
\end{lemm}

\begin{proof}
D'après la version cohérente du théorème de Berthelot-Kashiwara (voir \cite[5.3.3]{Beintro2}), 
le lemme est déjà connu en remplaçant la notion de surcohérence dans un $\V$-schéma formel lisse par celle de cohérence
(en effet, le cas respectif découle du cas non respectif, 
car les foncteurs $u _+$ et $u ^!$ sont acycliques sur les catégories non respectives). 
Soient $\FF$ un $\D ^\dag _{\ZZ} (\hdag D \cap Z )_{\Q}$-module surcohérent dans $\ZZ$ et $T$ un diviseur de $X$.
D'après la dernière assertion de \ref{lemm-surcoh-4}, 
$(\hdag T \cap Z) (\FF)$ est un $\D ^\dag _{\ZZ} (\hdag D \cap Z )_{\Q}$-module cohérent.
Il en résulte que $ u _+ ( (\hdag T \cap Z) (\FF))$ est un $\D ^\dag _{\X} (\hdag D) _{\Q}$-module cohérent.
Or, par \cite[2.2.18.2]{caro_surcoherent}, 
$ u _+ \circ (\hdag T \cap Z) (\FF) \riso   (\hdag T ) \circ u _+ ( \FF)$.
On a ainsi vérifié que $u _+ ( \FF)$ est un $\D ^\dag _{\X} (\hdag D) _{\Q}$-module surcohérent dans $\X$.
De même, en utilisant \cite[2.2.18.1]{caro_surcoherent} et via la version cohérente du théorème de Berthelot-Kashiwara, 
on vérifie que si 
$\E$ est un $\D ^\dag _{\X} (\hdag D) _{\Q}$-module à support dans $Z$ surcohérent dans $\X$, 
alors $ u ^! (\E)$ est un $\D ^\dag _{\ZZ} (\hdag D \cap Z )_{\Q}$-module surcohérent dans $\ZZ$.
\end{proof}

\begin{lemm}
\label{stab-surcoh-dansX}
Soient $\E\in D ^\mathrm{b} (\D ^\dag _{\X} (\hdag D) _{\Q})$ un complexe surcohérent dans $\X$
et $u\,:\, \ \ZZ \hookrightarrow \X$ une immersion fermée de $\V$-schémas formels lisses telle que $D \cap Z$ soit un diviseur de $Z$.
Alors, $u ^! (\E) $ est  un complexe de $D ^\mathrm{b} (\D ^\dag _{\ZZ} (\hdag D \cap Z )_{\Q})$ surcohérent dans $\ZZ$. 
\end{lemm}

\begin{proof}

Par \ref{lemm-surcoh-4}, via le triangle de localisation en $Z$ de $\E$ (e.g. voir \cite[5.3.6]{Beintro2}), 
la $\D ^\dag _{\X} (\hdag D) _{\Q}$-surcohérence dans $\X$ de $\E$ et $(\hdag Z )  (\E)$ implique celle de 
$\R \underline{\Gamma} ^{\dag} _Z(\E)$.
Comme ce dernier est à support dans $Z$, avec le théorème de Berthelot-Kashiwara de \ref{B-K-surcoh},
on obtient alors 
que $u ^! \R \underline{\Gamma} ^{\dag} _Z(\E)$
est $\D ^\dag _{\ZZ} (\hdag D \cap Z )_{\Q}$-surcohérent dans $\ZZ$. 
Comme d'après \cite[2.2.18.1]{caro_surcoherent}, 
$u ^! \circ \R \underline{\Gamma} ^{\dag} _Z(\E) \riso \R \underline{\Gamma} ^{\dag} _Z \circ u ^! (\E)
= u ^! (\E)$, on a terminé la démonstration. 
\end{proof}

\begin{vide}
\label{§def-a^*}
Soit $a\,:\, \ZZ \hookrightarrow \X$ une immersion fermée de $\V$-schémas formels lisses. 
Pour tout $\widehat{\D}  ^{(m)} _{\X}$-module $\smash{{\E }} ^{(m)}$, 
pour tout $\widehat{D}  ^{(m)} _{\X}$-module $\smash{{E }} ^{(m)}$,
on définit un complexe à gauche de $\widehat{\D}  ^{(m)} _{\ZZ}$-module ou respectivement de
$\widehat{D}  ^{(m)} _{\ZZ}$-module en posant
\begin{align}
\label{def-a^*}
\L a ^{*} (\smash{{\E }} ^{(m)}) &:= \widehat{\D}  ^{(m)} _{\ZZ \hookrightarrow \X} \otimes ^{\L} _{a ^{-1}\widehat{\D}  ^{(m)} _{\X}}
a ^{-1} \smash{{\E }} ^{(m)},
\\
\L a ^{*} (\smash{{E }} ^{(m)}) &:= \widehat{D}  ^{(m)} _{\ZZ \hookrightarrow \X} \otimes ^{\L} _{\widehat{D}  ^{(m)} _{\X}}
\smash{{E }} ^{(m)}.
\end{align}
On définit de la même façon le foncteur
$\L a ^*$ sur la catégorie des $\D ^\dag _{\X, \Q}$-modules (resp. des $D ^\dag _{\X, \Q}$-modules)
et le foncteur
$\L a _{i} ^{*}$ sur la catégorie des $\D ^{(m)} _{X _i}$-modules (resp. des $D ^{(m)} _{X _i}$-modules).
\end{vide}

\begin{vide}
\label{§La*coh}
Soit $a\,:\, \ZZ \hookrightarrow \X$ une immersion fermée de $\V$-schémas formels affines et lisses. 
Soit $\E $ un $\D  ^{\dag} _{\X,\Q}$-module cohérent. 
Soit $\I$ (resp. $\I _i$) l'idéal de $\O _\X$ (resp. $\O _{X _i}$) 
induit par l'immersion fermée $\ZZ \hookrightarrow \X$ (resp. $Z _i \hookrightarrow X _i $).
Comme $\O _{Z _i}= a  ^{-1} \O _{X _i} / a  ^{-1}  \I _i$,
on vérifie alors $\D ^{(m)} _{Z _i \hookrightarrow X _i}:= a _i ^{*} \D ^{(m)} _{X _i} = a  ^{-1} (\D ^{(m)} _{X _i} / \I _i \D ^{(m)} _{X _i} )$.
Comme $\I$ est $\O _\X$-cohérent, $\I \otimes _{\O _\X} \widehat{\D} ^{(m)} _{\X}$ est un 
$\widehat{\D} ^{(m)} _{\X}$-module à droite cohérent. 
Comme $\I$ est $\O _\X$-plat, alors il en résulte que $\I \widehat{\D} ^{(m)} _{\X}$ et donc 
$ \widehat{\D} ^{(m)} _{\X} / \I  \widehat{\D} ^{(m)} _{\X}$ sont des 
$\widehat{\D} ^{(m)} _{\X}$-modules à droite cohérent.
Via \cite[3.3.9]{Be1}, cela entraîne que le morphisme canonique
\begin{equation}
\label{a*D=D/ID1}
 \widehat{\D} ^{(m)} _{\X} / \I  \widehat{\D} ^{(m)} _{\X}
\to \underset{\underset{i}{\longleftarrow}}{\lim\,} 
\D  ^{(m)} _{X _i} / \I _i \D  ^{(m)} _{X _i}
\end{equation}
est un isomorphisme.
De plus, il découle de \cite[3.2.3]{Be1} que le morphisme canonique
\begin{equation}
\label{a*D=D/ID2}
 \widehat{D} ^{(m)} _{\X} / I  \widehat{D} ^{(m)} _{\X}
\to \underset{\underset{i}{\longleftarrow}}{\lim\,} 
D  ^{(m)} _{X _i} / I _i D  ^{(m)} _{X _i}
\end{equation}
est un isomorphisme.
Comme le foncteur $\Gamma (\X,-)$ commute aux limites projectives,
les isomorphismes \ref{a*D=D/ID1} et \ref{a*D=D/ID2} entraînent
$\Gamma (\X, \widehat{\D} ^{(m)} _{\X} / \I  \widehat{\D} ^{(m)} _{\X})\riso  \widehat{D} ^{(m)} _{\X} / I  \widehat{D} ^{(m)} _{\X}$.
Comme $\widehat{\D}  ^{(m)} _{\ZZ  \hookrightarrow \X} =
\underset{\underset{i}{\longleftarrow}}{\lim\,} 
\D  ^{(m)} _{Z _i  \hookrightarrow X _i} $,
comme $a _*$ commute aux limites projectives,
on déduit de \ref{a*D=D/ID1} l'isomorphisme
$a _* \widehat{\D}  ^{(m)} _{\ZZ  \hookrightarrow \X} \riso
\widehat{\D} ^{(m)} _{\X} / \I  \widehat{\D} ^{(m)} _{\X}$, i.e. 
$\widehat{\D}  ^{(m)} _{\ZZ  \hookrightarrow \X} \riso
a ^{-1} ( \widehat{\D} ^{(m)} _{\X} / \I  \widehat{\D} ^{(m)} _{\X})$.
On en tire
$\widehat{D}  ^{(m)} _{\ZZ  \hookrightarrow \X} \riso
\widehat{D} ^{(m)} _{\X} / I  \widehat{D} ^{(m)} _{\X}$.

En passant au produit tensoriel par $\Q$ sur $\Z$ puis à la limite inductive sur le niveau, 
on obtient alors : 
$\D  ^{\dag} _{\ZZ  \hookrightarrow \X} =
a ^{-1} ( \D ^{\dag} _{\X,\Q} / \I  \D ^{\dag} _{\X,\Q})$
et
$D  ^{\dag} _{\ZZ  \hookrightarrow \X} =
D ^{\dag} _{\X,\Q} / I  D ^{\dag} _{\X,\Q}$.
D'où les isomorphismes: 
\begin{gather}
\label{La*coh}
a ^{!} (\E  ) [ d _a]
\riso 
\L a ^{*} (\E  )\riso a ^{-1} ( \D ^{\dag} _{\X,\Q} / \I  \D ^{\dag} _{\X,\Q} \otimes ^{\L} _{ \D ^{\dag} _{\X,\Q}} \E),\\
\label{GammaLa*coh}
\L a ^{*} (E  )\riso D ^{\dag} _{\X,\Q} / I  D ^{\dag} _{\X,\Q} \otimes ^{\L} _{ D ^{\dag} _{\X,\Q}} E,
\end{gather}
où $d _a $ est la dimension relative de $a$, le premier isomorphisme provenant de \cite[4.3.2.2]{Beintro2}, 
le second résultant du fait que le foncteur $a ^{-1}$ préserve la platitude.

\end{vide}

\begin{rema}
\label{rema-surcoh-a*}
Avec les notations de \ref{§def-a^*}, il découle du premier isomorphisme de \ref{La*coh} et de \ref{stab-surcoh-dansX} que 
si $\E$ est un $\D  ^{\dag} _{\X,\Q}$-module surcohérent dans $\X$ alors 
$\L a ^{*} (\E  ) $ est surcohérent dans $\ZZ$. 
\end{rema}

\begin{lemm}
\label{lemm-pre-a*Gamma}
Soit $a\,:\,\ZZ \hookrightarrow \X$ une immersion fermée de $\V$-schémas formels affines et lisses 
dont l'idéal de $\O_\X$ de définition est principal.
Soit $\E $ un $\D  ^{\dag} _{\X,\Q}$-module cohérent tel que 
$ \mathcal{H} ^{0}a ^! (\E) $ soit un $\D  ^{\dag} _{\ZZ,\Q}$-module cohérent. 
On dispose alors de l'isomorphisme canonique:
\begin{gather}
\label{pre-a*Gamma}
a ^{*} (E)  
\riso 
\Gamma (\ZZ, a ^{*} (\E)).
\end{gather}
\end{lemm}

\begin{proof}
Soit $t$ un élément engendrant l'idéal de définition de $a$. 
Via la résolution plate 
$0\to  \D ^{\dag} _{\X,\Q}\overset{t}{\longrightarrow}  \D ^{\dag} _{\X,\Q}\to  \D ^{\dag} _{\X,\Q} / t  \D ^{\dag} _{\X,\Q}\to 0$,
on déduit de 
 \ref{La*coh} les deux égalités
$a _* a ^{*} (\E) =  \E / t \E$
et
$a _* \mathcal{H} ^{0} a ^{!} (\E) = \ker ( t\,:\, \E \to \E)$.
De même, par \ref{GammaLa*coh}, on obtient $a ^* (E) =  E /t E$.
Comme 
$\mathcal{H} ^{0} a ^{!} (\E)$ 
est un
 $\D  ^{\dag} _{\ZZ,\Q}$-module cohérent, il vérifie les théorèmes de type $A$ et $B$. 
Comme le foncteur $a _*$ est exact et préserve les injectifs (car son adjoint à gauche $a ^{-1}$ est exact), il en résulte que, 
pour tout $i \geq 1$,
$H ^{i} ( \X, a _* \mathcal{H} ^{0} a ^{!} (\E))=0$.
Comme $\E$ vérifie aussi le théorème de type $B$, 
en appliquant le foncteur $\R \Gamma ( \X, -)$ à la suite exacte
$0 \to a _* \mathcal{H} ^{0} a ^{!} (\E) \to \E \to \E/ a _* \mathcal{H} ^{0} a ^{!} (\E)\to 0$, 
il en dérive que, 
pour tout $i \geq 1$,
$H ^{i} ( \X,\E / a _* \mathcal{H} ^{0} a ^{!} (\E))=0$
et 
$\Gamma (\X, \E / a _* \mathcal{H} ^{0} a ^{!} (\E))= E / \Gamma (\ZZ, \mathcal{H} ^{0} a ^{!} (\E))$.
On en déduit qu'en appliquant le foncteur 
$\Gamma (\X, -)$ à la suite exacte
$0\to \E / a _* \mathcal{H} ^{0} a ^{!} (\E) \overset{t}{\longrightarrow} \E \to \E /t \E\to 0$, 
on obtient l'isomorphisme 
$\Gamma (\X, \E /t \E) \riso E /t E$.

\end{proof}

\begin{prop}
Soit $a\,:\, \ZZ \hookrightarrow \X$ une immersion fermée de $\V$-schémas formels affines et lisses. 
On suppose $\X$ muni de coordonnées locales $t _1, \dots, t _n$  telles que $\ZZ = V ( t _1, \dots, t _r)$.
Soient $\E$ un $\D ^\dag _{\X, \Q}$-module surcohérent dans $\X$.
On bénéficie de l'isomorphisme canonique:
\begin{gather}
\label{a*Gamma}
a ^{*} (E)  
\riso 
\Gamma (\ZZ, a ^{*} (\E)).
\end{gather}
\end{prop}

\begin{proof}
On procède par récurrence sur $r$. Lorsque $r=1$, cela découle du lemme \ref{lemm-pre-a*Gamma}. 
Notons $\X':=  V ( t _1)$. D'après \ref{rema-surcoh-a*}, 
en notant $b \,:\, \X' \hookrightarrow \X$
l'immersion fermée canonique, $\E' := b ^{*} (\E)$ est un $\D ^\dag _{\X', \Q}$-module surcohérent dans $\X'$.
En notant $c\,:\, \ZZ \hookrightarrow \X'$ l'immersion fermée canonique, 
par hypothèse de récurrence, on obtient les isomorphismes 
$b ^{*} (E)  
\riso 
\Gamma (\X', b ^{*} (\E))=E'$
et
$c ^{*} (E')  
\riso 
\Gamma (\ZZ, c ^{*} (\E'))$.
Comme $a ^{*} (E)\riso c ^{*} \circ b ^* (E)$ et 
$a ^{*} (\E)\riso c ^{*} \circ b ^* (\E)$, on en déduit le résultat.

\end{proof}

\begin{rema}
L'isomorphisme \ref{a*Gamma} est utilisé dans la démonstration du lemme \ref{E1-hol-ouvdense} (voir l'étape II.4)).
J'ignore si cet isomorphisme reste valable lorsque $\E$ est seulement un $\D ^\dag _{\X, \Q}$-module cohérent.
\end{rema}

\subsection{Preuve du résultat principal}

Soient $\X$ un $\V$-schéma formel lisse et $D$ un diviseur de sa fibre spéciale $X$.

\begin{lemm}
Soit $a\,:\, \ZZ \hookrightarrow \X$ une immersion fermée de $\V$-schémas formels affines et lisses. 
Soient $\smash{\overset{_\circ}{\E }} ^{(m)}$ un $\widehat{\D}  ^{(m)} _{\X}$-module cohérent sans $p$-torsion
et $\smash{\overset{_\circ}{E }} ^{(m)}:= \Gamma (\X, \smash{\overset{_\circ}{\E }} ^{(m)})$.
Notons 
$\smash{\overset{_\circ}{\E }} ^{(m)}  _i := 
\smash{\overset{_\circ}{\E}} ^{(m)} \otimes ^\L _{\V} \V / \pi ^{i+1} \V 
\liso 
\smash{\overset{_\circ}{\E }} ^{(m)} \otimes _{\V} \V / \pi ^{i+1} \V $,
$\smash{\overset{_\circ}{E }} ^{(m)}  _i:= \Gamma (\X, \smash{\overset{_\circ}{\E }} ^{(m)}  _i)$.

Si pour tout entier $i$ le morphisme canonique 
$\L a _i ^{*} (\smash{\overset{_\circ}{\E_i}} ^{(m)})\to 
a _i ^{*} (\smash{\overset{_\circ}{\E_i}} ^{(m)})$
est un isomorphisme,
on dispose alors des isomorphismes canoniques :
\begin{gather}
\label{MLsansu+cal}
\L a ^{*} (\smash{\overset{_\circ}{\E }} ^{(m)}) \riso a ^{*} (\smash{\overset{_\circ}{\E }} ^{(m)}) \riso
\underset{\underset{i}{\longleftarrow}}{\lim\,} 
a _i ^{*} (\smash{\overset{_\circ}{\E_i}} ^{(m)}), \\
\label{MLsansu+}
\L a ^{*} (\smash{\overset{_\circ}{E }} ^{(m)}) \riso a ^{*} (\smash{\overset{_\circ}{E }} ^{(m)}) \riso
\underset{\underset{i}{\longleftarrow}}{\lim\,} 
a _i ^{*} (\smash{\overset{_\circ}{E_i}} ^{(m)}), \\
\label{MLsansu+isocoro}
a ^{*} (\smash{\overset{_\circ}{E }} ^{(m)})  
\riso 
\Gamma (\ZZ, a ^{*} (\smash{\overset{_\circ}{\E }} ^{(m)}) ).
\end{gather}
\end{lemm}

\begin{proof}
D'après \cite[3.4.5]{Beintro2}, 
on dispose de l'isomorphisme 
$\L a ^{*} (\smash{\overset{_\circ}{\E }} ^{(m)})
\riso 
\R \underset{\underset{i}{\longleftarrow}}{\lim\,} 
\L a _i ^{*} (\smash{\overset{_\circ}{\E_i}} ^{(m)})$.
On obtient ainsi les deux isomorphismes canoniques du milieu 
\begin{equation}
\label{MLsansu+iso}
a ^{*} (\smash{\overset{_\circ}{\E }} ^{(m)}) \liso \L a ^{*} (\smash{\overset{_\circ}{\E }} ^{(m)})
\riso 
\R \underset{\underset{i}{\longleftarrow}}{\lim\,} 
\L a _i ^{*} (\smash{\overset{_\circ}{\E_i}} ^{(m)})
\riso 
\R \underset{\underset{i}{\longleftarrow}}{\lim\,} 
a _i ^{*} (\smash{\overset{_\circ}{\E_i}} ^{(m)})
\liso 
\underset{\underset{i}{\longleftarrow}}{\lim\,} 
a _i ^{*} (\smash{\overset{_\circ}{\E_i}} ^{(m)}).
\end{equation}
Comme $\L a ^{*} (\smash{\overset{_\circ}{\E }} ^{(m)})$ (resp. $\R \underset{\underset{i}{\longleftarrow}}{\lim\,} 
a _i ^{*} (\smash{\overset{_\circ}{\E_i}} ^{(m)})$)
est un complexe à gauche (resp. à droite), ces deux complexes sont donc isomorphes à un module. D'où les deux isomorphismes aux extrémités de \ref{MLsansu+iso}.

Traitons à présent le second isomorphisme. 
Via le théorème de type $A$ pour les $\D  ^{(m)} _{X _i} $-modules quasi-cohérents, 
on obtient par associativité du produit tensoriel l'isomorphisme canonique
$\L a _i ^{*} (\smash{\overset{_\circ}{\E_i}} ^{(m)})
\riso
\D  ^{(m)} _{Z _i}
\otimes _{D  ^{(m)} _{Z _i}}
\L a _i ^{*} (\smash{\overset{_\circ}{E_i}} ^{(m)})$.
Via le théorème de type $A$ pour les modules $\D  ^{(m)} _{Z _i} $-modules quasi-cohérents, 
comme $\L a _i ^{*} (\smash{\overset{_\circ}{\E_i}} ^{(m)})\riso 
a _i ^{*} (\smash{\overset{_\circ}{\E_i}} ^{(m)})$,
il 
en résulte l'isomorphisme 
$\L a _i ^{*} (\smash{\overset{_\circ}{E_i}} ^{(m)})\riso
a _i ^{*} (\smash{\overset{_\circ}{E_i}} ^{(m)})$.
De plus, on dispose d'après \ref{ThB-coro} de l'isomorphisme
canonique
$\smash{\overset{_\circ}{E }} ^{(m)}  _i
\riso 
\smash{\overset{_\circ}{E}} ^{(m)} \otimes  _{\V} \V / \pi ^{i+1} \V $.
On obtient alors de manière analogue (on utilise l'appendice $B$ de \cite{Berthelot-Ogus-cristalline} au lieu de  \cite[3.4.5]{Beintro2}) 
les isomorphismes \ref{MLsansu+iso} avec des lettres droites. En particulier, on a vérifié \ref{MLsansu+}. 
Comme le foncteur $\Gamma (\ZZ,-)$ commute aux limites projectives, 
comme $\Gamma (\ZZ, a _i ^{*} (\smash{\overset{_\circ}{\E_i}} ^{(m)}))
\riso a _i ^{*} (\smash{\overset{_\circ}{E_i}} ^{(m)})$, on en déduit l'isomorphisme \ref{MLsansu+isocoro} à partir de
\ref{MLsansu+cal} et \ref{MLsansu+}. 
\end{proof}

\begin{lemm}\label{lemm2.2.7-courbe}
  Soient $R$ un anneau de valuation discrète complet d'inégales caractéristiques $(0,p)$,
  $L$ son corps des fractions et
  $\overset{\circ}{M}$ un ${R}$-module séparé pour la topologie $p$-adique sans $p$-torsion.

  Pour toute famille $L$-libre ($P _1,\dots, P_r$) de $\overset{\circ}{M} _\Q$,
  il existe alors un entier $N _0 \geq 0$ tel que, pour tout entier $n$,
  on ait l'inclusion :
  \begin{equation}
    \label{libre-inclusion1}
    ({L} P _1 +\dots +{L} P_r)\cap p ^{N _0 +n} \overset{\circ}{M} \subset p ^n ({R} P _1 +\dots +{R} P_r).
  \end{equation}
\end{lemm}
\begin{proof}
Comme le sous-$L$-espace vectoriel 
$L P _1 +\dots + L P_r$ de $\overset{\circ}{M} _\Q$ est de dimension finie,
d'après \cite[4.13]{Schneider-NonarchFuncAn}, 
les filtrations 
données respectivement par $(L P _1 +\dots + L P_r)\cap p ^{N _0 +n} \overset{\circ}{M}$
et $p ^n ({R} P _1 +\dots +{R} P_r)$) induisent deux normes équivalentes. 
\end{proof}

\begin{lemm}
\label{E1-hol-ouvdense}
Soient $\E$  un $\D ^\dag _{\X, \Q}$-module surcohérent dans $\X$ et après tout changement de base. 
Notons $Z$ le support de $ \E / \E ^{\mathrm{hol}}$. 
Par l'absurde, on suppose que $Z$ est non-vide. 
Pour toute composante irréductible $Z'$ de $Z$,
il existe alors un ouvert affine $\Y$ de $\X$ tel que 
\begin{enumerate}
\item l'ouvert $Y \cap Z'$ soit lisse et dense dans $Z'$ ;
\item le module $(\E / \E ^{\mathrm{hol}} )| \Y$ soit holonome. 
\end{enumerate}

\end{lemm}

\begin{proof}
On remarquera que l'hypothèse de finitude sur $\E$ n'est utilisée qu'à l'étape $II.4)$.

{\it Étape $I)$: Préliminaires, réduction du problème et notations}.

{\it Étape $1).$}
Le lemme est local en $\X$. 
Quitte à remplacer $\X$ par un ouvert affine $\Y$ de $\X$ tel que $Y \cap Z'=Y \cap Z$ et tel que $Y \cap Z'$ soit un ouvert lisse et dense dans $Z'$,
on peut ainsi supposer $\X$ affine, $Z=Z'$ et $Z$ lisse. 
D'après le théorème \cite[3.7]{caro-holo-sansFrob} et la proposition \ref{stab-surcoh-dansX}, il existe un ouvert dense de $\X$ sur lequel $\E$ devient holonome.
Il en résulte que la dimension du support de $\E _{\mathrm{n\text{-}hol}}:= \E / \E ^{\mathrm{hol}}$ est strictement plus petite que celle de $X$.
Notons $r\geq 1$ la codimension de $Z$ dans $X$.
D'après \cite[Exp. III]{sga1},
il existe alors une immersion fermée de $\V$-schémas formels affines et lisses de la forme
$\ZZ \hookrightarrow \X$ qui relève $Z\hookrightarrow X$.
Grâce au théorème de Kedlaya de \cite{Kedlaya-coveraffinebis}, quitte à rétrécir $\X$, 
il existe un morphisme fini et étale de la forme $h \,:\, \X \to \widehat{\A} ^{d} _{\V}$ tel que 
$h (Z) \subset \A ^{d-r} _k$.
Grâce au lemme \ref{stab-hol-f_+fet},
on se ramène ainsi au cas où $\ZZ \hookrightarrow \X$ est l'immersion fermée $\widehat{\A} ^{d-r} _\V \hookrightarrow \widehat{\A} ^{d} _\V $.

{\it Étape $2).$
Comme $\E _{\mathrm{n\text{-}hol}}$ est à support dans $Z$, d'après le théorème de Berthelot-Kashiwara, il existe un $\D ^\dag _{\ZZ, \Q}$-module cohérent $\FF $ tel que
$ \E _{\mathrm{n\text{-}hol}} \riso u _+ (\FF )$.
Il suffit de prouver qu'il existe un ouvert dense $\mathfrak{V}$ de $\ZZ$ tel que $\FF |\mathfrak{V}$ soit
un $\O _{\mathfrak{V},\Q}$-module libre de type fini, ce qui sera établi à l'étape $II)$.}

En effet, cela entraîne a fortiori que $\FF |\mathfrak{V}$ est un $\D^\dag _{\mathfrak{V}, \Q}$-module holonome (voir l'exemple \cite[2.11]{caro-holo-sansFrob}). 
De plus, 
comme l'holonomie est préservée par l'image directe d'une immersion fermée, 
comme $\E _{\mathrm{n\text{-}hol}} \riso u _+ ( \FF)$,
il en résulte l'holonomie de $\E _{\mathrm{n\text{-}hol}} |\Y$
pour tout ouvert $\Y$ de $\X$ tel que 
$ \Y \cap \ZZ = \mathfrak{V}$. 
D'où la vérification de l'étape $2)$.

{\it Étape $3).$}
D'après \cite[3.6]{caro-holo-sansFrob}, 
quitte à faire un changement de base, grâce à l'étape $2)$, 
on se ramène au cas où $k$ est algébriquement clos et non dénombrable. 

{\it Étape $4).$}
Il existe $m _0$ assez grand, 
il existe un $\widehat{\D} ^{(m _0)} _{\ZZ,\Q}$-module cohérent $\FF  ^{(m _0)}$ 
tel que l'on dispose de l'isomorphisme 
$\D ^\dag _{\ZZ, \Q} \otimes _{\widehat{\D} ^{(m_0)} _{\ZZ,\Q}} \FF   ^{(m _0)} \riso \FF $.
À présent, $m$ désignera toujours un entier plus grand que $m _0$.
Notons 
$\FF   ^{(m)} :=\widehat{\D} ^{(m)} _{\ZZ,\Q}\otimes _{\widehat{\D} ^{(m_0)} _{\ZZ,\Q}} \FF   ^{(m _0)}.$ 
On dispose des morphismes canoniques $\widehat{D}  ^{(m)} _{\ZZ,\Q}$-linéaires
$\rho _{m}\,:\,F   ^{(m)} \to F   ^{(m+1)}$ (on rappelle que les lettres droites désignent les sections globales du faisceau correspondant).
Soit $\smash{\overset{_\circ}{\FF}} ^{(m _0)}$ un $\widehat{\D}  ^{(m_0)} _{\ZZ}$-module cohérent sans $p$-torsion
tel que $\smash{\overset{_\circ}{\FF}} ^{(m_0)} _\Q \riso \FF   ^{(m_0)}$. 
On construit par récurrence sur $m \geq m _0 +1$ 
un $\widehat{\D}  ^{(m)} _{\ZZ}$-module cohérent $\smash{\overset{_\circ}{\FF}} ^{(m)}$ sans $p$-torsion
tel que $\smash{\overset{_\circ}{\FF}} ^{(m)} _\Q \riso \FF   ^{(m)}$
et tel que $\rho _{m -1} (\smash{\overset{_\circ}{F}} ^{(m-1)}) \subset 
\smash{\overset{_\circ}{F}} ^{(m)}$ (en effet, 
via des théorèmes de type $A$, 
cela découle de la noethérianité de $\widehat{D}  ^{(m)} _{\ZZ}$).
Comme $\smash{\overset{_\circ}{\FF }} ^{(m)}$ est sans $p$-torsion, 
on pose et on obtient 
$\smash{\overset{_\circ}{\FF }} ^{(m)}  _i :=   \V / \pi ^{i+1} \V  \otimes ^\L _{\V}\smash{\overset{_\circ}{\FF }} ^{(m)} 
\riso 
 \V / \pi ^{i+1} \V  \otimes  _{\V}\smash{\overset{_\circ}{\FF }} ^{(m)} $.

On note ensuite $\smash{\overset{_\circ}{\E }} _{\mathrm{n\text{-}hol}} ^{(m)}:= u _+ ^{(m)} (\smash{\overset{_\circ}{\FF}} ^{(m)})$
et $\E _{\mathrm{n\text{-}hol}} ^{(m)} := \smash{\overset{_\circ}{\E }} _{\mathrm{n\text{-}hol},\Q} ^{(m)} $. 
Grâce à la première étape (étape qui n'utilise pas l'hypothèse de finitude sur les fibres) 
du théorème \cite[3.4]{caro-holo-sansFrob}, il existe un ouvert dense $\ZZ _{(m)}$ de $\ZZ$ tel que 
$\smash{\overset{_\circ}{\FF}} ^{(m)}|\ZZ _{(m)}$ est isomorphe au complété $p$-adique d'un  $\O _{\ZZ _{(m)}}$-module libre. 
Il est donc de la forme 
$\smash{\overset{_\circ}{\FF}} ^{(m)} |\ZZ _{(m)}\riso ( \O _{\ZZ _{(m)}} ^{(E _m)}) ^{\widehat{}}$ où $E _m$ est un ensemble a priori dénombrable.

Désignons alors par $t _1, \dots, t _d$ les coordonnées canoniques de $\X$
telles que $\ZZ = V ( t _1, \dots, t _r)$.
On obtient des sous-$\V$-schémas formels fermés de $\X$ en posant
$\X':= V ( t _{r+1}, \dots, t _d)$ et $\ZZ':= V ( t _{1}, \dots, t _d)$.
On obtient alors le diagramme canonique cartésien :
\begin{equation}
\xymatrix{
{\ZZ} 
\ar@{^{(}->}[r] ^-{u}
& 
{\X } 
\\ 
{\ZZ'} 
\ar@{^{(}->}[r] ^-{u'}
\ar@{^{(}->}[u] ^-{b}
& 
{\X'.} 
\ar@{^{(}->}[u] ^-{a}
}
\end{equation}

Comme $k$ est algébriquement clos et non dénombrable (voir l'étape $3)$), 
quitte à faire un changement de coordonnées, on peut supposer
que $ |Z'| \in \cap _{m \in \N} \, \ZZ _{(m)}$.

{\it Étape $5).$}
Comme $|Z'| \in  \ZZ _{(m)}$, via en outre \ref{MLsansu+},
on obtient l'isomorphisme
$ b ^{*} (\smash{\overset{_\circ}{\FF}} ^{(m)})
\riso ( \V ^{(E _m)}) ^{\widehat{}}$.
Ainsi, $ b ^{*} (\smash{\overset{_\circ}{\FF}} ^{(m)})$
est sans $p$-torsion, séparé et complet pour la topologie $p$-adique.
D'après \ref{MLsansu+} et \ref{MLsansu+isocoro} on bénéficie de plus des isomorphismes
$$b ^{*} (\smash{\overset{_\circ}{\FF}} ^{(m)})=
b  ^{*} (\smash{\overset{_\circ}{F}} ^{(m)})
\riso 
\underset{\underset{i}{\longleftarrow}}{\lim\,} b _i ^{*} (\smash{\overset{_\circ}{F_i}} ^{(m)}).$$
Remarquons enfin que via \ref{ThB-coro}, on bénéficie de l'isomorphisme 
$b _i ^{*} (\smash{\overset{_\circ}{F_i}} ^{(m)})
\riso 
\V / \pi ^{i+1} \V \otimes _{\V} b  ^{*} (\smash{\overset{_\circ}{F}} ^{(m)})$.

{\it Étape $II)$: il existe un entier $m _2$ assez grand tel que 
$\FF |\ZZ _{(m _2)}$ soit un $\O _{\ZZ _{(m _2)}, \Q}$-module libre de type fini.}

\medskip 
{\it Étape $1)$. Acyclicité.}
D'après \ref{iso-chgt-de-base}, on dispose des isomorphismes 
$\L a _i ^{*} \circ u _{i+} (\smash{\overset{_\circ}{\FF _i}} ^{(m)}) \riso u ^{\prime} _{i+}  \circ \L b _i ^{*} (\smash{\overset{_\circ}{\FF_i}} ^{(m)}) $ (en effet, on remarque que les dimensions relatives de $a$ et $b$ sont identiques).
Comme $\smash{\overset{_\circ}{\FF }} ^{(m)}|\ZZ _{(m)}$ est plat, 
on obtient $ \L b _i ^{*} (\smash{\overset{_\circ}{\FF_i}} ^{(m)})
\liso 
 b _i ^{*} (\smash{\overset{_\circ}{\FF_i}} ^{(m)})$.
 D'où les isomorphismes
 \begin{equation}
 \label{A*E-sansL}
 \L a _i ^{*} \circ u _{i+} (\smash{\overset{_\circ}{\FF _i}} ^{(m)}) 
 \riso 
a _i ^{*} \circ u _{i+} (\smash{\overset{_\circ}{\FF _i}} ^{(m)}) 
\riso u ^{\prime} _{i+}  \circ b _i ^{*} (\smash{\overset{_\circ}{\FF_i}} ^{(m)}).
\end{equation}

\medskip 
{\it Étape $2)$. Le module $a ^{*} (\smash{\overset{_\circ}{\E }} _{\mathrm{n\text{-}hol}} ^{(m)}) $ est pseudo-quasi-cohérent (voir la définition \ref{def-qcsstor}).
Plus précisément, on dispose des isomorphismes canoniques
\begin{gather}
\D ^{(m)} _{X ' _i}
 \otimes  _{\widehat{\D} ^{(m)} _{\X '}}
\left (a ^{*} \circ u _{+} (\smash{\overset{_\circ}{\FF }} ^{(m)}) \right )
\riso 
a _i ^{*} \circ u _{i+} (\smash{\overset{_\circ}{\FF _i}} ^{(m)}) 
\\
\label{3.2.3.4}
a ^{*} \circ u _{+} (\smash{\overset{_\circ}{\FF }} ^{(m)}) \riso 
\underset{\underset{i}{\longleftarrow}}{\lim\,} 
a _i ^{*} \circ u _{i+} (\smash{\overset{_\circ}{\FF _i}} ^{(m)}).
\end{gather}
}

\noindent Preuve: On bénéficie de l'isomorphisme canonique
\begin{equation}
\label{iso-3.2.3.3}
 \D ^{(m)} _{X ' _i}
 \otimes  ^{\L}_{\widehat{\D} ^{(m)} _{\X '}}
\left (\L a ^{*} \circ u _{+} (\smash{\overset{_\circ}{\FF }} ^{(m)}) \right )
\riso 
\L a _i ^{*} \left ( \D ^{(m)} _{X _i}
 \otimes ^{\L} _{\widehat{\D} ^{(m)} _{\X}}
 u _{+} (\smash{\overset{_\circ}{\FF }} ^{(m)})\right ).
\end{equation}
Comme $\smash{\overset{_\circ}{\FF }} ^{(m)}$ est cohérent et sans $p$-torsion,
il en est de même de $u _{+} (\smash{\overset{_\circ}{\FF }} ^{(m)})$. 
D'après \cite[3.5.1]{Beintro2},  
on dispose alors des isomorphismes canoniques
\begin{equation}
\label{iso-u+sansL}
\D ^{(m)} _{X _i}
 \otimes ^{\L} _{\widehat{\D} ^{(m)} _{\X}}
 u _{+} (\smash{\overset{_\circ}{\FF }} ^{(m)})
 \riso
 \D ^{(m)} _{X _i}
 \otimes  _{\widehat{\D} ^{(m)} _{\X}}
 u _{+} (\smash{\overset{_\circ}{\FF }} ^{(m)})
 \riso 
  u _{i+} (\smash{\overset{_\circ}{\FF _i}} ^{(m)}).
\end{equation}
Via le premier isomorphisme de \ref{A*E-sansL}, on en déduit les isomorphismes: 
\begin{equation}
\label{iso-3.2.3.3bis}
 \D ^{(m)} _{X ' _i}
 \otimes  ^{\L}_{\widehat{\D} ^{(m)} _{\X '}}
\left (\L a ^{*} \circ u _{+} (\smash{\overset{_\circ}{\FF }} ^{(m)}) \right )
\riso 
 \D ^{(m)} _{X ' _i}
 \otimes  _{\widehat{\D} ^{(m)} _{\X '}}
\left ( a ^{*} \circ u _{+} (\smash{\overset{_\circ}{\FF }} ^{(m)}) \right )
\riso
a _i ^{*} \circ u _{i+} (\smash{\overset{_\circ}{\FF _i}} ^{(m)}).
\end{equation}
D'après \cite[3.4.5]{Beintro2}, 
comme $u _{+} (\smash{\overset{_\circ}{\FF }} ^{(m)})$
est cohérent, 
$\L a ^{*} \circ u _{+} (\smash{\overset{_\circ}{\FF }} ^{(m)})$ est quasi-cohérent. 
 Il en dérive 
\begin{equation}
\label{MLu+RL}
\L a ^{*} \circ u _{+} (\smash{\overset{_\circ}{\FF }} ^{(m)})
\riso 
\R \underset{\underset{i}{\longleftarrow}}{\lim\,}  
\D ^{(m)} _{X ' _i}
 \otimes  ^{\L}_{\widehat{\D} ^{(m)} _{\X '}}
\left (\L a ^{*} \circ u _{+} (\smash{\overset{_\circ}{\FF }} ^{(m)}) \right ).
\end{equation}
Il dérive alors de \ref{iso-3.2.3.3bis} et \ref{MLu+RL} les isomorphismes du haut du diagramme ci-dessous:
\begin{equation}
\label{MLu+}
\xymatrix{
{\L a ^{*} \circ u _{+} (\smash{\overset{_\circ}{\FF }} ^{(m)})} 
\ar[r] ^-{\sim}
\ar[d] ^-{\sim}
& 
{\R \underset{\underset{i}{\longleftarrow}}{\lim\,} 
 \D ^{(m)} _{X ' _i}
 \otimes _{\widehat{\D} ^{(m)} _{\X '}}
\left (a ^{*} \circ u _{+} (\smash{\overset{_\circ}{\FF }} ^{(m)}) \right ) }
\ar[r] ^-{\sim}
& 
{\R \underset{\underset{i}{\longleftarrow}}{\lim\,} 
a _i ^{*} \circ u _{i+} (\smash{\overset{_\circ}{\FF _i}} ^{(m)})}
 \\ 
 {a ^{*} \circ u _{+} (\smash{\overset{_\circ}{\FF }} ^{(m)})} 
\ar[r] ^-{\sim}
 &
 {\underset{\underset{i}{\longleftarrow}}{\lim\,} 
 \D ^{(m)} _{X ' _i}
 \otimes _{\widehat{\D} ^{(m)} _{\X '}}
\left (a ^{*} \circ u _{+} (\smash{\overset{_\circ}{\FF }} ^{(m)}) \right ) } 
\ar[r] ^-{\sim}
\ar[u] ^-{\sim}
 &
 {\underset{\underset{i}{\longleftarrow}}{\lim\,} 
a _i ^{*} \circ u _{i+} (\smash{\overset{_\circ}{\FF _i}} ^{(m)}).} 
\ar[u] ^-{\sim}
  } 
\end{equation}
Comme la flèche horizontale du haut à droite est un isomorphisme entre un complexe à gauche sur la gauche et un complexe à droite sur la droite, 
il en résulte que 
les morphismes verticaux (et donc ceux du bas) sont bien des isomorphismes. 
D'où le résultat.

Remarquons au passage que le fait que 
le morphisme canonique
$\underset{\underset{i}{\longleftarrow}}{\lim\,}  u ^{\prime} _{i+}  \circ b _i ^{*} (\smash{\overset{_\circ}{\FF_i}} ^{(m)})
\to 
\R \underset{\underset{i}{\longleftarrow}}{\lim\,} 
u ^{\prime} _{i+}  \circ b _i ^{*} (\smash{\overset{_\circ}{\FF_i}} ^{(m)})$
soit un isomorphisme se voit directement.  
En effet, d'après \cite[2.2.3.2 et 2.4.2]{Beintro2}, 
la formation de $u ^{\prime} _{i+}  \circ b _i ^{*} (\smash{\overset{_\circ}{\FF_i}} ^{(m)})$ commute 
aux changements de base. 
Il en résulte que le système 
$(u ^{\prime} _{i+}  \circ b _i ^{*} (\smash{\overset{_\circ}{\FF_i}} ^{(m)})) _i$ est quasi-cohérent au sens de Berthelot.
Par \cite[7.20]{Berthelot-Ogus-cristalline}, on dispose alors d'une version faisceautique quasi-cohérente de Mittag-Leffler qui s'applique ici.
D'où le résultat.

\medskip
{\it Étape $3)$
On dispose des isomorphismes canoniques:
\begin{equation}
\label{iso-1''}
a ^* (\smash{\overset{_\circ}{E }} _{\mathrm{n\text{-}hol}} ^{(m)}) 
\riso
a ^{*} \circ u _{+} (\smash{\overset{_\circ}{F }} ^{(m)}) \riso 
\underset{\underset{i}{\longleftarrow}}{\lim\,} 
a _i ^{*} \circ u _{i+} (\smash{\overset{_\circ}{F _i}} ^{(m)})
\riso
\underset{\underset{i}{\longleftarrow}}{\lim\,} 
u ^{\prime} _{i+}  \circ b _i ^{*} (\smash{\overset{_\circ}{F_i}} ^{(m)})
\riso 
u ^{\prime(m)} _{+}  \circ b  ^{*} (\smash{\overset{_\circ}{F}} ^{(m)}).
\end{equation}
}

Preuve: Par \ref{u+-comm-Gamma}, il vient 
$\Gamma (\X, u _{+} (\smash{\overset{_\circ}{\FF }} ^{(m)}))
\riso 
u _{+} (\smash{\overset{_\circ}{F }} ^{(m)}) $.
Comme $\smash{\overset{_\circ}{\E }} _{\mathrm{n\text{-}hol}} ^{(m)} =u _{+} (\smash{\overset{_\circ}{\FF }} ^{(m)})$,
on a donc vérifié le premier isomorphisme de \ref{iso-1''}.
Vérifions à présent le second.
Grâce au premier isomorphisme de \ref{A*E-sansL} et à \ref{iso-u+sansL}, on remarque que l'on peut utiliser \ref{MLsansu+isocoro} pour 
$u _{+} (\smash{\overset{_\circ}{\FF }} ^{(m)})$.
On obtient donc le premier isomorphisme :
$\Gamma (\X', a ^* (u _{+} (\smash{\overset{_\circ}{\FF }} ^{(m)})))
\riso
a ^* (  \Gamma (\X, u _{+} (\smash{\overset{_\circ}{\FF }} ^{(m)})))
\riso
a ^{*} \circ u _{+} (\smash{\overset{_\circ}{F }} ^{(m)})$. 
Or, en utilisant les théorèmes de type $A$ pour les modules quasi-cohérents sur les schémas,
on vérifie l'isomorphisme canonique:
$\Gamma (X' _i, a _i ^{*} \circ u _{i+} (\smash{\overset{_\circ}{\FF _i}} ^{(m)}) )
\riso 
a _i ^{*} \circ u _{i+} (\smash{\overset{_\circ}{F _i}} ^{(m)})$.
Comme le foncteur $\Gamma (\X',-)$ commute aux limites projectives,
il en résulte 
$\Gamma (\X', \underset{\underset{i}{\longleftarrow}}{\lim\,} 
a _i ^{*} \circ u _{i+} (\smash{\overset{_\circ}{\FF _i}} ^{(m)}))
\riso
\underset{\underset{i}{\longleftarrow}}{\lim\,} 
a _i ^{*} \circ u _{i+} (\smash{\overset{_\circ}{F _i}} ^{(m)})$.
On a ainsi vérifié qu'en appliquant le foncteur 
$\Gamma (\X',-)$ à l'isomorphisme \ref{3.2.3.4} 
on obtient (à isomorphisme canonique près)
le deuxième isomorphisme de \ref{iso-1''}. 
De même, le troisième isomorphisme \ref{iso-1''} 
résulte du deuxième isomorphisme de \ref{A*E-sansL}.
Enfin, le dernier isomorphisme de \ref{iso-1''} résulte de l'étape I.5).
D'où le résultat.

\medskip
{\it Étape $4)$: construction de $G ^{(m)}$.} 

Via les théorèmes de type $A$ et $B$ pour les $\D _{\X,\Q} ^\dag$-modules cohérents, 
on dispose de la surjection
$E \twoheadrightarrow E _{\mathrm{n\text{-}hol}}$.
En lui appliquant le foncteur $a ^{*}$ exact à droite, 
on obtient la surjection 
$a ^{*} (E )\twoheadrightarrow a ^{*} (E _{\mathrm{n\text{-}hol}})$.
Comme par hypothèse $\E$ est surcohérent dans $\X$, grâce à \ref{a*Gamma}, on bénéficie alors de l'isomorphisme 
$\Gamma (\X', a ^{*} (\E  ) )
\riso
a ^{*} (E )$.
Comme $a ^{*} (\E )$ est un $\D ^\dag _{\X', \Q}$-module cohérent,
via le théorème de type $A$ pour les $\D ^\dag _{\X', \Q}$-modules cohérents (voir \cite[3.6.5]{Be1}), 
on en déduit que $a ^{*} (E )$ est un $D ^\dag _{\X', \Q}$-module de type fini. 
Il existe donc des éléments $ x _{1},\dots, x _{N} \in a ^{*} (E _{\mathrm{n\text{-}hol}})$ tels que 
$a ^{*} (E _{\mathrm{n\text{-}hol}})= \sum _{l=1} ^{N} D ^\dag _{\X', \Q}\cdot  x _{l}$.

Comme $a ^{*} (E _{\mathrm{n\text{-}hol}}) 
= 
\underset{\underset{m}{\longrightarrow}}{\lim\,} 
a ^{*} (E _{\mathrm{n\text{-}hol}} ^{(m)}) $, 
quitte à augmenter $m _0$,
il existe 
$x _{1} ^{(0)},\dots , x _{N} ^{(0)} \in a ^{*} (E ^{(m _0)} _\mathrm{n\text{-}hol})$
tels que, pour tout $l = 1,\dots, N$, 
$x _{l} ^{(0)} $ s'envoie sur $x _{l} $ via le morphisme canonique
$a ^{*} (E ^{(m _0)} _\mathrm{n\text{-}hol}) 
\to 
a ^{*} (E _{\mathrm{n\text{-}hol}})$.
Pour tout $m \geq m _0$, 
posons
$G ^{(m)} := \sum _{l=1} ^{N}  \widehat{D} ^{(m)} _{\X',\Q}\cdot  x _{l} ^{(m-m _0)}
\subset a ^{*} (E _{\mathrm{n\text{-}hol}} ^{(m)})$, 
où 
$x _{l} ^{(m-m _0)} $ désigne l'image de 
$x _{l} ^{(0)} $ via le morphisme canonique
$a ^{*} (E ^{(m _0)} _\mathrm{n\text{-}hol}) 
\to 
a ^{*} (E _{\mathrm{n\text{-}hol}} ^{(m)})$.

Terminons cette étape $4$ par une remarque. 
Avec les notations ci-dessus, pour tout $x ^{(0)} \in a ^{*} (E _{\mathrm{n\text{-}hol}} ^{(m _0)})$, il existe 
$m \geq m _0$ assez grand tel que l'image de $x^{(0)}$ via le morphisme 
$ a ^{*} (E _{\mathrm{n\text{-}hol}} ^{(m _0)})
\to 
 a ^{*} (E ^{(m )} _\mathrm{n\text{-}hol})$
 appartienne à 
 $G ^{(m)}$.

\medskip
{\it Étape $5)$ : construction de $\smash{\overset{_\circ}{G}} ^{(m)}$.}

Comme $G ^{(m)}$ est $\widehat{D} ^{(m)} _{\X',\Q}$-cohérent, il existe un $\widehat{D} ^{(m)} _{\X'}$-module cohérent sans $p$-torsion
$\smash{\overset{_\circ}{G}} ^{(m)}$
tel que 
$\smash{\overset{_\circ}{G}} ^{(m)} _\Q
\riso 
G ^{(m)}$.
Puisque l'on dispose de l'inclusion $\widehat{D} ^{(m)} _{\X',\Q}$-linéaire 
$G ^{(m)}\subset a ^{*} (E _{\mathrm{n\text{-}hol}} ^{(m)})$,
comme $\left (a ^{*} (\smash{\overset{_\circ}{E}}  ^{(m)}  _{\mathrm{n\text{-}hol}}) \right )_\Q
= a ^{*} (E _{\mathrm{n\text{-}hol}} ^{(m)})$,
quitte à multiplier par une puissance de $p$ convenable,
on peut en outre supposer
$\smash{\overset{_\circ}{G}} ^{(m)} 
\subset
a ^{*} (\smash{\overset{_\circ}{E}}  ^{(m)}  _{\mathrm{n\text{-}hol}})$.

\medskip
{\it Étape $6)$ : construction de $\smash{\overset{_\circ}{H}} ^{(m)}$ et $H^{(m)}$.}

D'après \ref{iso-1''}, on dispose de l'isomorphisme
$a ^* (\smash{\overset{_\circ}{E }} _{\mathrm{n\text{-}hol}} ^{(m)}) 
\riso 
u ^{\prime(m)} _{+}  \circ b  ^{*} (\smash{\overset{_\circ}{F}} ^{(m)})$
que l'on notera $\theta ^{(m)}$. On pose alors
$\smash{\overset{_\circ}{H}} ^{(m)}:= \theta ^{(m)} (\smash{\overset{_\circ}{G}} ^{(m)})$
et
$H^{(m)} :=( \smash{\overset{_\circ}{H}} ^{(m)} )_\Q$.

\medskip
{\it  Étape $7)$ : le $\V$-module
$\mathcal{H} ^{0} u ^{\prime !} (\smash{\overset{_\circ}{H}} ^{(m)})$ est libre de rang fini.}

Comme $\mathcal{H} ^{0} u ^{\prime !} (\smash{\overset{_\circ}{H}} ^{(m)})$ est sans $p$-torsion, 
séparé et complet pour la topologie $p$-adique (voir \ref{im-inv-D-mod}), celui-ci est le complété $p$-adique d'un $\V$-module libre. 
Il reste à vérifier qu'il est de type fini sur $\V$.
D'après l'étape $I.5)$, le module $ b ^{*} (\smash{\overset{_\circ}{F}} ^{(m)})$ est sans $p$-torsion, séparé et complet pour la topologie $p$-adique.
On dispose alors du diagramme canonique 
$$\xymatrix{
{\smash{\overset{_\circ}{H}} ^{(m)}}
\ar@{^{(}->}[r] ^-{} 
 &
 {u ^{\prime (m)} _{+}  \circ b ^{*} (\smash{\overset{_\circ}{F}} ^{(m)})}
 \ar@{=}[r]
  &
 {u ^{\prime (m)} _{+}  \circ b ^{*} (\smash{\overset{_\circ}{F}} ^{(m)})} 
 \\ 
{u ^{\prime (m)} _{+}  \mathcal{H} ^{0} u ^{\prime !} (\smash{\overset{_\circ}{H}} ^{(m)})} 
\ar[u] ^-{} 
\ar@{^{(}->}[r] ^-{}
&
{u ^{\prime (m)} _{+}  \mathcal{H} ^{0} u ^{\prime !} (u ^{\prime (m)} _{+}  \circ b ^{*} (\smash{\overset{_\circ}{F}} ^{(m)}))}
\ar[u] ^-{\sim} 
& 
{u ^{\prime (m)} _{+}  \circ b ^{*} (\smash{\overset{_\circ}{F}} ^{(m)}),} 
\ar[l] ^-{\sim}
\ar@{=}[u]
}$$
dont les morphismes horizontaux de gauche sont induits fonctoriellement par l'inclusion canonique
$\smash{\overset{_\circ}{H}} ^{(m)} 
\subset
u ^{\prime(m)} _{+}  \circ b  ^{*} (\smash{\overset{_\circ}{F}} ^{(m)})$, 
dont les flèches verticales de gauche et du centre sont induits par les morphismes d'adjonction de \ref{u_+u^!toif-limproj},
dont le morphisme horizontal en bas à droite se construit à partir du morphisme d'adjonction \ref{idtou+u!isohat}.
Le carré de gauche est commutatif par fonctorialité.
D'après \ref{comp-adj=id1}, le carré de droite est commutatif et ses flèches sont des isomorphismes. 
En particulier, la flèche verticale du centre est un isomorphisme.
Comme le foncteur $\mathcal{H} ^{0} u ^{\prime !}$ est exact à gauche,
on obtient l'inclusion canonique
$
\mathcal{H} ^{0} u ^{\prime !} (\smash{\overset{_\circ}{H}} ^{(m)} )
\subset 
\mathcal{H} ^{0} u ^{\prime !} (u ^{\prime(m)} _{+}  \circ b  ^{*} (\smash{\overset{_\circ}{F}} ^{(m)}))$.
Par \ref{inclusion-u+}, en lui appliquant 
le foncteur $u ^{\prime (m)} _{+}$, il en résulte que la flèche horizontale du bas à gauche est 
injective. 
On en déduit que la flèche canonique 
$u ^{\prime } _+ \mathcal{H} ^{0} u ^{\prime !} (\smash{\overset{_\circ}{H}} ^{(m)}) \to 
\smash{\overset{_\circ}{H}} ^{(m)}$ 
est injective. 
Comme $\smash{\overset{_\circ}{H}} ^{(m)}$ est $\widehat{D} ^{(m)} _{\X'}$-cohérent,
il en résulte par noethérianité de $\widehat{D} ^{(m)} _{\X'}$ que 
$u ^{\prime } _+ \mathcal{H} ^{0} u ^{\prime !} (\smash{\overset{_\circ}{H}} ^{(m)}) $ est $\widehat{D} ^{(m)} _{\X'}$-cohérent.
Il découle du lemme 
\ref{coro-inclusion-u+} que 
$\mathcal{H} ^{0} u ^{\prime !} (\smash{\overset{_\circ}{H}} ^{(m)}) $ est un $\V$-module libre de rang fini.

\medskip
{\it Étape $8)$ : Par l'absurde, on suppose que, pour tout $m \geq m _0$, 
$b ^{*} (\smash{\overset{_\circ}{F}} ^{(m)})$ n'est pas de type fini sur $\V$. On aboutit alors à une contradiction.}

i) Avec notre hypothèse par l'absurde, nous construisons par récurrence sur $s\in \N $ des éléments 
$y ^{(0)} _0, \dots, y ^{(0)}_s \in b ^{*} (\smash{\overset{_\circ}{F}} ^{(m _0)})$ 
ainsi que des entiers
$n _0, \dots, n _s$ de la manière suivante. 

Il résulte de \ref{idtou+u!isohat} que
le morphisme canonique
$ b  ^{*} (\smash{\overset{_\circ}{F}} ^{(m)}) \to  \mathcal{H} ^{0} u ^{\prime !} (u ^{\prime(m)} _{+}  (b  ^{*} (\smash{\overset{_\circ}{F}} ^{(m)})))$ est un isomorphisme.
Via cet isomorphisme, on identifiera $\mathcal{H} ^{0} u ^{\prime !} (\smash{\overset{_\circ}{H}} ^{(m)} )$
comme un sous-ensemble de $ b  ^{*} (\smash{\overset{_\circ}{F}} ^{(m)})$.
Comme  
$b ^{*} (\smash{\overset{_\circ}{F}} ^{(m _0)})$ est sans $p$-torsion, séparé et complet pour la topologie $p$-adique (voir l'étape $I.5)$),
comme on suppose que $b ^{*} (\smash{\overset{_\circ}{F}} ^{(m _0)})$ n'est pas un $\V$-module de type fini,
le $K$-espace vectoriel $b ^{*} (F ^{(m _0)}) =\left (b ^{*} (\smash{\overset{_\circ}{F}} ^{(m _0)}) \right ) _\Q$ 
n'est pas de dimension finie.
Comme 
$\mathcal{H} ^{0} u ^{\prime !} (\smash{\overset{_\circ}{H}} ^{(m)} )$ est un $\V$-module libre de type fini,
il existe donc un élément 
$y ^{(0)} _0 \in b ^{*} (\smash{\overset{_\circ}{F}} ^{(m _0)})$ tel que 
$y ^{(0)} _0 \not \in \mathcal{H} ^{0} u ^{\prime !} (H ^{(m _0)}) $.
Comme de plus le $K$-espace vectoriel $K y ^{(0)} _0 +\mathcal{H} ^{0} u ^{\prime !} (H ^{(m)})$ est de dimension finie,
d'après \ref{lemm2.2.7-courbe}, 
il en résulte qu'il existe un entier $n _0 \geq 0$ tel que, pour tout entier $n$,
  on ait l'inclusion :
  \begin{equation}
    \label{libre-inclusion1-application0}
  \left ( K y ^{(0)} _0 \oplus   \mathcal{H} ^{0} u ^{\prime !} (H ^{(m _0)}) \right )
  \cap 
  p ^{n _0 +n} b ^{*} (\smash{\overset{_\circ}{F }} ^{(m_0)})
  \subset 
  p ^n \left (\V y ^{(0)} _0 \oplus   \mathcal{H} ^{0} u ^{\prime !} (\smash{\overset{_\circ}{H}} ^{(m_0)}) \right ).
  \end{equation}

Supposons à présent construits $y ^{(0)} _0, \dots, y ^{(0)} _s \in b ^{*} (\smash{\overset{_\circ}{F}} ^{(m _0)})$ 
ainsi que les entiers $n _0, \dots, n _s$. 
Comme pour tout $m \geq m _0$ le  $K$-espace vectoriel $b ^{*} (F ^{(m )})$ 
n'est pas de dimension finie,
il résulte de \cite[2.2.8]{caro_courbe-nouveau} que  l'image de
$b ^{*} (F ^{(m _0)}) \to b ^{*} (F ^{(m)})$ est un $K$-espace vectoriel de dimension infinie. 
Notons $\M ^{(s+1)}:= \mathcal{H} ^{0} u ^{\prime !} (H ^{(m_0 +s +1)}) +\sum _{j=0} ^{s} K y _j  ^{(s+1)}$, 
où $y _j  ^{(s+1)}$ désigne l'image de $y ^{(0)} _j $ via le morphisme $b ^{*} (F ^{(m _0)}) \to b ^{*} (F ^{(m_0 +s+1)})$.
Soit $\smash{\overset{_\circ}{\M}} ^{(s +1)}$ un $\V$-module libre tel que 
$\smash{\overset{_\circ}{\M}} ^{(s +1)} \otimes \Q 
\riso 
\M ^{(s+1)}$.
Comme $\M ^{(s+1)}$ est de dimension finie sur $K$,
il existe donc 
$y _{s+1} ^{(0)}\in p^{s+1} b ^{*} (\smash{\overset{_\circ}{F}} ^{(m _0)})$ 
tel que 
$y _{s+1} ^{(s +1)}\not \in\M ^{(s+1)}$. 
Quitte à multiplier par une puissance de $p$, on peut en outre supposer que
pour tout $j = 0,\dots , s$, 
on ait 
$y _{s+1}^{(j)} \in p ^{1+n _j} b ^{*} (\smash{\overset{_\circ}{F}} ^{(m _0 +j)})$.
D'après \ref{lemm2.2.7-courbe}, il existe alors un entier $n _{s+1} \geq n _s +1$ tel que pour tout entier $n$ on ait
\begin{equation}
    \label{libre-inclusion1-applications+1}
  \left (  \M ^{(s+1)}\oplus K y _{s+1} ^{(s+1)} \right )
  \cap 
  p ^{n _{s+1} +n} b ^{*} (\smash{\overset{_\circ}{F }} ^{(m_0 +s+1)})
  \subset 
  p ^n   \left (   \smash{\overset{_\circ}{\M}} ^{(s+1)}\oplus \V y _{s+1} ^{(s+1)}  \right ).
  \end{equation}

ii) Posons alors $y ^{(0)} := \sum _{j=0} ^{\infty} y ^{(0)} _j \in b ^{*} (\smash{\overset{_\circ}{F}} ^{(m _0)})$.
Soit alors $x^{(0)} \in a ^{*} (\smash{\overset{_\circ}{E}} _\mathrm{n\text{-}hol} ^{(m _0)})$
l'élément tel que 
$\theta ^{(m _0)} (x ^{(0)})= 1 \otimes y ^{(0)}$.

D'après la remarque à la fin de l'étape 4), il existe 
$s \geq 1$ assez grand tel que l'image $x ^{(s)} $ de $x^{(0)}$ via le morphisme canonique
$ a ^{*} (E _\mathrm{n\text{-}hol} ^{(m _0)})
\to 
 a ^{*} (E ^{(m_0 +s)} _\mathrm{n\text{-}hol}) $
 appartienne à 
 $G ^{(m_0 +s)}$.
 Or, il découle par complétion $p$-adique du carré de droite de \ref{iso-chgt-de-basebis-chgtbase} (i.e. $S=T$)
 que l'on dispose du diagramme commutatif
 \begin{equation}
\label{iso-chgt-de-basebis-chgtbase-completion}
\xymatrix{
{a ^{*} (E _\mathrm{n\text{-}hol} ^{(m _0)})}
\ar[d] ^-{}
&
{a ^{*} (\smash{\overset{_\circ}{E}} _\mathrm{n\text{-}hol} ^{(m _0)})}
\ar@{^{(}->}[l] ^-{}
\ar@{=}[r] ^-{}
\ar[d] ^-{}
&
{a ^{*} \circ u ^{(m _0)} _+ (\smash{\overset{_\circ}{F}} ^{(m _0)}) } 
\ar[r] ^-{\sim} _-{\theta ^{(m_0)}}
\ar[d] ^-{}
& 
{u ^{\prime^{(m _0)}} _+  \circ b ^{*} (\smash{\overset{_\circ}{F}} ^{(m _0)}) } 
\ar[d] ^-{}
\\ 
{a ^{*} (E _\mathrm{n\text{-}hol} ^{(m _0+s)})}
&
\ar@{^{(}->}[l] ^-{}
{a ^{*} (\smash{\overset{_\circ}{E}} _\mathrm{n\text{-}hol} ^{(m _0+s )})}
\ar@{=}[r] ^-{}
&
{a ^{*} \circ u ^{(m _0+s)} _+ (\smash{\overset{_\circ}{F}} ^{(m _0+s)}) } 
\ar[r] ^-{\sim} _-{\theta ^{(m_0 +s)}}
& 
{u ^{\prime^{(m _0+s)}} _+  \circ b ^{*} (\smash{\overset{_\circ}{F}} ^{(m _0+s)}) .} 
}
\end{equation}
 L'image de $ 1 \otimes y ^{(0)}=\sum _{j=0} ^{\infty} 1 \otimes  y ^{(0)} _j $ via la flèche de droite de 
 \ref{iso-chgt-de-basebis-chgtbase-completion} 
 est
$1 \otimes  y ^{(s)}$, où $y ^{(s)}:= \sum _{j=0} ^{\infty}y _j ^{(s)}\in b ^{*} (\smash{\overset{_\circ}{F}} ^{(m _0+s)})$.
 On en déduit l'égalité
 $\theta ^{(m_0 +s)} (x ^{(s)} )=1 \otimes  y ^{(s)}$.
 Il en résulte que $ 1 \otimes y ^{(s)}
 \in H ^{(m_0 +s)}$. Comme  $ 1 \otimes y ^{(s)}$ est annulé par $t _1, \dots, t _r$, 
 on a en fait 
$ 1 \otimes y ^{(s)}\in \mathcal{H} ^{0} u ^{\prime !} (H ^{(m_0 +s)})$.
Modulo l'identification de  $\mathcal{H} ^{0} u ^{\prime !} (\smash{\overset{_\circ}{H}} ^{(m_0 +s)} )$
a un sous-ensemble de $ b  ^{*} (\smash{\overset{_\circ}{F}} ^{(m_0 +s)})$, 
on obtient donc
$y ^{(s)}  \in \mathcal{H} ^{0} u ^{\prime !} (H ^{(m_0 +s)})$.
 Ainsi, $y ^{(s)} - \sum _{j=0} ^{s} y _j ^{(s)}  \in 
 (\mathcal{H} ^{0} u ^{\prime !} (H ^{(m_0 +s)}) + \sum _{j=0} ^{s-1} K y _j  ^{(s)})+K y _s  ^{(s)}
 = \M ^{(s)} \oplus K y _s  ^{(s)}$.
 Or, pour tout $j \geq s+1$, 
 on a 
 $y _j ^{(s)} \in p ^{1+n _{s}} b ^{*} (\smash{\overset{_\circ}{F}} ^{(m _0 +s)})$.
 Comme 
 $y ^{(s)} - \sum _{j=0} ^{s} y _j ^{(s)} 
 = 
 \sum _{j=s+1} ^{\infty} y _j ^{(s)} $,
 on obtient
$y ^{(s)} - \sum _{j=0} ^{s} y _j ^{(s)}  \in 
\left (  \M ^{(s)}\oplus K y _{s} ^{(s)} \right )
\cap 
 p ^{1+n _{s}} b ^{*} (\smash{\overset{_\circ}{F}} ^{(m _0 +s)})
 \subset 
 p    \left (   \smash{\overset{_\circ}{\M}} ^{(s)}\oplus \V y _{s} ^{(s)}  \right )$.
 Comme $y ^{(s)}- \sum _{j=0} ^{s} y _j ^{(s)}  $ se décompose  dans 
 $ \M ^{(s)} \oplus K y _s  ^{(s)} $ 
 en 
 $(y ^{(s)} - \sum _{j=0} ^{s-1} y _j ^{(s)} ) - y _s ^{(s)}$, on a donc abouti à une contradiction.
 
 {\it Étape $9)$ : conclusion}
 
D'après l'étape $8)$, il existe un entier $m _1$ assez grand tel que 
$b ^{*} (F ^{(m _1)})$ soit un $K$-espace vectoriel de dimension finie.
Par \cite[2.2.8]{caro_courbe-nouveau}, il en résulte que
pour tout $m \geq m _1$, 
$b ^{*} (F ^{(m )})$ 
est un $K$-espace vectoriel de dimension finie.
En reprenant les étapes I.5) et I.6) de la preuve de \cite[3.4]{caro-holo-sansFrob} 
qui restent valables (l'hypothèse de finitude sur les fibres n'est utilisée qu'en amont de la preuve),
on établit qu'il existe un entier $m _2$ assez grand tel que 
$\FF |\ZZ _{(m _2)}$ soit un $\O _{\ZZ _{(m _2)}, \Q}$-module libre de type fini. 
 
\end{proof}

Je remercie le rapporteur pour le lemme ci-dessous et sa preuve. 
Il est vraisemblable que cela soit utile pour simplifier la fin de la preuve du lemme \ref{E1-hol-ouvdense} peu après le début de l'étape II.4, ce qui 
est laissé en exercice au lecteur.
\begin{lemm}
Soit $(M ^{(m)}) _{m\in \N}$ un système inductif d'espaces de Banach tel que 
$M:= \underrightarrow{\lim}\, _m M ^{(m)}$ soit un $K$-espace vectoriel de dimension finie. 
Alors, pour tout entier $m _0$, il existe $m _1 \geq m _0$ tel que, pour tout $m \geq m _1$,
la flèche canonique $\mathrm{Im} (M ^{(m _0)} \to M ^{(m)}) \to M$ soit injective.  
\end{lemm}

\begin{proof}
Soit $m _0 \in \N$.
Traitons d'abord le cas où $M=0$. 
Soit $K _{m}$ le noyau de la flèche $M ^{(m _0)} \to M ^{(m)}$.
Comme $M ^{(m)}$ est séparé, alors $K _{m}$ est fermé. Puisque $M=0$, on a alors
$\cup _{m \geq m _0} K _{m} = M ^{(m _0)} $. Comme $M ^{(m _0)}$ est un $K$-espace de Banach, 
il vérifie la propriété de Baire. On en déduit que pour $m _1$ assez grand 
$K _{m _1}= M ^{(m _0)}$, ce qui nous permet de conclure.
Considérons à présent le cas général.
Quitte à augmenter $m _0$, on peut supposer que 
$M ^{(m _0)} \to M $ soit surjective.
Choisissons $V ^{(m_0)}$ un $K$-sous-espace vectoriel de $M ^{(m _0)} $ de dimension finie
tel que $V  ^{(m_0)} \to M$ soit surjective. 
En notant $V ^{(m)}$ l'image de $V ^{(m _0)}$ dans $M ^{(m)}$ muni de la topologie induite par celle $V ^{(m)}$, 
il suffit alors d'appliquer le premier cas au système inductif 
$(M ^{(m)}/V ^{(m)}) _{m\geq m _0}$.
\end{proof}

\begin{theo}
\label{Pre-surcoh-hol}
Soit $\E$  un $\D ^\dag _{\X} (\hdag D) _{\Q}$-module surcohérent dans $\X$ et après tout changement de base. 
Alors $\E$ est $\D ^\dag _{\X} (\hdag D) _{\Q}$-holonome.
\end{theo}

\begin{proof}
Grâce à \cite[4.3.12]{Be1}, on se ramène au cas où le diviseur $D$ est vide. 
Posons $\E _{\mathrm{n\text{-}hol}} := \E / \E ^{\mathrm{hol}}$.
Notons $Z$ le support de $\E _{\mathrm{n\text{-}hol}}$.
Par l'absurde, supposons $Z$ non-vide et soit $Z'$ une composante irréductible. 
D'après \ref{E1-hol-ouvdense},
il existe alors un ouvert $\Y$ de $\X$ tel que 
l'ouvert $Y \cap Z'$ soit lisse et dense dans $Z'$ et
tel que le module $\E _{\mathrm{n\text{-}hol}} | \Y$ soit holonome.  
Or, $(\E _{\mathrm{n\text{-}hol}} | \Y) ^{\mathrm{hol}}= (\E _{\mathrm{n\text{-}hol}}) ^{\mathrm{hol}}|\Y=0$,
la dernière égalité provenant de \ref{E1*=0}.
De plus, comme le module $\E _{\mathrm{n\text{-}hol}} | \Y$ est holonome,
on obtient l'égalité 
$(\E _{\mathrm{n\text{-}hol}} | \Y) ^{\mathrm{hol}} = (\E _{\mathrm{n\text{-}hol}} | \Y)$ (voir \ref{Nota-hol}).
On a ainsi aboutit à une contradiction. 
\end{proof}

\subsection{Application: autour de la surholonomie}

On désignera par $\X$ un $\V$-schéma formel lisse. 
Le foncteur $\DD $ indique le foncteur dual $\D ^\dag _{\X,\Q}  $-linéaire comme défini par Virrion (voir \cite{virrion}).

\begin{defi}\label{defi-surhol}
  Soit $\E$ un objet de ($F$-)$D (\D ^\dag _{\X,\,\Q})$. On définit par récurrence sur
  l'entier $r\geq 0$, la notion
  de {\it $r$-surholonomie dans $\X$}, de la façon suivante :
  \begin{enumerate}
    \item Le complexe $\E$ est {\og $0$-surholonome dans $\X$\fg} si 
    $\E$ est surcohérent dans $\X$ ;
    \item Pour tout entier $r \geq 1$, $\E$ est {\og $r$-surholonome dans $\X$\fg } si $\E $ est $r-1$-surholonome dans $\X$
    et pour tout diviseur $T $ de $X$, le complexe
    $\DD \circ (\hdag T )  (\E)$ est $r-1$-surholonome dans $\X$.
  \end{enumerate}
  On dit que $\E$ est {\og surholonome dans $\X$\fg}
  si $\E$ est $r$-surholonome dans $\X$ pour tout entier $r$.
  Enfin un ($F$-)$\D ^\dag _{\X,\,\Q}$-module est $r$-surholonome (resp. surholonome) dans $\X$ s'il l'est en tant
  qu'objet de ($F$-)$D ^\mathrm{b} (\D ^\dag _{\X,\,\Q})$.
\end{defi}

\begin{coro}
\label{coro-stab-surcoh-dansX}
Soient $\E\in D ^\mathrm{b} (\D ^\dag _{\X,\,\Q})$ et $r\in \N$. 
 Alors, le complexe $\E$ est $r$-surholonome dans $\X$ (resp. surholonome dans $\X$) après tout changement de base (voir \ref{defi-surcoh} pour la signification de {\og après tout changement de base\fg}) si et seulement si pour tout entier $j \in \Z$ les modules
$\mathcal{H} ^{j} (\E) $ sont $r$-surholonome dans $\X$ (resp. surholonome dans $\X$) après tout changement de base.
\end{coro}

\begin{proof}
Il découle de \ref{Pre-surcoh-hol} que
les complexes $r$-surholonomes dans $\X$ sont holonomes. 
On vérifie alors le corollaire par récurrence sur $r$ en utilisant
le fait que le foncteur dual $\DD$ (resp. le foncteur extension $(\hdag T)$ avec $T$ un diviseur de $X$) 
est acyclique sur la catégorie des 
$\D ^\dag _{\X,\,\Q}$-modules holonomes (resp. cohérents et donc holonomes).

\end{proof}

\begin{vide}
D'après 
\cite[3.8]{Abe-Frob-Poincare-dual},
 pour tout morphisme lisse 
$f\,:\, \X \to \Y$ de $\V$-schémas formels lisses de dimension relative $d _f$,
on dispose d'après Abe de l'isomorphisme canonique 
\begin{equation}
\label{th3.8Abe-dual-f!}
f ^{!} \circ \DD \riso \DD \circ f ^{!} (d _f) [2d _f].
\end{equation}
\end{vide}

\begin{prop}
\label{surhol-surholdans}
Soient $\E\in D ^\mathrm{b} (\D ^\dag _{\X,\,\Q})$ et $r\in \N$.
  Alors $\E$ est surcohérent (resp. {$r$-surholonome}, resp. surholonome) si et seulement si,
  pour tout morphisme lisse $f$ : $\X' \rightarrow \X$ de $\V$-schémas formels,
  $f ^{!} (\E)$ est surcohérent (resp. {$r$-surholonome}, resp. surholonome) dans $\X'$.
\end{prop}

\begin{proof}
Le cas concernant la surcohérence résulte de la commutation de l'image inverse extraordinaire au foncteur de localisation 
(voir \cite[2.2.18.1]{caro_surcoherent}).
Le cas de la $r$-surholonome se traite par récurrence sur $r$. On procède de manière analogue en utilisant en plus 
l'isomorphisme \ref{th3.8Abe-dual-f!} vérifié par Abe. 
\end{proof}

\begin{coro}
\label{coro2-stab-surcoh-dansX}
Soient $\E\in D ^\mathrm{b} (\D ^\dag _{\X,\,\Q})$ et $r\in \N$. 
 Alors, le complexe $\E$ est $r$-surholonome  (resp. surholonome) après tout changement de base si et seulement si pour tout entier $j \in \Z$ les modules
$\mathcal{H} ^{j} (\E) $ sont $r$-surholonomes  (resp. surholonomes) après tout changement de base.
\end{coro}

\begin{proof}
Cela découle de \ref{coro-stab-surcoh-dansX}, de la caractérisation de \ref{surhol-surholdans}
et du fait que le foncteur $f ^*$ avec $f\,:\,\X' \to \X$ un morphisme lisse
est acyclique sur 
la catégorie des 
$\D ^\dag _{\X,\,\Q}$-modules cohérents.

\end{proof}

\begin{rema}
L'isomorphisme \ref{th3.8Abe-dual-f!} d'Abe est nécessaire afin de vérifier la caractérisation \ref{surhol-surholdans} de la surholonomie. 
Cependant, on aurait pu vérifier directement le corollaire \ref{coro2-stab-surcoh-dansX}
à partir de \ref{Pre-surcoh-hol} sans utiliser cet isomorphisme en reprenant les arguments d'acyclicité du foncteur dual, localisation et image inverse extraordinaire par un morphisme lisse.
\end{rema}

\bibliographystyle{smfalpha}
%\bibliography{Bibliotheque}

\def\cprime{$'$}
\providecommand{\bysame}{\leavevmode ---\ }
\providecommand{\og}{``}
\providecommand{\fg}{''}
\providecommand{\smfandname}{et}
\providecommand{\smfedsname}{\'eds.}
\providecommand{\smfedname}{\'ed.}
\providecommand{\smfmastersthesisname}{M\'emoire}
\providecommand{\smfphdthesisname}{Th\`ese}

\bigskip
\noindent Daniel Caro\\
Laboratoire de Mathématiques Nicolas Oresme\\
Université de Caen
Campus 2\\
14032 Caen Cedex\\
France.\\
email: daniel.caro@math.unicaen.fr

\end{document}